\documentclass[journal]{IEEEtran}

\usepackage{amsmath,amssymb,amsthm}
\usepackage{graphicx}
\usepackage{epstopdf}
\usepackage{subfigure}
\usepackage{caption}
\usepackage{setspace}
\usepackage{textcomp}
\usepackage{algorithmic}
\usepackage[lined,algoruled]{algorithm2e}
\usepackage{cases}
\usepackage{bbm}
\usepackage{upgreek}
\usepackage[square,comma,numbers,sort&compress]{natbib}
\usepackage{hyperref}
\usepackage[all]{hypcap}

\hypersetup{colorlinks=true,linkcolor=blue,citecolor=red}

\newcommand{\mb}[1]{\mathbf{#1}}
\newcommand{\msb}[1]{\boldsymbol{#1}}

\newcommand{\ts}{\textstyle}

\newcommand{\sscs}{\scriptscriptstyle}

\newcommand{\teq}{\triangleq}
\newcommand{\fx}[2]{#1\!\left(#2\right)}
\newcommand{\bfx}[2]{#1\big(#2\big)}

\newcommand{\varmin}[2]{\underset{#1}{\text{min}}\left\{#2\right\}}
\newcommand{\bvarmin}[2]{\underset{#1}{\text{min}}\,\big\{#2\!\big\}}

\newcommand{\nbvarmin}[2]{\underset{#1}{\text{min}}\,#2}
\newcommand{\argmin}[2]{\text{arg}\,\varmin{#1}{#2}}
\newcommand{\bargmin}[2]{\text{arg}\,\bvarmin{#1}{#2}}

\newcommand{\nbvarmax}[2]{\underset{#1}{\text{max}}\,#2}

\newcommand{\nbargminmax}[3]{\text{arg}\,\nbvarmin{#1}{}\nbvarmax{#2}{#3}}
\newcommand{\norm}[2]{\left\|#1\right\|_{#2}}
\newcommand{\bnorm}[2]{\big\|#1\big\|_{#2}}

\newcommand{\iter}[2]{#1^{\left(#2\right)}}

\newcommand{\LA}{\mathcal{L}_{\text{AL}}}
\renewcommand{\L}{\mathsf{L}}
\newcommand{\R}{\mathsf{R}}

\newcommand{\prox}[2]{\mathsf{prox}_{#1}\!\left(#2\right)}

\newcommand{\diag}[1]{\mathsf{diag}\!\left\{#1\right\}}

\newcommand{\nx}[1]{#1^{\sscs+}}

\makeatletter
\def\blfootnote{\xdef\@thefnmark{}\@footnotetext}
\makeatother

\usepackage{mathtools}
\newcommand{\alhs}[2]{\phantom{#2}\mathllap{#1}}

\usepackage{color}
\usepackage{soul}

\usepackage{amssymb,tikz}

\newtheorem{theorem}{Theorem}

\begin{document}

\title{Relaxed Linearized Algorithms for Faster X-Ray CT Image Reconstruction}
\author{Hung Nien, \textit{Member, IEEE}, and Jeffrey A. Fessler, \textit{Fellow, IEEE}\thanks{This work is supported in part by National Institutes of Health (NIH) grant \mbox{U01-EB-018753} and by equipment donations from Intel Corporation. Hung Nien and Jeffrey A. Fessler are with the Department of Electrical Engineering and Computer Science, University of Michigan, Ann Arbor, MI 48109, USA (e-mail: \texttt{\{hungnien,fessler\}@umich.edu}).}}

\maketitle

\begin{abstract}
Statistical image reconstruction (SIR) methods are studied extensively for X-ray computed tomography (CT) due to the potential of acquiring CT scans with reduced X-ray dose while maintaining image quality. However, the longer reconstruction time of SIR methods hinders their use in X-ray CT in practice. To accelerate statistical methods, many optimization techniques have been investigated. Over-relaxation is a common technique to speed up convergence of iterative algorithms. For instance, using a relaxation parameter that is close to two in alternating direction method of multipliers (ADMM) has been shown to speed up convergence significantly. This paper proposes a relaxed linearized augmented Lagrangian (AL) method that shows theoretical faster convergence rate with over-relaxation and applies the proposed relaxed linearized AL method to X-ray CT image reconstruction problems. Experimental results with both simulated and real CT scan data show that the proposed relaxed algorithm (with ordered-subsets [OS] acceleration) is about twice as fast as the existing unrelaxed fast algorithms, with negligible computation and memory overhead.
\end{abstract}

\begin{IEEEkeywords}
Statistical image reconstruction, computed tomography, ordered subsets, augmented Lagrangian, relaxation.
\end{IEEEkeywords}

\section{Introduction} \label{sec:nien-16-rla:intro}
\IEEEPARstart{S}{tatistical} image reconstruction (SIR) methods \cite{fessler:94:pwl,nuyts:13:mtp} have been studied extensively and used widely in medical imaging. In SIR methods, one models the physics of the imaging system, the statistics of noisy measurements, and the prior information of the object to be imaged, and then finds the best fitted estimate by minimizing a cost function using iterative algorithms. By considering noise statistics when reconstructing images, SIR methods have better bias-variance performance and noise robustness. However, the iterative nature of algorithms in SIR methods also increases the reconstruction time, hindering their ubiquitous use in X-ray CT in practice.

Penalized weighted least-squares (PWLS) cost functions with a statistically weighted quadratic data-fidelity term are commonly used in SIR methods for X-ray CT \cite{thibault:07:atd}. Conventional SIR methods include the preconditioned conjugate gradient (PCG) method \cite{fessler:99:cgp} and the separable quadratic surrogate (SQS) method with ordered-subsets (OS) acceleration \cite{erdogan:99:osa}. These first-order methods update the image based on the gradient of the cost function at the current estimate. Due to the time-consuming forward/back-projection operations in X-ray CT when computing gradients, conventional first-order methods are typically very slow. The efficiency of PCG relies on choosing an appropriate preconditioner of the highly shift-variant Hessian caused by the huge dynamic range of the statistical weighting. In \mbox{2-D} CT, one can introduce an auxiliary variable that separates the shift-variant and approximately shift-invariant components of the weighted quadratic data-fidelity term using a variable splitting technique \cite{ramani:12:asb}, leading to better conditioned inner least-squares problems. However, this method has not worked well in \mbox{3-D} CT, probably due to the 3-D cone-beam geometry and helical trajectory.

\mbox{OS-SQS} accelerates convergence using more frequent image updates by incremental gradients, i.e., computing image gradients with only a subset of data. This method usually exhibits fast convergence behavior in early iterations and becomes faster by using more subsets. However, it is not convergent in general \cite{ahn:03:gci,ahn:06:cio}. When more subsets are used, larger limit cycles can be observed. Unlike methods that update all voxels simultaneously, the iterative coordinate descent (ICD) method \cite{yu:11:fmb} updates one voxel at a time. Experimental results show that ICD approximately minimizes the PWLS cost function in several passes of the image volume if initialized appropriately; however, the sequential nature of ICD makes it difficult to parallelize and restrains the use of modern parallel computing architectures like GPU for speed-up.

\mbox{OS-mom} \cite{kim:15:cos} and \mbox{OS-LALM} \cite{nien:15:fxr} are two recently proposed iterative algorithms that demonstrate promising fast convergence speed when solving \mbox{3-D} X-ray CT image reconstruction problems. In short, \mbox{OS-mom} combines Nesterov's momentum techniques \cite{nesterov:83:amf,nesterov:05:smo} with the conventional \mbox{OS-SQS} algorithm, greatly accelerating convergence in early iterations. \mbox{OS-LALM}, on the other hand, is a linearized augmented Lagrangian (AL) method \cite{zhang:11:aup} that does not require inverting an enormous Hessian matrix involving the forward projection matrix when updating images, unlike typical splitting-based algorithms \cite{ramani:12:asb}, but still enjoys the empirical fast convergence speed and error tolerance of AL methods such as the alternating direction method of multipliers (ADMM) \cite{gabay:76:ada,eckstein:92:otd,boyd:10:doa}. Further acceleration from an algorithmic perspective is possible but seems to be more challenging. Kim et al. \cite{kim:14:oms,kim:16:ofo} proposed two optimal gradient methods (OGM's) that use a new momentum term and showed a $\sqrt{2}$-times speed-up for minimizing smooth convex functions, comparing to existing fast gradient methods (FGM's) \cite{nesterov:83:amf,nesterov:88:oaa,nesterov:05:smo,beck:09:afi}.

Over-relaxation is a common technique to speed up convergence of iterative algorithms. For example, it is very effective for accelerating ADMM \cite{eckstein:92:otd,boyd:10:doa}. The same relaxation technique was also applied to linearized ADMM very recently \cite{fang:15:gad}, but the speed-up was less significant than expected. Chambolle et al. proposed a relaxed primal-dual algorithm (whose unrelaxed variant happens to be a linearized ADMM \cite[Section~4.3]{chambolle:11:afo}) and showed the first theoretical justification for speeding up convergence with over-relaxation \cite[Theorem~2]{chambolle:16:ote}. However, their theorem also pointed out that when the smooth explicit term (majorization of the Lipschitz part in the cost function mentioned later) is not zero, one must use smaller primal step size to ensure convergence with over-relaxation, precluding the use of larger relaxation parameter (close to two) for more acceleration. This paper proposes a non-trivial relaxed variant of linearized AL methods that improves the convergence rate by using larger relaxation parameter values (close to two) but does not require the step-size adjustment in \cite{chambolle:16:ote}. We apply the proposed relaxed linearized algorithm to X-ray CT image reconstruction problems, and experimental results show that our proposed relaxation works much better than the simple relaxation \cite{fang:15:gad} and significantly accelerates X-ray CT image reconstruction, even with ordered-subsets (OS) acceleration.

This paper is organized as follows. Section~\ref{sec:nien-16-rla:relaxed_lalm} shows the convergence rate of a linearized AL method (LALM) with simple relaxation and proposes a novel relaxed LALM whose convergence rate scales better with the relaxation parameter. Section~\ref{sec:nien-16-rla:method} applies the proposed relaxed LALM to X-ray CT image reconstruction and uses a second-order recursive system analysis to derive a continuation sequence that speeds up the proposed algorithm. Section~\ref{sec:nien-16-rla:result} reports the experimental results of X-ray CT image reconstruction using the proposed algorithm. Finally, we draw conclusions in Section~\ref{sec:nien-16-rla:discussion_conclusion}. Online supplementary material contains many additional results and derivation details.

\section{Relaxed linearized AL methods} \label{sec:nien-16-rla:relaxed_lalm}
We begin by discussing a more general constrained minimization problem for which X-ray CT image reconstruction is a special case considered in Section~\ref{sec:nien-16-rla:method}. Consider an equality-constrained minimization problem:
\begin{equation} \label{eq:nien-16-rla:eq_constrained_prob}
    \left(\hat{\mb{x}},\hat{\mb{u}}\right)
    \in
    \bargmin{\mb{x},\mb{u}}{\fx{g_{\mb{y}}}{\mb{u}}+\fx{h}{\mb{x}}}
    \text{ s.t. }
    \mb{u} = \mb{Ax} \, ,
\end{equation}
where $g_{\mb{y}}$ and $h$ are closed and proper convex functions. In particular, $g_{\mb{y}}$ is a loss function that measures the discrepancy between the linear model $\mb{Ax}$ and noisy measurement $\mb{y}$, and $h$ is a regularization term that introduces prior knowledge of $\mb{x}$ to the reconstruction. We assume that the regularizer $h\teq\phi+\psi$ is the sum of two convex components $\phi$ and $\psi$, where $\phi$ has inexpensive proximal mapping (prox-operator) defined as
\begin{equation} \label{eq:nien-16-rla:def_prox}
    \prox{\phi}{\mb{x}}
    \teq
    \argmin{\mb{z}}{\fx{\phi}{\mb{z}}+\tfrac{1}{2}\norm{\mb{z}-\mb{x}}{2}^2} \, ,
\end{equation}
e.g., soft-shrinkage for the $\ell_1$-norm and truncating zeros for non-negativity constraints, and where $\psi$ is continuously differentiable with $L_{\psi}$-Lipschitz gradients \cite[p. 48]{bertsekas:99:np}, i.e.,
\begin{equation} \label{eq:nien-16-rla:def_lipschitz}
    \norm{\fx{\nabla\psi}{\mb{x}_1}-\fx{\nabla\psi}{\mb{x}_2}}{2}\leq L_{\psi}\norm{\mb{x}_1-\mb{x}_2}{2}
\end{equation}
for any $\mb{x}_1$ and $\mb{x}_2$ in the domain of $\psi$. The Lipschitz condition of $\nabla\psi$ implies the ``(quadratic) majorization condition'' of $\psi$:
\begin{equation} \label{eq:nien-16-rla:def_maj_lipschitz}
    \fx{\psi}{\mb{x}_2}
    \leq
    \fx{\psi}{\mb{x}_1}+\langle\fx{\nabla\psi}{\mb{x}_1},\mb{x}_2-\mb{x}_1\rangle+\tfrac{L_{\psi}}{2}\norm{\mb{x}_2-\mb{x}_1}{2}^2 \, .
\end{equation}
More generally, one can replace the Lipschitz constant $L_{\psi}$ by a diagonal majorizing matrix $\mb{D}_{\psi}$ based on the maximum curvature \cite{erdogan:99:maf} or Huber's optimal curvature \cite[p. 184]{huber:81:rs} of $\psi$ while still guaranteeing the majorization condition:
\begin{equation} \label{eq:nien-16-rla:def_maj_diagonal}
    \fx{\psi}{\mb{x}_2}
    \leq
    \fx{\psi}{\mb{x}_1}+\langle\fx{\nabla\psi}{\mb{x}_1},\mb{x}_2-\mb{x}_1\rangle+\tfrac{1}{2}\norm{\mb{x}_2-\mb{x}_1}{\mb{D}_{\psi}}^2 \, .
\end{equation}

We show later that decomposing $h$ into the proximal part $\phi$ and the Lipschitz part $\psi$ is useful when solving minimization problems with composite regularization. For example, Section~\ref{sec:nien-16-rla:method} writes iterative X-ray CT image reconstruction as a special case of \eqref{eq:nien-16-rla:eq_constrained_prob}, where $g_{\mb{y}}$ is a weighted quadratic function, and $h$ is an edge-preserving regularizer with a non-negativity constraint on the reconstructed image.

\subsection{Preliminaries} \label{subsec:nien-16-rla:preliminaries}
Solving the equality-constrained minimization problem \eqref{eq:nien-16-rla:eq_constrained_prob} is equivalent to finding a saddle-point of the Lagrangian:
\begin{equation} \label{eq:nien-16-rla:def_lagrangian}
    \fx{\mathcal{L}}{\mb{x},\mb{u},\msb{\mu}}
    \teq
    \fx{f}{\mb{x},\mb{u}}-\langle\msb{\mu},\mb{Ax}-\mb{u}\rangle \, ,
\end{equation}
where $\fx{f}{\mb{x},\mb{u}}\teq\fx{g_{\mb{y}}}{\mb{u}}+\fx{h}{\mb{x}}$, and $\msb{\mu}$ is the Lagrange multiplier of the equality constraint \cite[p. 237]{boyd:04:co}. In other words, $\left(\hat{\mb{x}},\hat{\mb{u}},\hat{\msb{\mu}}\right)$ solves the minimax problem:
\begin{equation} \label{eq:nien-16-rla:def_lag_minimax}
    \left(\hat{\mb{x}},\hat{\mb{u}},\hat{\msb{\mu}}\right)
    \in
    \nbargminmax{\mb{x},\mb{u}}{\msb{\mu}}{\fx{\mathcal{L}}{\mb{x},\mb{u},\msb{\mu}}} \, .
\end{equation}
Moreover, since $\left(\hat{\mb{x}},\hat{\mb{u}},\hat{\msb{\mu}}\right)$ is a saddle-point of $\mathcal{L}$, the following inequalities hold for any $\mb{x}$, $\mb{u}$, and $\msb{\mu}$:
\begin{equation} \label{eq:nien-16-rla:saddle_pt_cond}
    \fx{\mathcal{L}}{\mb{x},\mb{u},\hat{\msb{\mu}}}
    \geq
    \fx{\mathcal{L}}{\hat{\mb{x}},\hat{\mb{u}},\hat{\msb{\mu}}}
    \geq
    \fx{\mathcal{L}}{\hat{\mb{x}},\hat{\mb{u}},\msb{\mu}} \, .
\end{equation}
The non-negative duality gap function:
\begin{multline} \label{eq:nien-16-rla:def_dual_gap}
    \fx{\mathcal{G}}{\mb{x},\mb{u},\msb{\mu};\hat{\mb{x}},\hat{\mb{u}},\hat{\msb{\mu}}}
    \teq
    \fx{\mathcal{L}}{\mb{x},\mb{u},\hat{\msb{\mu}}}-\fx{\mathcal{L}}{\hat{\mb{x}},\hat{\mb{u}},\msb{\mu}} \\
    =
    \big[\fx{f}{\mb{x},\mb{u}}-\fx{f}{\hat{\mb{x}},\hat{\mb{u}}}\big]-\langle\hat{\msb{\mu}},\mb{Ax}-\mb{u}\rangle
\end{multline}
characterizes the accuracy of an approximate solution $\left(\mb{x},\mb{u},\msb{\mu}\right)$ to the saddle-point problem \eqref{eq:nien-16-rla:def_lag_minimax}. Note that $\hat{\mb{u}}=\mb{A}\hat{\mb{x}}$ due to the equality constraint. Besides solving the classic Lagrangian minimax problem \eqref{eq:nien-16-rla:def_lag_minimax}, $\left(\hat{\mb{x}},\hat{\mb{u}},\hat{\msb{\mu}}\right)$ also solves a family of minimax problems:
\begin{equation} \label{eq:nien-16-rla:def_aug_lag_minimax}
    \left(\hat{\mb{x}},\hat{\mb{u}},\hat{\msb{\mu}}\right)
    \in
    \nbargminmax{\mb{x},\mb{u}}{\msb{\mu}}{\fx{\LA}{\mb{x},\mb{u},\msb{\mu}}} \, ,
\end{equation}
where the augmented Lagrangian (AL) \cite[p. 297]{bertsekas:99:np} is
\begin{equation} \label{eq:nien-16-rla:def_aug_lag}
    \fx{\LA}{\mb{x},\mb{u},\msb{\mu}}\teq\fx{\mathcal{L}}{\mb{x},\mb{u},\msb{\mu}}+\tfrac{\rho}{2}\norm{\mb{Ax}-\mb{u}}{2}^2 \, .
\end{equation}
The augmented quadratic penalty term penalizes the feasibility violation of the equality constraint, and the AL penalty parameter $\rho>0$ controls the curvature of $\LA$ but does not change the solution, sometimes leading to better conditioned minimax problems.

One popular iterative algorithm for solving equality-constrained minimization problems based on the AL theory is ADMM, which solves the AL minimax problem \eqref{eq:nien-16-rla:def_aug_lag_minimax}, and thus the equality-constrained minimization problem \eqref{eq:nien-16-rla:eq_constrained_prob}, in an alternating direction manner. More precisely, ADMM minimizes AL \eqref{eq:nien-16-rla:def_aug_lag} with respect to $\mb{x}$ and $\mb{u}$ alternatingly, followed by a gradient ascent of $\msb{\mu}$ with step size $\rho$. One can also interpolate or extrapolate variables in subproblems, leading to a relaxed AL method \cite[Theorem~8]{eckstein:92:otd}:
\begin{equation} \label{eq:nien-16-rla:relaxed_al_iter}
    \begin{cases}
    \iter{\mb{x}}{k+1}\in\argmin{\mb{x}}{\fx{h}{\mb{x}}-\langle\iter{\msb{\mu}}{k},\mb{Ax}\rangle+\tfrac{\rho}{2}\bnorm{\mb{Ax}-\iter{\mb{u}}{k}}{2}^2} \\
    \iter{\mb{u}}{k+1}\in\argmin{\mb{u}}{\fx{g_{\mb{y}}}{\mb{u}}+\langle\iter{\msb{\mu}}{k},\mb{u}\rangle+\tfrac{\rho}{2}\bnorm{\iter{\mb{r}_{\mb{u},\alpha}}{k+1}-\mb{u}}{2}^2} \\
    \iter{\msb{\mu}}{k+1}=\iter{\msb{\mu}}{k}-\rho\big(\iter{\mb{r}_{\mb{u},\alpha}}{k+1}-\iter{\mb{u}}{k+1}\big) \, ,
    \end{cases}
\end{equation}
where the relaxation variable of $\mb{u}$ is:
\begin{equation} \label{eq:nien-16-rla:relax_sv_u}
    \iter{\mb{r}_{\mb{u},\alpha}}{k+1}\teq\alpha\mb{A}\iter{\mb{x}}{k+1}+\left(1-\alpha\right)\iter{\mb{u}}{k} \, ,
\end{equation}
and $0<\alpha<2$ is the relaxation parameter. It is called over-relaxation when $\alpha>1$ and  under-relaxation when $\alpha<1$. When $\alpha$ is unity, \eqref{eq:nien-16-rla:relaxed_al_iter} reverts to the standard (alternating direction) AL method \cite{gabay:76:ada}. Experimental results suggest that over-relaxation with $\alpha\in\left[1.5,1.8\right]$ can accelerate convergence \cite{boyd:10:doa}.

Although \eqref{eq:nien-16-rla:relaxed_al_iter} is used widely in applications, two concerns about the relaxed AL method \eqref{eq:nien-16-rla:relaxed_al_iter} arise in practice. First, the cost function of the $\mb{x}$-subproblem in \eqref{eq:nien-16-rla:relaxed_al_iter} contains the augmented quadratic penalty of AL that involves $\mb{A}$, deeply coupling elements of $\mb{x}$ and often leading to an expensive iterative $\mb{x}$-update, especially when $\mb{A}$ is large and unstructured, e.g., in X-ray CT. This motivates alternative methods like LALM \cite{zhang:11:aup,nien:15:fxr}. Second, even though LALM removes the $\mb{x}$-coupling due to the augmented quadratic penalty, the regularization term $h$ might not have inexpensive proximal mapping and still require an iterative $\mb{x}$-upate (albeit without using $\mb{A}$). This consideration inspires the decomposition $h\teq\phi+\psi$ used in the algorithms discussed next.

\subsection{Linearized AL methods with simple relaxation} \label{subsec:nien-16-rla:simple_relax}
In LALM\footnote{Because \eqref{eq:nien-16-rla:aug_quad_penalty_plus_prox_term} is quadratic, not linear, a more apt term would be ``majorized'' rather than ``linearized.'' We stick with the term linearized for consistency with the literature on LALM.}, one adds an iteration-dependent proximity term:
\begin{equation} \label{eq:nien-16-rla:la_prox_term}
    \tfrac{1}{2}\bnorm{\mb{x}-\iter{\mb{x}}{k}}{\mb{P}}^2
\end{equation}
to the $\mb{x}$-update in \eqref{eq:nien-16-rla:relaxed_al_iter} with $\alpha=1$, where $\mb{P}$ is a positive semi-definite matrix. Choosing $\mb{P}=\rho\mb{G}$, where \mbox{$\mb{G}\teq L_{\mb{A}}\mb{I}-\mb{A}'\mb{A}$}, and $L_{\mb{A}}$ denotes the maximum eigenvalue of $\mb{A}'\mb{A}$, the non-separable Hessian of the augmented quadratic penalty of AL is cancelled, and the Hessian of
\begin{equation} \label{eq:nien-16-rla:aug_quad_penalty_plus_prox_term}
    \tfrac{\rho}{2}\bnorm{\mb{Ax}-\iter{\mb{u}}{k}}{2}^2+\tfrac{\rho}{2}\bnorm{\mb{x}-\iter{\mb{x}}{k}}{\mb{G}}^2
\end{equation}
becomes a diagonal matrix $\rho L_{\mb{A}}\mb{I}$, decoupling $\mb{x}$ in the $\mb{x}$-update except for the effect of $h$. This technique is known as linearization (more precisely, majorization) because it majorizes a non-separable quadratic term by its linear component plus some separable qradratic proximity term. In general, one can also use
\begin{equation} \label{eq:nien-16-rla:aa_majorization}
    \mb{G}\teq\mb{D}_{\mb{A}}-\mb{A}'\mb{A} \, ,
\end{equation}
where $\mb{D}_{\mb{A}}\succeq\mb{A}'\mb{A}$ is a diagonal majorizing matrix of $\mb{A}'\mb{A}$, e.g., $\mb{D}_{\mb{A}}=\diag{|\mb{A}|'|\mb{A}|\mb{1}}\succeq\mb{A}'\mb{A}$ \cite{erdogan:99:osa}, and still guarantee the positive semi-definiteness of $\mb{P}$. This trick can be applied to \eqref{eq:nien-16-rla:relaxed_al_iter} when $\alpha\neq1$, too.

To remove the possible coupling due to the regularization term $h$, we replace the Lipschitz part of $h\teq\phi+\psi$ in the $\mb{x}$-update of \eqref{eq:nien-16-rla:relaxed_al_iter} with its separable quadratic surrogate (SQS):
\begin{multline} \label{eq:nien-16-rla:psi_sqs}
    \bfx{Q_{\psi}}{\mb{x};\iter{\mb{x}}{k}}\teq
    \bfx{\psi}{\iter{\mb{x}}{k}}\\+\langle\bfx{\nabla\psi}{\iter{\mb{x}}{k}},\mb{x}-\iter{\mb{x}}{k}\rangle+\tfrac{1}{2}\bnorm{\mb{x}-\iter{\mb{x}}{k}}{\mb{D}_{\psi}}^2
\end{multline}
shown in \eqref{eq:nien-16-rla:def_maj_lipschitz} and \eqref{eq:nien-16-rla:def_maj_diagonal}. Note that \eqref{eq:nien-16-rla:def_maj_lipschitz} is just a special case of \eqref{eq:nien-16-rla:def_maj_diagonal} when $\mb{D}_{\psi}=L_{\psi}\mb{I}$. Incorporating all techniques mentioned above, the $\mb{x}$-update becomes simply a proximal mapping of $\phi$, which by assumption is inexpensive. The resulting ``LALM with simple relaxation'' algorithm is:
\begin{equation} \label{eq:nien-16-rla:simple_relaxed_lal_iter}
    \begin{cases}
    \iter{\mb{x}}{k+1}\in\argmin{\mb{x}}{
    \begin{aligned}
    &\fx{\phi}{\mb{x}}+\bfx{Q_{\psi}}{\mb{x};\iter{\mb{x}}{k}}
    -\langle\iter{\msb{\mu}}{k},\mb{Ax}\rangle\\
    &+\tfrac{\rho}{2}\bnorm{\mb{Ax}-\iter{\mb{u}}{k}}{2}^2+\tfrac{\rho}{2}\bnorm{\mb{x}-\iter{\mb{x}}{k}}{\mb{G}}^2
    \end{aligned}
    } \\
    \iter{\mb{u}}{k+1}\in\argmin{\mb{u}}{\fx{g_{\mb{y}}}{\mb{u}}+\langle\iter{\msb{\mu}}{k},\mb{u}\rangle+\tfrac{\rho}{2}\bnorm{\iter{\mb{r}_{\mb{u},\alpha}}{k+1}-\mb{u}}{2}^2} \\
    \iter{\msb{\mu}}{k+1}=\iter{\msb{\mu}}{k}-\rho\big(\iter{\mb{r}_{\mb{u},\alpha}}{k+1}-\iter{\mb{u}}{k+1}\big) \, .
    \end{cases}
\end{equation}

When $\psi=0$, \eqref{eq:nien-16-rla:simple_relaxed_lal_iter} reverts to the \mbox{L-GADMM} algorithm proposed in \cite{fang:15:gad}. In \cite{fang:15:gad}, the authors analyzed the convergence rate of \mbox{L-GADMM} (for solving an equivalent variational inequality problem; however, there is no analysis on how relaxation parameter $\alpha$ affects the convergence rate) and investigated solving problems in statistical learning using \mbox{L-GADMM}. The speed-up resulting from over-relaxation was less significant than expected (e.g., when solving an X-ray CT image reconstruction problem discussed later). To explain the small speed-up, the following theorem shows that the duality gap \eqref{eq:nien-16-rla:def_dual_gap} of the time-averaged approximate solution $\mb{w}_K=\left(\mb{x}_K,\mb{u}_K,\msb{\mu}_K\right)$ generated by \eqref{eq:nien-16-rla:simple_relaxed_lal_iter} vanishes at rate $\fx{\mathcal{O}}{1/K}$, where $K$ is the number of iterations, and
\begin{equation} \label{eq:nien-16-rla:def_time_avg_vec}
    \mb{c}_K
    \teq
    \tfrac{1}{K}\ts\sum_{k=1}^{K}\iter{\mb{c}}{k}
\end{equation}
denotes the time-average of some iterate $\iter{\mb{c}}{k}$ for $k=1$ to $K$.

\begin{theorem} \label{thm:nien-16-rla:simple_relax}
Let $\mb{w}_K=\left(\mb{x}_K,\mb{u}_K,\msb{\mu}_K\right)$ be the time-averages of the iterates of LALM with simple relaxation in \eqref{eq:nien-16-rla:simple_relaxed_lal_iter}, where $\rho>0$ and  $0<\alpha<2$. We have
\begin{equation} \label{eq:nien-16-rla:conv_rate_simple_relax}
    \fx{\mathcal{G}}{\mb{w}_K;\hat{\mb{w}}}
    \leq
    \tfrac{1}{K}\left(A_{\mb{D}_{\psi}}+B_{\rho,\mb{D}_{\mb{A}}}+C_{\alpha,\rho}\right) \, ,
\end{equation}
where the first two constants
\begin{align}
    A_{\mb{D}_{\psi}}
    &\teq
    \tfrac{1}{2}\bnorm{\iter{\mb{x}}{0}-\hat{\mb{x}}}{\mb{D}_{\psi}}^2 \label{eq:nien-16-rla:def_A} \\
    B_{\rho,\mb{D}_{\mb{A}}}
    &\teq
    \tfrac{\rho}{2}\bnorm{\iter{\mb{x}}{0}-\hat{\mb{x}}}{\mb{D}_{\mb{A}}-\mb{A}'\mb{A}}^2 \label{eq:nien-16-rla:def_B}
\end{align}
depend on how far the initial guess is from a minimizer, and the last constant depends on the relaxation parameter
\begin{equation} \label{eq:nien-16-rla:def_C}
    C_{\alpha,\rho}
    \teq
    \tfrac{1}{2\alpha}\left[\sqrt{\rho}\bnorm{\iter{\mb{u}}{0}-\hat{\mb{u}}}{2}+\tfrac{1}{\sqrt{\rho}}\bnorm{\iter{\msb{\mu}}{0}-\hat{\msb{\mu}}}{2}\right]^2 \, .
\end{equation}
\end{theorem}
\begin{proof}
The proof is in the supplementary material.
\end{proof}

Theorem~\ref{thm:nien-16-rla:simple_relax} shows that \eqref{eq:nien-16-rla:simple_relaxed_lal_iter} converges at rate $\fx{\mathcal{O}}{1/K}$, and the constant multiplying $1/K$ consists of three terms: $A_{\mb{D}_{\psi}}$, $B_{\rho,\mb{D}_{\mb{A}}}$, and $C_{\alpha,\rho}$. The first term $A_{\mb{D}_{\psi}}$ comes from the majorization of $\psi$, and it is large when $\psi$ has large curvature. The second term $B_{\rho,\mb{D}_{\mb{A}}}$ comes from the linearization trick in \eqref{eq:nien-16-rla:aug_quad_penalty_plus_prox_term}. One can always decrease its value by decreasing $\rho$. The third term $C_{\alpha,\rho}$ is the only $\alpha$-dependent component. The trend of $C_{\alpha,\rho}$ when varying $\rho$ depends on the norms of $\iter{\mb{u}}{0}-\hat{\mb{u}}$ and $\iter{\msb{\mu}}{0}-\hat{\msb{\mu}}$, i.e., how one initializes the algorithm. Finally, the convergence rate of \eqref{eq:nien-16-rla:simple_relaxed_lal_iter} scales well with $\alpha$ iff $C_{\alpha,\rho}\gg A_{\mb{D}_{\psi}}$ and $C_{\alpha,\rho}\gg B_{\rho,\mb{D}_{\mb{A}}}$. When $\psi$ has large curvature or $\mb{D}_{\mb{A}}$ is a loose majorizing matrix of $\mb{A}'\mb{A}$ (like in X-ray CT), the above inequalities do not hold, leading to poor scalability of convergence rate with the relaxation parameter $\alpha$.

\subsection{Linearized AL methods with proposed relaxation} \label{subsec:nien-16-rla:proposed_relax}
To better scale the convergence rate of relaxed LALM with $\alpha$, we want to design an algorithm that replaces the $\alpha$-independent components by $\alpha$-dependent ones in the constant multiplying $1/K$ in \eqref{eq:nien-16-rla:conv_rate_simple_relax}. This can be (partially) done by linearizing (more precisely, majorizing) the non-separable AL penalty term in \eqref{eq:nien-16-rla:relaxed_al_iter} implicitly. Instead of explicitly adding a $\mb{G}$-weighted proximity term, where $\mb{G}$ is defined in \eqref{eq:nien-16-rla:aa_majorization}, to the $\mb{x}$-update like \eqref{eq:nien-16-rla:simple_relaxed_lal_iter}, we consider solving an equality-constrained minimization problem equivalent to \eqref{eq:nien-16-rla:eq_constrained_prob} with an additional redundant equality constraint $\mb{v}=\mb{G}^{1/2}\mb{x}$, i.e.,
\begin{multline} \label{eq:nien-16-rla:redundant_eq_constrained_prob}
    \left(\hat{\mb{x}},\hat{\mb{u}},\hat{\mb{v}}\right)
    \in
    \bargmin{\mb{x},\mb{u},\mb{v}}{\fx{g_{\mb{y}}}{\mb{u}}+\fx{h}{\mb{x}}} \\
    \text{ s.t. }
    \mb{u} = \mb{Ax}\text{ and }\mb{v}=\mb{G}^{1/2}\mb{x} \, ,
\end{multline}
using the relaxed AL method \eqref{eq:nien-16-rla:relaxed_al_iter} as follows:
\begin{equation} \label{eq:nien-16-rla:relaxed_lal_iter0}
    \begin{cases}
    \iter{\mb{x}}{k+1}\in\argmin{\mb{x}}{
    \begin{aligned}
    &\fx{\phi}{\mb{x}}+\bfx{Q_{\psi}}{\mb{x};\iter{\mb{x}}{k}}-\langle\iter{\msb{\mu}}{k},\mb{Ax}\rangle \\
    &
    \begin{aligned}
    -\langle\iter{\msb{\nu}}{k},\mb{G}^{1/2}\mb{x}\rangle+\tfrac{\rho}{2}\bnorm{\mb{Ax}-\iter{\mb{u}}{k}}{2}^2 \\
    +\tfrac{\rho}{2}\bnorm{\mb{G}^{1/2}\mb{x}-\iter{\mb{v}}{k}}{2}^2
    \end{aligned}
    \end{aligned}
    } \\
    \iter{\mb{u}}{k+1}\in\argmin{\mb{u}}{\fx{g_{\mb{y}}}{\mb{u}}+\langle\iter{\msb{\mu}}{k},\mb{u}\rangle+\tfrac{\rho}{2}\bnorm{\iter{\mb{r}_{\mb{u},\alpha}}{k+1}-\mb{u}}{2}^2} \\
    \iter{\msb{\mu}}{k+1}=\iter{\msb{\mu}}{k}-\rho\big(\iter{\mb{r}_{\mb{u},\alpha}}{k+1}-\iter{\mb{u}}{k+1}\big) \\
    \iter{\mb{v}}{k+1}=\iter{\mb{r}_{\mb{v},\alpha}}{k+1}-\rho^{-1}\iter{\msb{\nu}}{k} \\
    \iter{\msb{\nu}}{k+1}=\iter{\msb{\nu}}{k}-\rho\big(\iter{\mb{r}_{\mb{v},\alpha}}{k+1}-\iter{\mb{v}}{k+1}\big) \, ,
    \end{cases}
\end{equation}
where the relaxation variable of $\mb{v}$ is:
\begin{equation} \label{eq:nien-16-rla:relax_sv_v}
    \iter{\mb{r}_{\mb{v},\alpha}}{k+1}\teq\alpha\mb{G}^{1/2}\iter{\mb{x}}{k+1}+\left(1-\alpha\right)\iter{\mb{v}}{k} \, ,
\end{equation}
and $\msb{\nu}$ is the Lagrange multiplier of the redundant equality constraint. One can easily verify that $\iter{\msb{\nu}}{k}=\mb{0}$ for $k=0,1,\ldots$ if we initialize $\msb{\nu}$ as $\iter{\msb{\nu}}{0}=\mb{0}$.

The additional equality constraint introduces an additional inner-product term and a quadratic penalty term to the $\mb{x}$-update. The latter can be used to cancel the non-separable Hessian of the AL penalty term as in explicit linearization. By choosing the same AL penalty parameter $\rho>0$ for the additional constraint, the Hessian matrix of the quadratic penalty term in the $\mb{x}$-update of \eqref{eq:nien-16-rla:relaxed_lal_iter0} is $\rho\mb{A}'\mb{A}+\rho\mb{G}=\rho\mb{D}_{\mb{A}}$. In other words, by choosing $\mb{G}$ in \eqref{eq:nien-16-rla:aa_majorization}, the quadratic penalty term in the $\mb{x}$-update of \eqref{eq:nien-16-rla:relaxed_lal_iter0} becomes separable, and the $\mb{x}$-update becomes an efficient proximal mapping of $\phi$, as seen in \eqref{eq:nien-16-rla:relaxed_lal_iter} below.

Next we analyze the convergence rate of the proposed relaxed LALM method \eqref{eq:nien-16-rla:relaxed_lal_iter0}. With the additional redundant equality constraint, the Lagrangian becomes
\begin{equation} \label{eq:nien-16-rla:new_lag}
    \fx{\mathcal{L}'}{\mb{x},\mb{u},\msb{\mu},\mb{v},\msb{\nu}}
    \teq
    \fx{\mathcal{L}}{\mb{x},\mb{u},\msb{\mu}}-\langle\msb{\nu},\mb{G}^{1/2}\mb{x}-\mb{v}\rangle \, .
\end{equation}
Setting gradients of $\mathcal{L}'$ with respect to $\mb{x}$, $\mb{u}$, $\msb{\mu}$, $\mb{v}$, and $\msb{\nu}$ to be zero yields a necessary condition for a saddle-point \mbox{$\hat{\mb{w}}=\left(\hat{\mb{x}},\hat{\mb{u}},\hat{\msb{\mu}},\hat{\mb{v}},\hat{\msb{\nu}}\right)$} of $\mathcal{L}'$. It follows that $\hat{\msb{\nu}}=\fx{\nabla_{\mb{v}}\mathcal{L}'}{\hat{\mb{w}}}=\mb{0}$. Therefore, setting $\iter{\msb{\nu}}{0}=\mb{0}$ is indeed a natural choice for initializing $\msb{\nu}$. Moreover, since $\hat{\msb{\nu}}=\mb{0}$, the gap function $\mathcal{G}'$ of the new problem \eqref{eq:nien-16-rla:redundant_eq_constrained_prob} coincides with \eqref{eq:nien-16-rla:def_dual_gap}, and we can compare the convergence rate of the simple and proposed relaxed algorithms directly.

\begin{theorem} \label{thm:nien-16-rla:proposed_relax}
Let $\mb{w}_K=\left(\mb{x}_K,\mb{u}_K,\mb{v}_K,\msb{\mu}_K,\msb{\nu}_K\right)$ be the time-averages of the iterates of LALM with proposed relaxation in \eqref{eq:nien-16-rla:relaxed_lal_iter0}, where $\rho>0$ and  $0<\alpha<2$. When initializing $\mb{v}$ and $\msb{\nu}$ as $\iter{\mb{v}}{0}=\mb{G}^{1/2}\iter{\mb{x}}{0}$ and $\iter{\msb{\nu}}{0}=\mb{0}$, respectively, we have
\begin{equation} \label{eq:nien-16-rla:conv_rate_proposed_relax}
    \fx{\mathcal{G}'}{\mb{w}_K;\hat{\mb{w}}}
    \leq
    \tfrac{1}{K}\left(A_{\mb{D}_{\psi}}+\overline{B}_{\alpha,\rho,\mb{D}_{\mb{A}}}+C_{\alpha,\rho}\right) \, ,
\end{equation}
where $A_{\mb{D}_{\psi}}$ and $C_{\alpha,\rho}$ were defined in \eqref{eq:nien-16-rla:def_A} and \eqref{eq:nien-16-rla:def_C}, and
\begin{equation} \label{eq:nien-16-rla:def_B_bar}
    \overline{B}_{\alpha,\rho,\mb{D}_{\mb{A}}}
    \teq
    \tfrac{\rho}{2\alpha}\bnorm{\iter{\mb{v}}{0}-\hat{\mb{v}}}{2}^2
    =
    \tfrac{\rho}{2\alpha}\bnorm{\iter{\mb{x}}{0}-\hat{\mb{x}}}{\mb{D}_{\mb{A}}-\mb{A}'\mb{A}}^2 \, .
\end{equation}
\end{theorem}
\begin{proof}
The proof is in the supplementary material.
\end{proof}

Theorem~\ref{thm:nien-16-rla:proposed_relax} shows the $\fx{\mathcal{O}}{1/K}$ convergence rate of \eqref{eq:nien-16-rla:relaxed_lal_iter0}. Due to the different variable splitting scheme, the term introduced by the implicit linearization trick in \eqref{eq:nien-16-rla:relaxed_lal_iter0} (i.e., $\overline{B}_{\alpha,\rho,\mb{D}_{\mb{A}}}$) also depends on the relaxation parameter $\alpha$, improving convergence rate scalibility with $\alpha$ in \eqref{eq:nien-16-rla:relaxed_lal_iter0} over \eqref{eq:nien-16-rla:simple_relaxed_lal_iter}. This theorem provides a theoretical explanation why \eqref{eq:nien-16-rla:relaxed_lal_iter0} converges faster than \eqref{eq:nien-16-rla:simple_relaxed_lal_iter} in the experiments shown later\footnote{When $\psi$ has large curvature (thus, $\alpha$-dependent terms do not dominate the constant multiplying $1/K$), we can use techniques as in \cite{azadi:14:tao,ouyang:15:aal} to reduce the $\psi$-dependent constant. In X-ray CT, the data-fidelity term often dominates the cost function, so $A_{\mb{D}_{\psi}}\ll\overline{B}_{\alpha,\rho,\mb{D}_{\mb{A}}}$.}.

For practical implementation, the remaining concern is multiplications by $\mb{G}^{1/2}$ in \eqref{eq:nien-16-rla:relaxed_lal_iter0}. There is no efficient way to compute the square root of $\mb{G}$ for any $\mb{A}$ in general, especially when $\mb{A}$ is large and unstructured like in X-ray CT. To solve this problem, let $\mb{h}\teq\mb{G}^{1/2}\mb{v}+\mb{A}'\mb{y}$. We rewrite \eqref{eq:nien-16-rla:relaxed_lal_iter0} so that no explicit multiplication by $\mb{G}^{1/2}$ is needed (the derivation is in the supplementary material), leading to the following ``LALM with proposed relaxation'' algorithm:
\begin{equation} \label{eq:nien-16-rla:relaxed_lal_iter}
    \begin{cases}
    \iter{\mb{x}}{k+1}\in\argmin{\mb{x}}{
    \begin{aligned}
    &\fx{\phi}{\mb{x}}+\bfx{Q_{\psi}}{\mb{x};\iter{\mb{x}}{k}} \\
    &+\tfrac{1}{2}\bnorm{\mb{x}-\left(\rho\mb{D}_{\mb{A}}\right)^{-1}\iter{\msb{\gamma}}{k+1}}{\rho\mb{D}_{\mb{A}}}^2
    \end{aligned}
    } \\
    \iter{\mb{u}}{k+1}\in\argmin{\mb{u}}{\fx{g_{\mb{y}}}{\mb{u}}+\langle\iter{\msb{\mu}}{k},\mb{u}\rangle+\tfrac{\rho}{2}\bnorm{\iter{\mb{r}_{\mb{u},\alpha}}{k+1}-\mb{u}}{2}^2} \\
    \iter{\msb{\mu}}{k+1}=\iter{\msb{\mu}}{k}-\rho\big(\iter{\mb{r}_{\mb{u},\alpha}}{k+1}-\iter{\mb{u}}{k+1}\big) \\
    \iter{\mb{h}}{k+1}=\alpha\iter{\msb{\eta}}{k+1}+\left(1-\alpha\right)\iter{\mb{h}}{k} \, ,
    \end{cases}
\end{equation}
where
\begin{equation} \label{eq:nien-16-rla:def_gamma}
    \iter{\msb{\gamma}}{k+1}\teq\rho\mb{A}'\big(\iter{\mb{u}}{k}-\mb{y}+\rho^{-1}\iter{\msb{\mu}}{k}\big)+\rho\iter{\mb{h}}{k} \, ,
\end{equation}
and
\begin{equation} \label{eq:nien-16-rla:def_s}
    \iter{\msb{\eta}}{k+1}\teq\mb{D}_{\mb{A}}\iter{\mb{x}}{k+1}-\mb{A}'\big(\mb{A}\iter{\mb{x}}{k+1}-\mb{y}\big) \, .
\end{equation}

When $g_{\mb{y}}$ is a quadratic loss, i.e., $\fx{g_{\mb{y}}}{\mb{z}}=\left(1/2\right)\norm{\mb{z}-\mb{y}}{2}^2$, we further simplify the proposed relaxed LALM by manipulations like those in \cite{nien:15:fxr} (omitted here for brevity) as:
\begin{equation} \label{eq:nien-16-rla:relaxed_lal_iter_quad}
    \begin{cases}
    \iter{\msb{\gamma}}{k+1}=\left(\rho-1\right)\iter{\mb{g}}{k}+\rho\iter{\mb{h}}{k} \\
    \iter{\mb{x}}{k+1}\in\argmin{\mb{x}}{
    \begin{aligned}
    &\fx{\phi}{\mb{x}}+\bfx{Q_{\psi}}{\mb{x};\iter{\mb{x}}{k}} \\
    &+\tfrac{1}{2}\bnorm{\mb{x}-\left(\rho\mb{D}_{\mb{A}}\right)^{-1}\iter{\msb{\gamma}}{k+1}}{\rho\mb{D}_{\mb{A}}}^2
    \end{aligned}
    } \\
    \iter{\msb{\zeta}}{k+1}\teq\fx{\nabla\L}{\iter{\mb{x}}{k+1}}=\mb{A}'\left(\mb{A}\iter{\mb{x}}{k+1}-\mb{y}\right) \\
    \iter{\mb{g}}{k+1}=\tfrac{\rho}{\rho+1}\big(\alpha\iter{\msb{\zeta}}{k+1}+\left(1-\alpha\right)\iter{\mb{g}}{k}\big)+\tfrac{1}{\rho+1}\iter{\mb{g}}{k} \\
    \iter{\mb{h}}{k+1}=\alpha\big(\mb{D}_{\mb{A}}\iter{\mb{x}}{k+1}-\iter{\msb{\zeta}}{k+1}\big)+\left(1-\alpha\right)\iter{\mb{h}}{k} \, ,
    \end{cases}
\end{equation}
where $\fx{\L}{\mb{x}}\teq\fx{g_{\mb{y}}}{\mb{Ax}}$ is the quadratic data-fidelity term, and $\mb{g}\teq\mb{A}'\left(\mb{u}-\mb{y}\right)$ \cite{nien:15:fxr}. For initialization, we suggest using $\iter{\mb{g}}{0}=\iter{\msb{\zeta}}{0}$ and $\iter{\mb{h}}{0}=\mb{D}_{\mb{A}}\iter{\mb{x}}{0}-\iter{\msb{\zeta}}{0}$ (Theorem~\ref{thm:nien-16-rla:proposed_relax}). The algorithm \eqref{eq:nien-16-rla:relaxed_lal_iter_quad} computes multiplications by $\mb{A}$ and $\mb{A}'$ only once per iteration and does not have to invert $\mb{A}'\mb{A}$, unlike standard relaxed AL methods \eqref{eq:nien-16-rla:relaxed_al_iter}. This property is especially useful when $\mb{A}'\mb{A}$ is large and unstructured. When $\alpha=1$, \eqref{eq:nien-16-rla:relaxed_lal_iter_quad} reverts to the unrelaxed LALM in \cite{nien:15:fxr}.

Lastly, we contrast our proposed relaxed LALM \eqref{eq:nien-16-rla:relaxed_lal_iter} with Chambolle's relaxed primal-dual algorithm \cite[Algorithm~2]{chambolle:16:ote}. Both algorithms exhibit $\fx{\mathcal{O}}{1/K}$ ergodic (i.e., with respect to the time-averaged iterates) convergence rate and $\alpha$-times speed-up when $\psi=0$. Using \eqref{eq:nien-16-rla:relaxed_lal_iter} would require one more multiplication by $\mb{A}'$ per iteration than in Chambolle's relaxed algorithm; however, the additional $\mb{A}'$ is not required with quadratic loss in \eqref{eq:nien-16-rla:relaxed_lal_iter_quad}. When $\psi\neq 0$, unlike Chambolle's relaxed algorithm in which one has to adjust the primal step size according to the value of $\alpha$ (effectively, one scales $\mb{D}_{\psi}$ by $1/(2-\alpha)$) \cite[Remark~6]{chambolle:16:ote}, the proposed relaxed LALM \eqref{eq:nien-16-rla:relaxed_lal_iter} does not require such step-size adjustment, which is especially useful when using $\alpha$ that is close to two.

\section{X-ray CT image reconstruction} \label{sec:nien-16-rla:method}
Consider the X-ray CT image reconstruction problem \cite{thibault:07:atd}:
\begin{equation} \label{eq:nien-16-rla:ct_problem}
    \hat{\mb{x}}\in\argmin{\mb{x}\in\Omega}{\tfrac{1}{2}\norm{\mb{y}-\mb{Ax}}{\mb{W}}^2+\fx{\R}{\mb{x}}} \, ,
\end{equation}
where $\mb{A}$ is the forward projection matrix of a CT scan \cite{long:10:3fa}, $\mb{y}$ is the noisy sinogram, $\mb{W}$ is the statistical diagonal weighting matrix, $\R$ denotes an edge-preserving regularizer, and $\Omega$ denotes a box-constraint on the image $\mb{x}$. We focus on the edge-preserving regularizer $\R$ defined as:
\begin{equation} \label{eq:nien-16-rla:def_reg}
    \fx{\R}{\mb{x}}
    \teq
    \sum_{i}^{} \beta_i \sum_{n}^{} \kappa_{n}\kappa_{n+s_i}\fx{\varphi_i}{[\mb{C}_i\mb{x}]_n} \, ,
\end{equation}
where $\beta_i$, $s_i$, $\varphi_i$, and $\mb{C}_i$ denote the regularization parameter, spatial offset, potential function, and finite difference matrix in the $i$th direction, respectively, and $\kappa_n$ is a voxel-dependent weight for improving resolution uniformity \cite{fessler:96:srp,cho:15:rdf}. In our experiments, we used $13$ directions to include all $26$ neighbors in 3-D CT.

\subsection{Relaxed OS-LALM for faster CT reconstruction} \label{subsec:nien-16-rla:ct_image_recon}
To solve X-ray CT image reconstruction \eqref{eq:nien-16-rla:ct_problem} using the proposed relaxed LALM \eqref{eq:nien-16-rla:relaxed_lal_iter_quad}, we apply the following substitution:
\begin{equation} \label{eq:nien-16-rla:ct_substitution}
    \begin{cases}
    \mb{A}\leftarrow\mb{W}^{1/2}\mb{A} \\
    \alhs{\mb{y}}{\mb{A}}\leftarrow\mb{W}^{1/2}\mb{y} \, ,
    \end{cases}
\end{equation}
and we set $\phi=\iota_{\Omega}$ and $\psi=\R$, where $\fx{\iota_{\Omega}}{\mb{x}}=0$ if $\mb{x}\in\Omega$, and $\fx{\iota_{\Omega}}{\mb{x}}=+\infty$ otherwise. The proximal mapping of $\iota_{\Omega}$ simply projects the input vector to the convex set $\Omega$, e.g., clipping negative values of $\mb{x}$ to zero for a non-negativity constraint. Theorems developed in Section~\ref{sec:nien-16-rla:relaxed_lalm} considered the ergodic convergence rate of the non-negative duality gap, which is not a common convergence metric for X-ray CT image reconstruction. However, the ergodic convergence rate analysis suggests how factors like $\alpha$, $\rho$, $\mb{D}_{\mb{A}}$, and $\mb{D}_{\psi}$ affect convergence speed (a LASSO regression example can be found in the supplementary material) and motivates our ``more practical'' (over-)relaxed \mbox{OS-LALM} summarized below.

\begin{algorithm}
    \KwIn{$M\geq1$, $1\leq\alpha<2$, and an initial (FBP) image $\mb{x}$.}
    \BlankLine
    set $\rho=1$, $\msb{\zeta}=\mb{g}=M\fx{\nabla\L_M}{\mb{x}}$, $\mb{h}=\mb{D}_{\L}\mb{x}-\msb{\zeta}$ \\
    \For{$k=1,2,\ldots$}{
    \For{$m=1,2,\ldots,M$}{
    $\alhs{\mb{s}\,\,\,}{\nx{\mb{x}}}=\rho\left(\mb{D}_{\L}\mb{x}-\mb{h}\right)+\left(1-\rho\right)\mb{g}$ \\
    $\nx{\mb{x}}=\big[\mb{x}-\left(\rho\mb{D}_{\L}+\mb{D}_{\R}\right)^{-1}\left(\mb{s}+\fx{\nabla\R}{\mb{x}}\right)\big]_{\Omega}$ \\
    $\alhs{\msb{\zeta}\,\,\,}{\nx{\mb{x}}}=M\fx{\nabla\L_m}{\nx{\mb{x}}}$ \\
    $\alhs{\nx{\mb{g}}}{\nx{\mb{x}}}=\tfrac{\rho}{\rho+1}\left(\alpha\msb{\zeta}+\left(1-\alpha\right)\mb{g}\right)+\tfrac{1}{\rho+1}\mb{g}$ \\
    $\alhs{\nx{\mb{h}}}{\nx{\mb{x}}}=\alpha\left(\mb{D}_{\L}\nx{\mb{x}}-\msb{\zeta}\right)+\left(1-\alpha\right)\mb{h}$ \\
    decrease $\rho$ using \eqref{eq:nien-16-rla:dec_rho}
    }
    }
    \caption{Proposed (over-)relaxed \mbox{OS-LALM} for \eqref{eq:nien-16-rla:ct_problem}.} \label{alg:nien-16-rla:relaxed_os_lalm}
\end{algorithm}

Algorithm~\ref{alg:nien-16-rla:relaxed_os_lalm} describes the proposed relaxed algorithm for solving the X-ray CT image reconstruction problem \eqref{eq:nien-16-rla:ct_problem}, where $\L_m$ denotes the data-fidelity term of the $m$th subset, and $[\cdot]_{\Omega}$ is an operator that projects the input vector onto the convex set $\Omega$, e.g., truncating zeros for $\Omega\teq\left\{\mb{x}\,|x_i\geq0\text{ for all $i$}\right\}$. All variables are updated in-place, and we use the superscript $\nx{(\cdot)}$ to denote the new values that replace the old values. We also use the substitution $\mb{s}\teq\rho\mb{D}_{\L}\mb{x}-\nx{\msb{\gamma}}$ in the proposed method, so Algorithm~\ref{alg:nien-16-rla:relaxed_os_lalm} has comparable form with the unrelaxed \mbox{OS-LALM} \cite{nien:15:fxr}; however, such substitution is not necessary.

As seen in Algorithm~\ref{alg:nien-16-rla:relaxed_os_lalm}, the proposed relaxed \mbox{OS-LALM} has the form of \eqref{eq:nien-16-rla:relaxed_lal_iter_quad} but uses some modifications that violate assumptions in our theorems but speed up ``convergence'' in practice. First, although Theorem~\ref{thm:nien-16-rla:proposed_relax} assumes a constant majorizing matrix $\mb{D}_{\R}$ for the Lipschitz term $\R$ (e.g., the maximum curvature of $\R$), we use the iteration-dependent Huber's curvature of $\R$ \cite{erdogan:99:maf} for faster convergence (the same in other algorithms for comparison). Second, since the updates in \eqref{eq:nien-16-rla:relaxed_lal_iter_quad} depend only on the gradients of $\L$, we can further accelerate the gradient computation by using partial projection data, i.e., ordered subsets. Lastly, we incorporate continuation technique (i.e., decreasing the AL penalty parameter $\rho$ every iteration) in the proposed algorithm as described in the next subsection.

To select the number of subsets, we used the rule suggested in \cite[Eqn. 55 and 57]{nien:15:fxr}. However, since over-relaxation provides two-times acceleration, we used $50\%$ of the suggested number of subsets (for the unrelaxed \mbox{OS-LALM}) yet achieved similar convergence speed (faster in runtime since fewer regularizer gradient evaluations are performed) and more stable reconstruction. For the implicit linearization, we use the diagonal majorizing matrix $\diag{\mb{A}'\mb{WA1}}$ for $\mb{A}'\mb{WA}$ \cite{erdogan:99:osa}, the same diagonal majorizing matrix $\mb{D}_{\L}$ for the quadratic loss function used in OS algorithms.

Furthermore, Algorithm~\ref{alg:nien-16-rla:simple_relaxed_os_lalm} depicts the OS version of the simple relaxed algorithm \eqref{eq:nien-16-rla:simple_relaxed_lal_iter} for solving \eqref{eq:nien-16-rla:ct_problem} (derivation is omitted here). The main difference between Algorithm~\ref{alg:nien-16-rla:relaxed_os_lalm} and Algorithm~\ref{alg:nien-16-rla:simple_relaxed_os_lalm} is the extra recursion of variable $\mb{h}$. When $\alpha=1$, both algorithms revert to the unrelaxed \mbox{OS-LALM} \cite{nien:15:fxr}.

\begin{algorithm}
    \KwIn{$M\geq1$, $1\leq\alpha<2$, and an initial (FBP) image $\mb{x}$.}
    \BlankLine
    set $\rho=1$, $\msb{\zeta}=\mb{g}=M\fx{\nabla\L_M}{\mb{x}}$ \\
    \For{$k=1,2,\ldots$}{
    \For{$m=1,2,\ldots,M$}{
    $\alhs{\mb{s}\,\,\,}{\nx{\mb{x}}}=\rho\,\msb{\zeta}+\left(1-\rho\right)\mb{g}$ \\
    $\nx{\mb{x}}=\big[\mb{x}-\left(\rho\mb{D}_{\L}+\mb{D}_{\R}\right)^{-1}\left(\mb{s}+\fx{\nabla\R}{\mb{x}}\right)\big]_{\Omega}$ \\
    $\alhs{\nx{\msb{\zeta}}}{\nx{\mb{x}}}=M\fx{\nabla\L_m}{\nx{\mb{x}}}$ \\
    $\alhs{\nx{\mb{g}}}{\nx{\mb{x}}}=\tfrac{\rho}{\rho+1}\left(\alpha\nx{\msb{\zeta}}+\left(1-\alpha\right)\mb{g}\right)+\tfrac{1}{\rho+1}\mb{g}$ \\
    decrease $\rho$ using \eqref{eq:nien-16-rla:dec_rho}
    }
    }
    \caption{Simple (over-)relaxed \mbox{OS-LALM} for \eqref{eq:nien-16-rla:ct_problem}.} \label{alg:nien-16-rla:simple_relaxed_os_lalm}
\end{algorithm}

\subsection{Further speed-up with continuation} \label{subsec:nien-16-rla:opt_cont_seq}
We also use a continuation technique \cite{nien:15:fxr} to speed up convergence; that is, we decrease $\rho$ gradually with iteration. Note that $\rho\mb{D}_{\L}+\mb{D}_{\R}$ is the inverse of the voxel-dependent step size of image updates; decreasing $\rho$ increases step sizes gradually as iteration progress. Due to the extra relaxation parameter $\alpha$, the good decreasing continuation sequence differs from that in \cite{nien:15:fxr}. We use the following $\alpha$-dependent continuation sequence for the proposed relaxed LALM ($1\leq\alpha<2$):
\begin{equation} \label{eq:nien-16-rla:dec_rho}
    \fx{\rho_k}{\alpha}
    =
    \begin{cases}
    1, & \text{if $k=0$} \\
    \tfrac{\pi}{\alpha(k+1)}\sqrt{1-\left(\tfrac{\pi}{2\alpha(k+1)}\right)^2}, & \text{otherwise} \, .
    \end{cases}
\end{equation}
The supplementary material describes the rationale for this continuation sequence. When using OS, $\rho$ decreases every subiteration, and the counter $k$ in \eqref{eq:nien-16-rla:dec_rho} denotes the number of subiterations, instead of the number of iterations.

\section{Experimental results} \label{sec:nien-16-rla:result}
This section reports numerical results for 3-D X-ray CT image reconstruction using one conventional algorithm (\mbox{OS-SQS} \cite{erdogan:99:osa}) and four contemporary algorithms:
\begin{itemize}
    \item \textbf{OS-FGM2}: the OS variant of the standard fast gradient method proposed in \cite{kim:16:ofo,kim:15:cos},
    \item \textbf{OS-LALM}: the OS variant of the unrelaxed linearized AL method proposed in \cite{nien:15:fxr},
    \item \textbf{OS-OGM2}: the OS variant of the optimal fast gradient method proposed in \cite{kim:16:ofo}, and
    \item \textbf{Relaxed OS-LALM}: the OS variants of the proposed relaxed linearized AL methods given in Algorithm~\ref{alg:nien-16-rla:relaxed_os_lalm} (proposed) and Algorithm~\ref{alg:nien-16-rla:simple_relaxed_os_lalm} (simple) above ($\alpha=1.999$ unless otherwise specified).
\end{itemize}

\subsection{XCAT phantom} \label{subsec:nien-16-rla:xcat}
We simulated an axial CT scan using a $1024 \times 1024 \times 154$ XCAT phantom \cite{segars:08:rcs} for $500$ mm transaxial field-of-view (FOV), where $\Delta_x=\Delta_y=0.4883$ mm and $\Delta_z=0.625$ mm. An $888 \times 64 \times 984$ ([detector columns] $\times$ [detector rows] $\times$ [projection views]) noisy (with Poisson noise) sinogram is numerically generated with GE LightSpeed fan-beam geometry corresponding to a monoenergetic source at $70$ keV with $10^5$ incident photons per ray and no scatter. We reconstructed a $512 \times 512 \times 90$ image volume with a coarser grid, where $\Delta_x=\Delta_y=0.9776$ mm and $\Delta_z=0.625$ mm. The statistical weighting matrix $\mb{W}$ is defined as a diagonal matrix with diagonal entries $w_j\teq\fx{\exp}{-y_j}$, and an edge-preserving regularizer is used with $\fx{\varphi_i}{t}\teq\delta^2\left(\left|t/\delta\right|-\fx{\log}{1+\left|t/\delta\right|}\right)$ ($\delta=10$ HU) and parameters $\beta_i$ set to achieve a reasonable noise-resolution trade-off. We used $12$ subsets for the relaxed \mbox{OS-LALM}, while \cite[Eqn. 55]{nien:15:fxr} suggests using about $24$ subsets for the unrelaxed \mbox{OS-LALM}.

\begin{figure*}
    \centering
    \includegraphics[width=\textwidth]{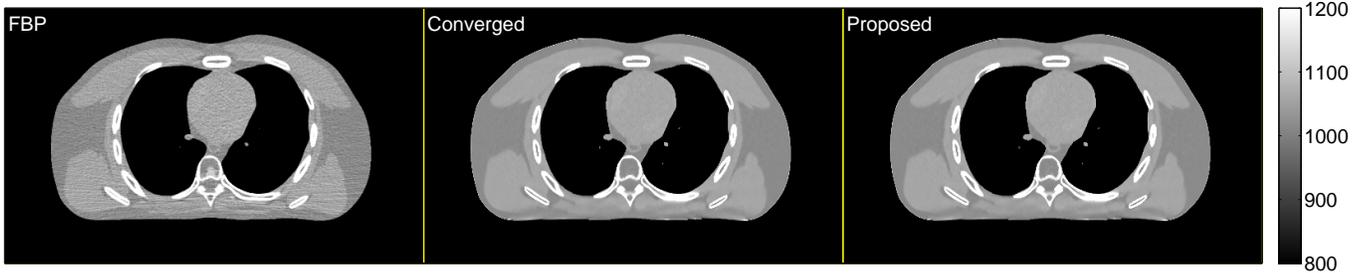}
    \caption{XCAT: Cropped images (displayed from $800$ to $1200$ HU) from the central transaxial plane of the initial FBP image $\iter{\mb{x}}{0}$ (left), the reference reconstruction $\mb{x}^{\star}$ (center), and the reconstructed image $\iter{\mb{x}}{20}$ using the proposed algorithm (relaxed \mbox{OS-LALM} with $12$ subsets) after $20$ iterations (right).}
    \label{fig:nien-16-rla:xcat_fbp_conv_prop}
\end{figure*}

\begin{figure*}
    \centering
    \subfigure[$12$ subsets]{
    \includegraphics[width=0.45\textwidth]{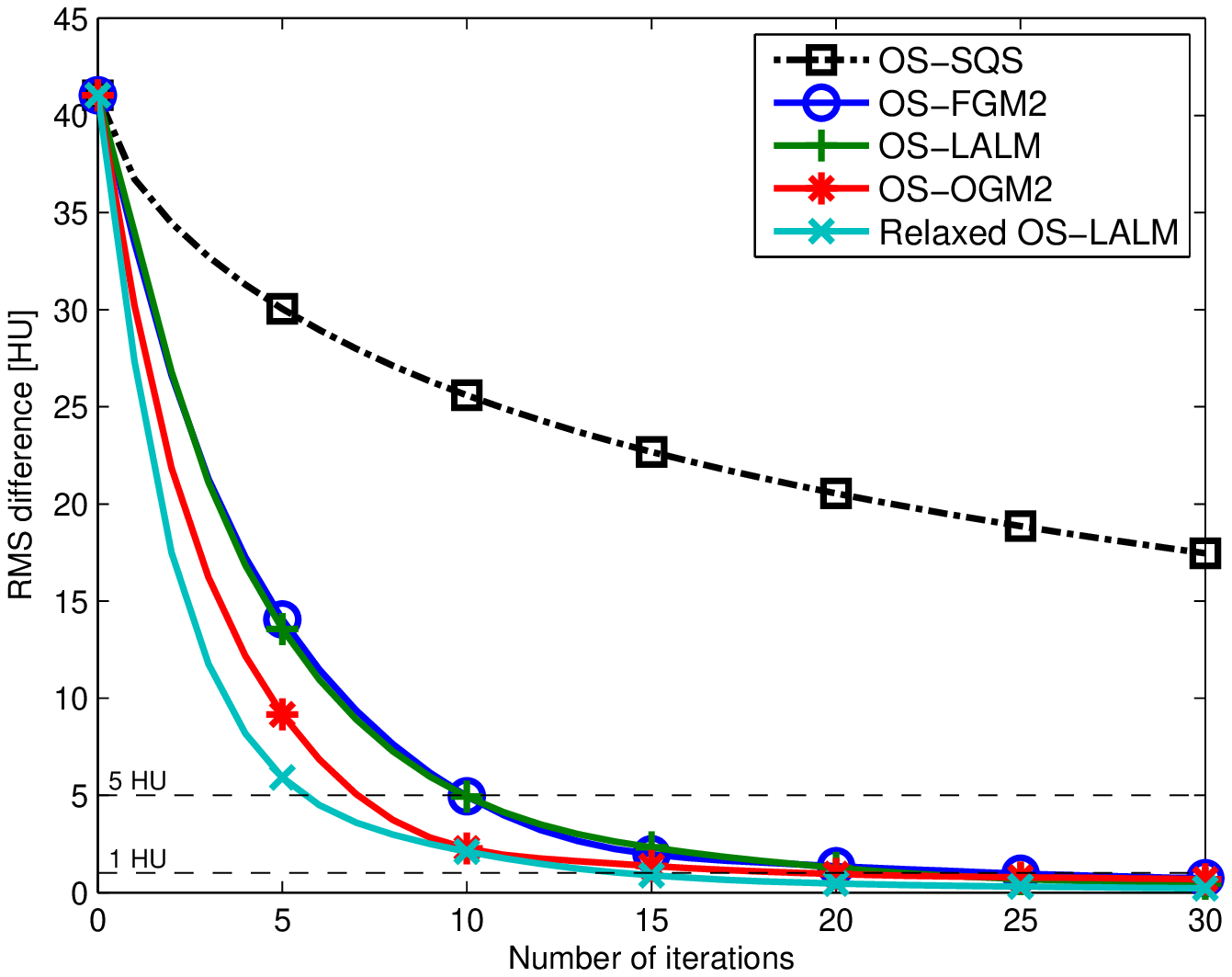}
    \label{fig:nien-16-rla:xcat_rmsd_12_blocks}
    }
    \subfigure[$24$ subsets]{
    \includegraphics[width=0.45\textwidth]{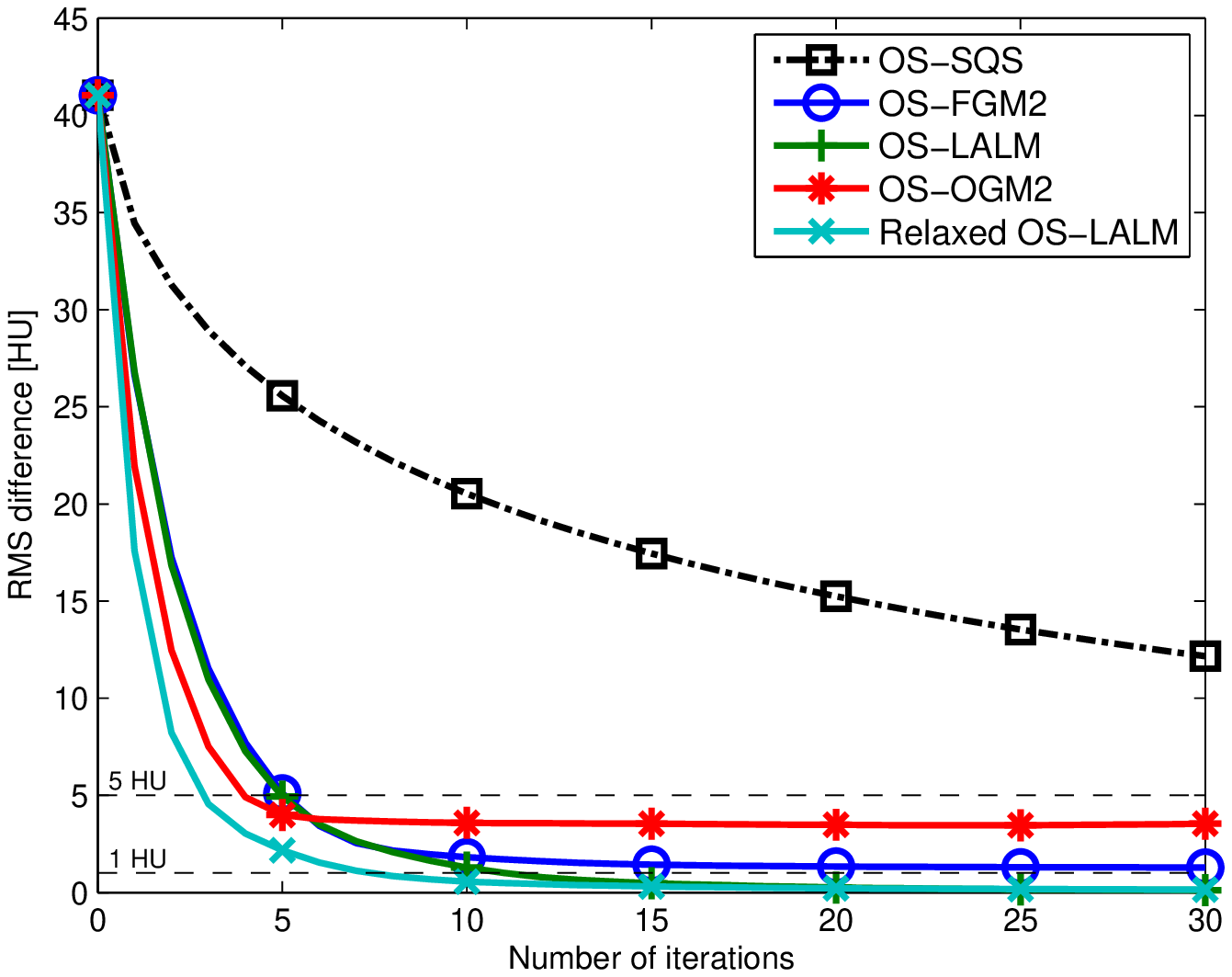}
    \label{fig:nien-16-rla:xcat_rmsd_24_blocks}
    }
    \caption{XCAT: Convergence rate curves of different OS algorithms with (a) $12$ subsets and (b) $24$ subsets. The proposed relaxed \mbox{OS-LALM} with $12$ subsets exhibits similar convergence rate as the unrelaxed \mbox{OS-LALM} with $24$ subsets.}
    \label{fig:nien-16-rla:xcat_rmsd_cmp}
\end{figure*}

Figure~\ref{fig:nien-16-rla:xcat_fbp_conv_prop} shows the cropped images (displayed from $800$ to $1200$ HU [modified so that air is $0$]) from the central transaxial plane of the initial FBP image $\iter{\mb{x}}{0}$, the reference reconstruction $\mb{x}^{\star}$ (generated by running thousands of iterations of the convergent FGM with adaptive restart \cite{odonoghue:15:arf}), and the reconstructed image $\iter{\mb{x}}{20}$ using the proposed algorithm (relaxed \mbox{OS-LALM} with $12$ subsets) after $20$ iterations. There is no visible difference between the reference reconstruction and our reconstruction. To analyze the proposed algorithm quantitatively, Figure~\ref{fig:nien-16-rla:xcat_rmsd_cmp} shows the RMS differences between the reference reconstruction $\mb{x}^{\star}$ and the reconstructed image $\iter{\mb{x}}{k}$ using different algorithms as a function of iteration\footnote{All algorithms listed above require one forward/back-projection pair and $M$ (the number of subsets) regularizer gradient evaluations (plus some negligible overhead) per iteration, so comparing the convergence rate as a function of iteration is fair.} with $12$ and $24$ subsets. As seen in Figure~\ref{fig:nien-16-rla:xcat_rmsd_cmp}, the proposed algorithm (cyan curves) is approximately twice as fast as the unrelaxed \mbox{OS-LALM} (green curves) at least in early iterations. Furthermore, comparing with \mbox{OS-FGM2} and \mbox{OS-OGM2}, the proposed algorithm converges faster and is more stable when using more subsets for acceleration. Difference images using different algorithms and additional experimental results are shown in the supplementary material.

To illustrate the improved speed-up of the proposed relaxation (Algorithm~\ref{alg:nien-16-rla:relaxed_os_lalm}) over the simple one (Algorithm~\ref{alg:nien-16-rla:simple_relaxed_os_lalm}), Figure~\ref{fig:nien-16-rla:xcat_rlx_cmp_12_blocks} shows convergence rate curves of different relaxed algorithms ($12$ subsets and $\alpha=1.999$) with (a) a fixed AL penalty parameter $\rho=0.05$ and (b) the decreasing sequence $\rho_k$ in \eqref{eq:nien-16-rla:dec_rho}. As seen in Figure~\ref{fig:nien-16-rla:xcat_relax_cmp_rho_p05_12_blocks}, the simple relaxation does not provide much acceleration, especially after $10$ iterations. In contrast, the proposed relaxation accelerates convergence about twice (i.e., $\alpha$-times), as predicted by Theorem~\ref{thm:nien-16-rla:proposed_relax}. When the decreasing sequence of $\rho_k$ is used, as seen in Figure~\ref{fig:nien-16-rla:xcat_relax_cmp_12_blocks}, the simple relaxation seems to provide somewhat more acceleration than before; however, the proposed relaxation still outperforms the simple one, illustrating approximately two-fold speed-up over the unrelaxed counterpart.

\begin{figure*}
    \centering
    \subfigure[Fixed $\rho=0.05$]{
    \includegraphics[width=0.45\textwidth]{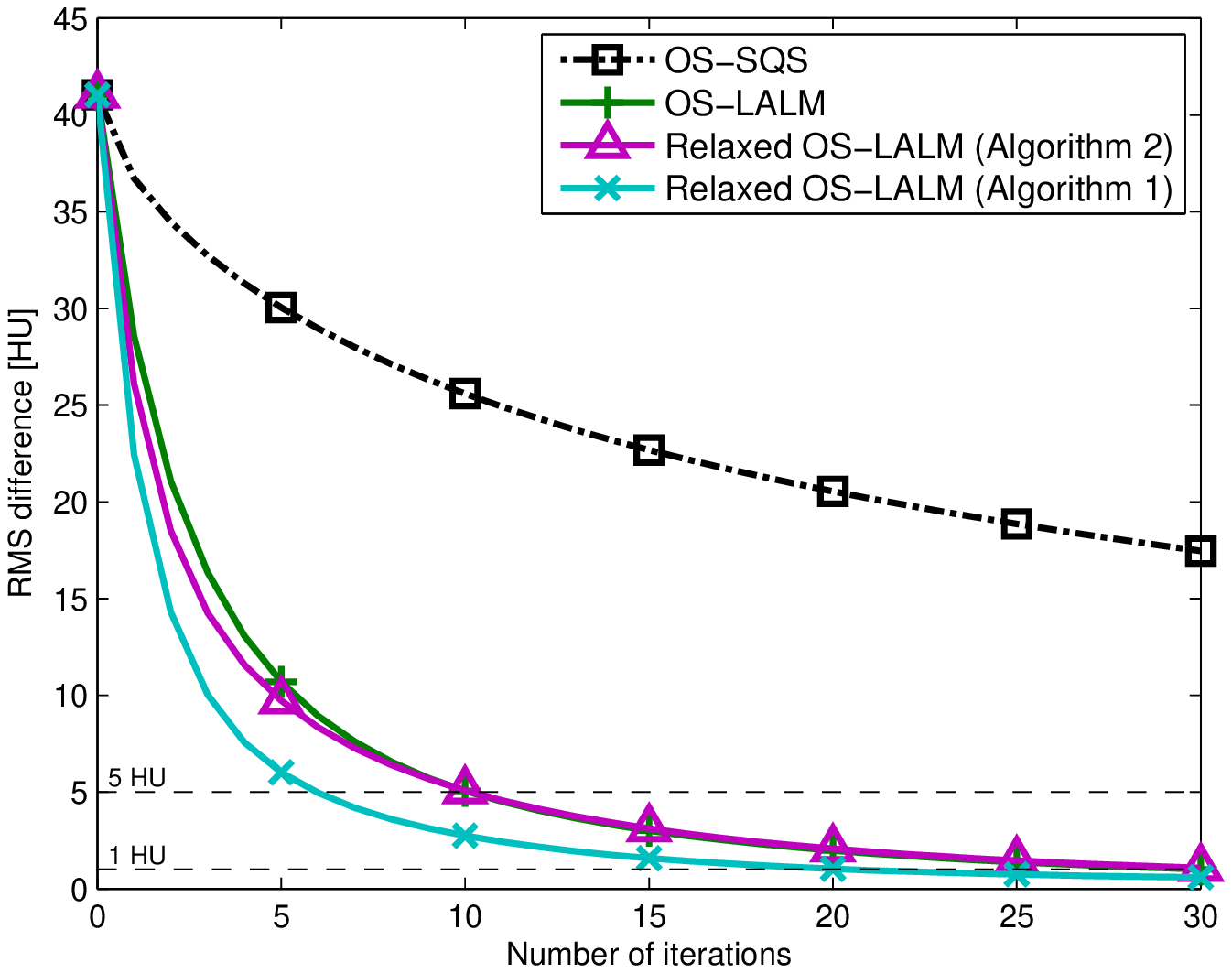}
    \label{fig:nien-16-rla:xcat_relax_cmp_rho_p05_12_blocks}
    }
    \subfigure[Decreasing $\rho_k$ in \eqref{eq:nien-16-rla:dec_rho}]{
    \includegraphics[width=0.45\textwidth]{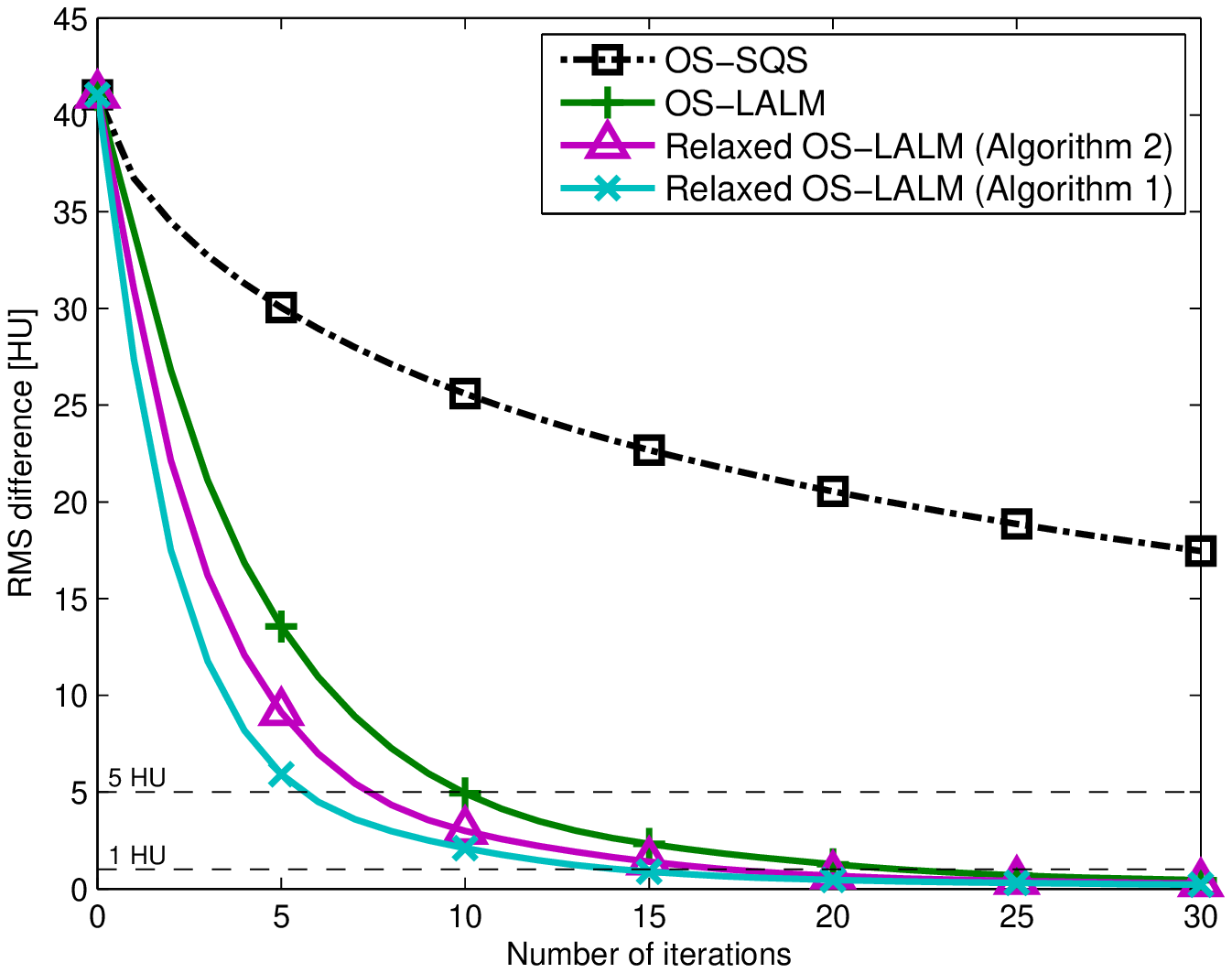}
    \label{fig:nien-16-rla:xcat_relax_cmp_12_blocks}
    }
    \caption{XCAT: Convergence rate curves of different relaxed algorithms ($12$ subsets and $\alpha=1.999$) with (a) a fixed AL penalty parameter $\rho=0.05$ and (b) the decreasing sequence $\rho_k$ in \eqref{eq:nien-16-rla:dec_rho}.}
    \label{fig:nien-16-rla:xcat_rlx_cmp_12_blocks}
\end{figure*}

\subsection{Chest scan} \label{subsec:nien-16-rla:chest}
\begin{figure*}
    \centering
    \includegraphics[width=\textwidth]{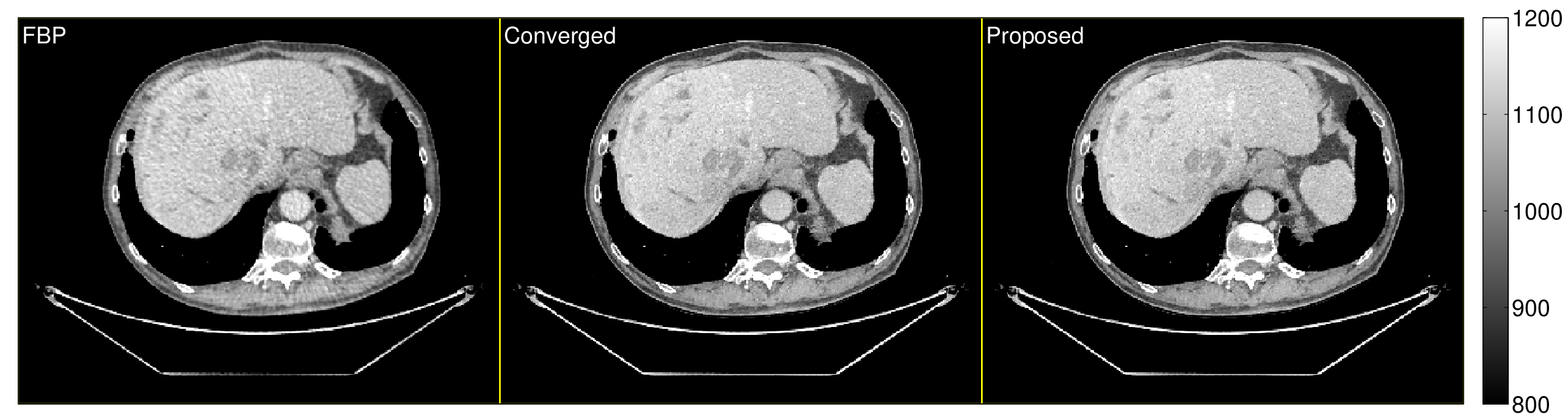}
    \caption{Chest: Cropped images (displayed from $800$ to $1200$ HU) from the central transaxial plane of the initial FBP image $\iter{\mb{x}}{0}$ (left), the reference reconstruction $\mb{x}^{\star}$ (center), and the reconstructed image $\iter{\mb{x}}{20}$ using the proposed algorithm (relaxed \mbox{OS-LALM} with $10$ subsets) after $20$ iterations (right).}
    \label{fig:nien-16-rla:ct53_fbp_conv_prop}
\end{figure*}

\begin{figure*}
    \centering
    \subfigure[$10$ subsets]{
    \includegraphics[width=0.45\textwidth]{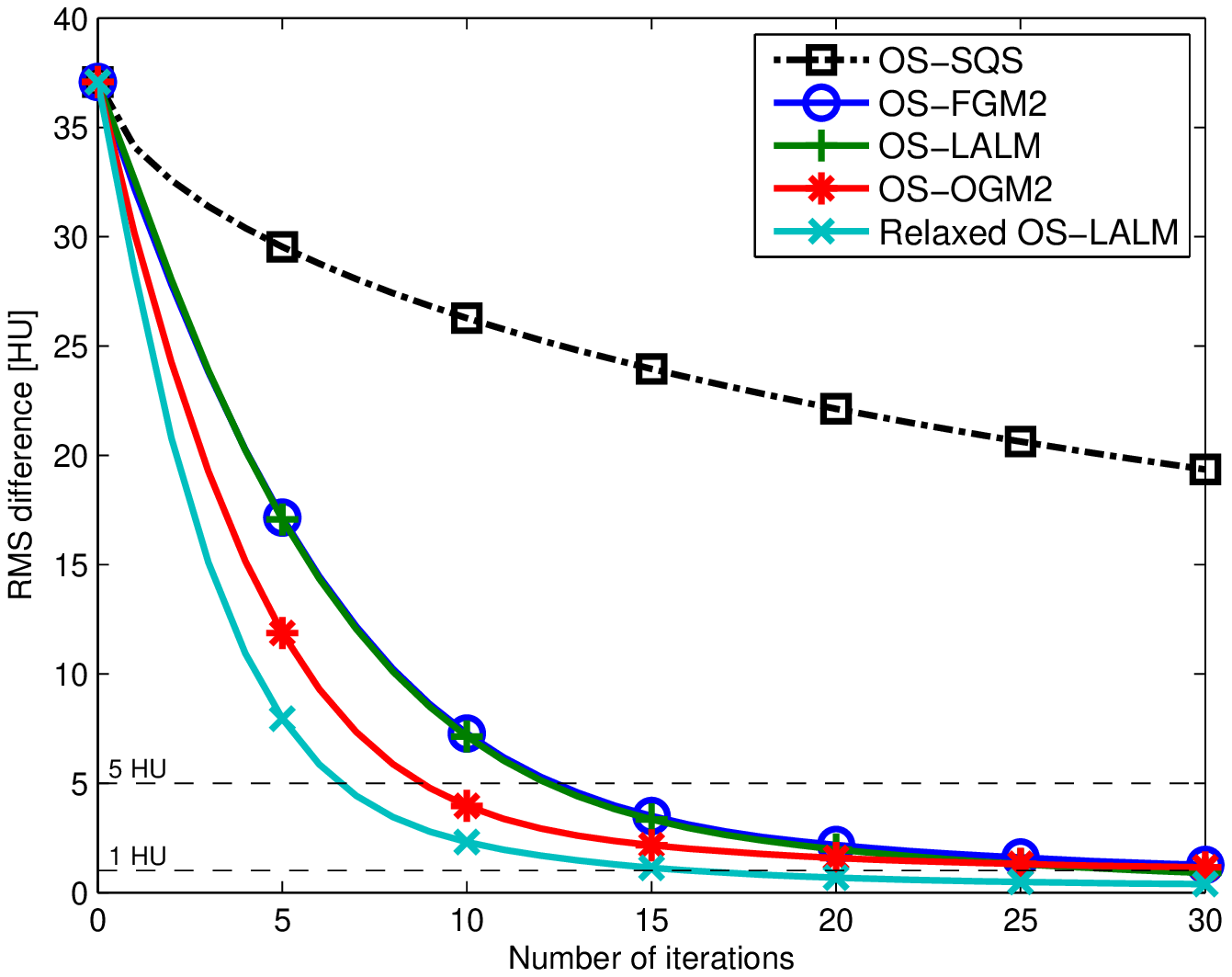}
    \label{fig:nien-16-rla:chest_rmsd_10_blocks}
    }
    \subfigure[$20$ subsets]{
    \includegraphics[width=0.45\textwidth]{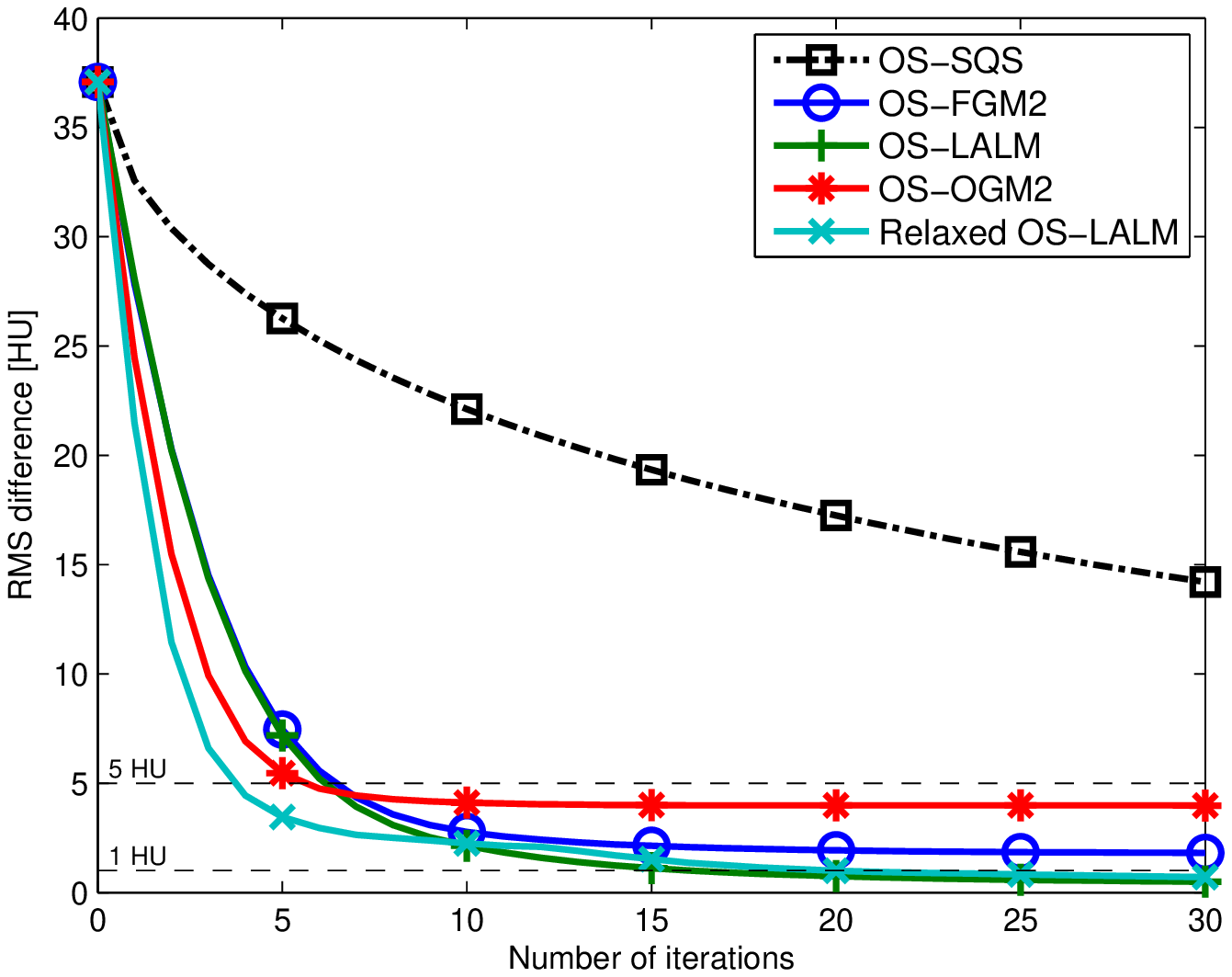}
    \label{fig:nien-16-rla:chest_rmsd_20_blocks}
    }
    \caption{Chest: Convergence rate curves of different OS algorithms with (a) $10$ subsets and (b) $20$ subsets. The proposed relaxed \mbox{OS-LALM} with $10$ subsets exhibits similar convergence rate as the unrelaxed \mbox{OS-LALM} with $20$ subsets.}
    \label{fig:nien-16-rla:chest_rmsd_cmp}
\end{figure*}

We reconstructed a $600\times 600\times 222$ image volume, where $\Delta_x=\Delta_y=1.1667$ mm and $\Delta_z=0.625$ mm, from a chest region helical CT scan. The size of sinogram is $888\times 64\times 3611$ and pitch $1.0$ (about $3.7$ rotations with rotation time $0.4$ seconds). The tube current and tube voltage of the X-ray source are $750$ mA and $120$ kVp, respectively. We started from a smoothed FBP image $\iter{\mb{x}}{0}$ and tuned the statistical weights \cite{chang:14:sxr} and the $q$-generalized Gaussian MRF regularization parameters \cite{cho:15:rdf} to emulate the MBIR method \cite{thibault:07:atd,shuman:13:mbi}. We used $10$ subsets for the relaxed \mbox{OS-LALM}, while \cite[Eqn. 57]{nien:15:fxr} suggests using about $20$ subsets for the unrelaxed \mbox{OS-LALM}. Figure~\ref{fig:nien-16-rla:ct53_fbp_conv_prop} shows the cropped images from the central transaxial plane of the initial FBP image $\iter{\mb{x}}{0}$, the reference reconstruction $\mb{x}^{\star}$, and the reconstructed image $\iter{\mb{x}}{20}$ using the proposed algorithm (relaxed \mbox{OS-LALM} with $10$ subsets) after $20$ iterations. Figure~\ref{fig:nien-16-rla:chest_rmsd_cmp} shows the RMS differences between the reference reconstruction $\mb{x}^{\star}$ and the reconstructed image $\iter{\mb{x}}{k}$ using different algorithms as a function of iteration with $10$ and $20$ subsets. The proposed relaxed \mbox{OS-LALM} shows about two-times faster convergence rate, comparing to its unrelaxed counterpart, with moderate number of subsets. The speed-up diminishes as the iterate approaches the solution. Furthermore, the faster relaxed \mbox{OS-LALM} seems likely to be more sensitive to gradient approximation errors and exhibits ripples in convergence rate curves when using too many subsets for acceleration. In contrast, the slower unrelaxed \mbox{OS-LALM} is less sensitive to gradient error when using more subsets and does not exhibit such ripples in convergence rate curves. Compared with \mbox{OS-FGM2} and \mbox{OS-OGM2}, the proposed relaxed \mbox{OS-LALM} has smaller limit cycles and might be more stable for practical use.

\section{Discussion and conclusions} \label{sec:nien-16-rla:discussion_conclusion}
In this paper, we proposed a non-trivial relaxed variant of LALM and applied it to X-ray CT image reconstruction. Experimental results with simulated and real CT scan data showed that our proposed relaxed algorithm ``converges'' about twice as fast as its unrelaxed counterpart, outperforming state-of-the-art fast iterative algorithms using momentum \cite{kim:15:cos,kim:16:ofo}. This speed-up means that one needs fewer subsets to reach an RMS difference criteria like $1$ HU in a given number of iterations. For instance, we used $50\%$ of the number of subsets suggested by \cite{nien:15:fxr} (for the unrelaxed \mbox{OS-LALM}) in our experiment but found similar convergence speed with over-relaxation. Moreover, using fewer subsets can be beneficial for distributed computing \cite{rosen:13:ihc}, reducing communication overhead required after every update.

\section*{Acknowledgment} \label{sec:nien-16-rla:ack}
The authors thank GE Healthcare for providing sinogram data in our experiments. The authors would also like to thank the anonymous reviewers for their comments and suggestions.

\bibliographystyle{ieeetr}
\bibliography{nien-16-rla-main.bbl}

\end{document}


\title{Relaxed Linearized Algorithms for Faster X-Ray CT Image Reconstruction: Supplementary Material}
\author{Hung Nien, \textit{Member, IEEE}, and Jeffrey A. Fessler, \textit{Fellow, IEEE}\thanks{This work is supported in part by National Institutes of Health (NIH) grants \mbox{U01-EB-018753} and by equipment donations from Intel Corporation. Hung Nien and Jeffrey A. Fessler are with the Department of Electrical Engineering and Computer Science, University of Michigan, Ann Arbor, MI 48109, USA (e-mail: \texttt{\{hungnien,fessler\}@umich.edu}). Date of current version: December 11, 2015.}}

\maketitle

This supplementary material for \cite{nien:16:rla} has three parts. The first part analyzes the convergence rate of the simple and proposed relaxed linearized augmented Lagrangian (AL) methods (LALM's) in \cite{nien:16:rla} for solving an equality-constrained composite convex optimization problem. We demonstrate the convergence rate bound and the effect of relaxation with a numerical example (LASSO regression). The second part derives the continuation sequence we used in \cite{nien:16:rla}. The third part shows additional experimental results of applying the proposed relaxed LALM with ordered subsets (OS) for solving model-based X-ray computed tomography (CT) image reconstruction problems. The additional experimental results are consistent with the results we showed in \cite{nien:16:rla}, illustrating the efficiency and stability of the proposed relaxed OS-LALM over existing methods.

\section{Convergence rate analyses of the simple and proposed LALM's} \label{sec:nien-16-rla-supp:conv_analysis}
We begin by considering a more general equality-constrained composite convex optimization problem (for which the equality-constrained minimization problem considered in \cite{nien:16:rla} is a special case):
\begin{equation} \label{eq:nien-16-rla-supp:eq_cvx}
    \left(\hat{\mb{x}},\hat{\mb{u}}\right)
    \in
    \argmin{\mb{x},\mb{u}}{\fx{f}{\mb{x},\mb{u}}\teq\fx{g}{\mb{u}}+\fx{h}{\mb{x}}}\text{ s.t. }\mb{Kx}+\mb{Bu}=\mb{b} \, ,
\end{equation}
where both $g$ and $h$ are closed and proper convex functions. We further decompose $h\teq\phi+\psi$ into two convex functions $\phi$ and $\psi$, where $\phi$ is ``simple'' in the sense that it has an efficient proximal mapping, e.g., soft-shrinkage for the $\ell_1$-norm, and $\psi$ is continuously differentiable with $\mb{D}_{\psi}$-Lipschitz gradients (defined in \cite{nien:16:rla}). One example of $h$ is the edge-preserving regularizer with a non-negativity constraint (e.g., sum of a ``corner-rounded'' total-variation [TV] regularizer and the characteristic function of the non-negativity set) used in statistical image reconstruction methods \cite{nien:15:fxr,nien:16:rla}.

As mentioned in \cite{nien:16:rla}, solving a composite convex optimization problem with equality constraints like \eqref{eq:nien-16-rla-supp:eq_cvx} is equivalent to finding a saddle-point of the Lagrangian:
\begin{equation} \label{eq:nien-16-rla-supp:def_lagrangian}
    \fx{\mathcal{L}}{\mb{x},\mb{u},\msb{\mu}}
    \teq
    \fx{f}{\mb{x},\mb{u}}-\langle\msb{\mu},\mb{Kx}+\mb{Bu}-\mb{b}\rangle \, ,
\end{equation}
where $\msb{\mu}$ is the Lagrange multiplier of the equality constraint \cite[p. 237]{boyd:04:co}. In other words, $\left(\hat{\mb{x}},\hat{\mb{u}},\hat{\msb{\mu}}\right)$ solves the minimax problem:
\begin{equation} \label{eq:nien-16-rla-supp:saddle_point_problem}
    \left(\hat{\mb{x}},\hat{\mb{u}},\hat{\msb{\mu}}\right)
    \in
    \nbargminmax{\mb{x},\mb{u}}{\msb{\mu}}{\fx{\mathcal{L}}{\mb{x},\mb{u},\msb{\mu}}} \, .
\end{equation}
Moreover, since $\left(\hat{\mb{x}},\hat{\mb{u}},\hat{\msb{\mu}}\right)$ is a saddle-point of $\mathcal{L}$, the following inequalities
\begin{equation} \label{eq:nien-16-rla-supp:saddle_point_cond}
    \fx{\mathcal{L}}{\mb{x},\mb{u},\hat{\msb{\mu}}}\geq\fx{\mathcal{L}}{\hat{\mb{x}},\hat{\mb{u}},\hat{\msb{\mu}}}\geq\fx{\mathcal{L}}{\hat{\mb{x}},\hat{\mb{u}},\msb{\mu}}
\end{equation}
hold for any $\mb{x}$, $\mb{u}$, and $\msb{\mu}$, and the duality gap function:
\begin{equation} \label{eq:nien-16-rla-supp:def_duality_gap}
    \fx{\mathcal{G}}{\mb{x},\mb{u},\msb{\mu};\hat{\mb{x}},\hat{\mb{u}},\hat{\msb{\mu}}}
    \teq
    \fx{\mathcal{L}}{\mb{x},\mb{u},\hat{\msb{\mu}}}-\fx{\mathcal{L}}{\hat{\mb{x}},\hat{\mb{u}},\msb{\mu}}
    =
    \big[\fx{f}{\mb{x},\mb{u}}-\fx{f}{\hat{\mb{x}},\hat{\mb{u}}}\big]-\langle\hat{\msb{\mu}},\mb{Kx}+\mb{Bu}-\mb{b}\rangle
    \geq
    0
\end{equation}
characterizes the accuracy of an approximate solution $\left(\mb{x},\mb{u},\msb{\mu}\right)$ to the saddle-point problem \eqref{eq:nien-16-rla-supp:saddle_point_problem}. Note that $\mb{K}\hat{\mb{x}}+\mb{B}\hat{\mb{u}}-\mb{b}=\mb{0}$ due to the equality constraint. We consider the following (generalized alternating direction method of multipliers [ADMM]) iteration:
\begin{equation} \label{eq:nien-16-rla-supp:comp_admm}
    \begin{cases}
    \iter{\mb{x}}{k+1}\in\argmin{\mb{x}}{\fx{\phi}{\mb{x}}+\langle\fx{\nabla\psi}{\iter{\mb{x}}{k}},\mb{x}\rangle+\tfrac{1}{2}\norm{\mb{x}-\iter{\mb{x}}{k}}{\mb{D}_{\psi}}^2-\langle\iter{\msb{\mu}}{k},\mb{Kx}\rangle+\tfrac{\rho}{2}\norm{\mb{Kx}+\mb{B}\iter{\mb{u}}{k}-\mb{b}}{2}^2+\tfrac{1}{2}\norm{\mb{x}-\iter{\mb{x}}{k}}{\mb{P}}^2} \\
    \iter{\mb{u}}{k+1}\in\argmin{\mb{u}}{\fx{g}{\mb{u}}-\langle\iter{\msb{\mu}}{k},\mb{Bu}\rangle+\tfrac{\rho}{2}\norm{\alpha\mb{K}\iter{\mb{x}}{k+1}+\left(1-\alpha\right)\left(\mb{b}-\mb{B}\iter{\mb{u}}{k}\right)+\mb{Bu}-\mb{b}}{2}^2} \\
    \iter{\msb{\mu}}{k+1}=\iter{\msb{\mu}}{k}-\rho\left(\alpha\mb{K}\iter{\mb{x}}{k+1}+\left(1-\alpha\right)\left(\mb{b}-\mb{B}\iter{\mb{u}}{k}\right)+\mb{B}\iter{\mb{u}}{k+1}-\mb{b}\right)
    \end{cases}
\end{equation}
and show that the duality gap of the time-averaged solution $\mb{w}_K=\left(\mb{x}_K,\mb{u}_K,\msb{\mu}_K\right)$ it generates converges to zero at rate $\fx{\mathcal{O}}{1/K}$, where $K$ is the number of iterations,
\begin{equation} \label{eq:nien-16-rla-supp:def_time_avg}
    \mb{c}_K\teq\tfrac{1}{K}\ts\sum_{k=1}^{K}\iter{\mb{c}}{k}
\end{equation}
denotes the time-average of some iterate $\iter{\mb{c}}{k}$ for $k=1,\ldots,K$, $\rho>0$ is the corresponding AL penalty parameter, $\mb{P}\succeq0$ is a positive semi-definite weighting matrix, and $0<\alpha<2$ is the relaxation parameter.

\subsection{Preliminaries} \label{subsec:nien-16-rla-supp:preliminaries}
The convergence rate analysis of the iteration \eqref{eq:nien-16-rla-supp:comp_admm} is inspired by previous work \cite{he:12:oto,ouyang:13:sad,zhong:14:fsa,ouyang:15:aal,azadi:14:tao,fang:15:gad}. For simplicity, we use the following notations:
\begin{equation} \label{eq:nien-16-rla-supp:def_vec}
    \mb{w}
    \teq
    \begin{bmatrix}
    \mb{x} \\
    \mb{u} \\
    \msb{\mu}
    \end{bmatrix} \, ,
    \,
    \underline{\mb{w}}
    \teq
    \begin{bmatrix}
    \mb{x} \\
    \mb{u} \\
    \msb{\lambda}
    \end{bmatrix} \, ,
    \,
    \iter{\msb{\lambda}}{k+1}
    \teq
    \iter{\msb{\mu}}{k}-\rho\big(\mb{K}\iter{\mb{x}}{k+1}+\mb{B}\iter{\mb{u}}{k}-\mb{b}\big) \, ,
    \text{ and }
    \fx{F}{\underline{\mb{w}}}
    \teq
    \begin{bmatrix}
    -\mb{K}'\msb{\lambda} \\
    -\mb{B}'\msb{\lambda} \\
    \mb{Kx}+\mb{Bu}-\mb{b}
    \end{bmatrix} \, .
\end{equation}
We also introduce three matrices:
\begin{equation} \label{eq:nien-16-rla-supp:def_mtx}
    \mb{H}
    \teq
    \begin{bmatrix}
    \mb{D}_{\psi}+\mb{P} & 0 & 0 \\
    0 & \tfrac{\rho}{\alpha}\mb{B}'\mb{B} & \tfrac{1-\alpha}{\alpha}\mb{B}' \\
    0 & \tfrac{1-\alpha}{\alpha}\mb{B} & \tfrac{1}{\alpha\rho}\mb{I}
    \end{bmatrix} \, ,
    \,
    \mb{M}
    \teq
    \begin{bmatrix}
    \mb{I} & 0 & 0 \\
    0 & \mb{I} & 0 \\
    0 & -\rho\mb{B} & \alpha\mb{I}
    \end{bmatrix} \, ,
    \text{ and }
    \mb{Q}
    \teq
    \mb{HM}
    =
    \begin{bmatrix}
    \mb{D}_{\psi}+\mb{P} & 0 & 0 \\
    0 & \rho\mb{B}'\mb{B} & \left(1-\alpha\right)\mb{B}' \\
    0 & -\mb{B} & \tfrac{1}{\rho}\mb{I}
    \end{bmatrix} \, .
\end{equation}
The following lemmas show the properties of vectors and matrices defined in \eqref{eq:nien-16-rla-supp:def_vec} and \eqref{eq:nien-16-rla-supp:def_mtx} and an identity used in our derivation.
\begin{lemma} \label{lma:nien-16-rla-supp:psd_H}
The matrix $\mb{H}$ defined in \eqref{eq:nien-16-rla-supp:def_mtx} is positive semi-definite for any $0<\alpha<2$ and $\rho>0$.
\end{lemma}
\begin{proof} \label{pf:nien-16-rla-supp:proof_psd_H}
For any $\mb{w}$, completing the square yields
\begin{align}
    \mb{w}'\mb{H}\mb{w}
    &=
    \mb{x}'\left(\mb{D}_{\psi}+\mb{P}\right)\mb{x}+\tfrac{\rho}{\alpha}\mb{u}'\mb{B}'\mb{Bu}+\tfrac{2\left(1-\alpha\right)}{\alpha}\mb{u}'\mb{B}'\msb{\mu}+\tfrac{1}{\alpha\rho}\msb{\mu}'\msb{\mu} \nonumber \\
    &=
    \norm{\mb{x}}{\mb{D}_{\psi}+\mb{P}}^2+\tfrac{1}{\alpha}\big(\norm{\sqrt{\rho}\mb{Bu}}{2}^2+2\cdot\fx{\mathsf{sgn}}{1-\alpha}\left|1-\alpha\right|\left(\sqrt{\rho}\mb{Bu}\right)'\big(\tfrac{1}{\sqrt{\rho}}\msb{\mu}\big)+\bnorm{\tfrac{1}{\sqrt{\rho}}\msb{\mu}}{2}^2\big) \nonumber \\
    &=
    \norm{\mb{x}}{\mb{D}_{\psi}+\mb{P}}^2+\tfrac{1}{\alpha}\big(\left|1-\alpha\right|\bnorm{\sqrt{\rho}\mb{Bu}+\fx{\mathsf{sgn}}{1-\alpha}\tfrac{1}{\sqrt{\rho}}\msb{\mu}}{2}^2+\left(1-\left|1-\alpha\right|\right)\big(\norm{\sqrt{\rho}\mb{Bu}}{2}^2+\bnorm{\tfrac{1}{\sqrt{\rho}}\msb{\mu}}{2}^2\big)\big) \label{eq:nien-16-rla-supp:wHw_l3} \, .
\end{align}
All terms in \eqref{eq:nien-16-rla-supp:wHw_l3} are non-negative for any $0<\alpha<2$ and $\rho>0$. Thus under such conditions, $\mb{w}'\mb{H}\mb{w}\geq0$ for any $\mb{w}$, and $\mb{H}$ is positive semi-definite.
\end{proof}
\begin{lemma} \label{lma:nien-16-rla-supp:M_transition}
For any $k\geq0$, we have $\iter{\mb{w}}{k}-\iter{\mb{w}}{k+1}=\mb{M}\left(\iter{\mb{w}}{k}-\iter{\underline{\mb{w}}}{k+1}\right)$.
\end{lemma}
\begin{proof} \label{pf:nien-16-rla-supp:proof_M_transition}
Since two stacked vectors ($\mb{x}$ and $\mb{u}$) of $\mb{w}$ and $\underline{\mb{w}}$ are the same, we need only show that $\iter{\msb{\mu}}{k}-\iter{\msb{\mu}}{k+1}$ is equal to $\alpha\big(\iter{\msb{\mu}}{k}-\iter{\msb{\lambda}}{k+1}\big)-\rho\mb{B}\big(\iter{\mb{u}}{k}-\iter{\mb{u}}{k+1}\big)$ for any $k\geq0$. By the definition of $\iter{\msb{\lambda}}{k+1}$ in \eqref{eq:nien-16-rla-supp:def_vec}, we have
\begin{equation} \label{eq:nien-16-rla-supp:lambda_minus_mu}
    \iter{\msb{\mu}}{k}-\iter{\msb{\lambda}}{k+1}
    =
    \rho\big(\mb{K}\iter{\mb{x}}{k+1}+\mb{B}\iter{\mb{u}}{k}-\mb{b}\big) \, .
\end{equation}
Then, by the definition of the $\msb{\mu}$-update in \eqref{eq:nien-16-rla-supp:comp_admm}, we get
\begin{align}
    \iter{\msb{\mu}}{k}-\iter{\msb{\mu}}{k+1}
    &=
    \rho\big(\alpha\mb{K}\iter{\mb{x}}{k+1}+\left(1-\alpha\right)\big(\mb{b}-\mb{B}\iter{\mb{u}}{k}\big)+\mb{B}\iter{\mb{u}}{k+1}-\mb{b}\big) \nonumber \\
    &=
    \rho\big(\alpha\big(\mb{K}\iter{\mb{x}}{k+1}+\mb{B}\iter{\mb{u}}{k}-\mb{b}\big)+\mb{B}\big(\iter{\mb{u}}{k+1}-\iter{\mb{u}}{k}\big)\big) \nonumber \\
    &=
    \alpha\big(\rho\big(\mb{K}\iter{\mb{x}}{k+1}+\mb{B}\iter{\mb{u}}{k}-\mb{b}\big)\big)-\rho\mb{B}\big(\iter{\mb{u}}{k}-\iter{\mb{u}}{k+1}\big) \nonumber \\
    &=
    \alpha\big(\iter{\msb{\mu}}{k}-\iter{\msb{\lambda}}{k+1}\big)-\rho\mb{B}\big(\iter{\mb{u}}{k}-\iter{\mb{u}}{k+1}\big) \, . \label{eq:nien-16-rla-supp:M_transition_l4}
\end{align}
Thus the lemma holds.
\end{proof}
\begin{lemma} \label{lma:nien-16-rla-supp:inner_prod_lemma}
For any positive semi-definite matrix $\mb{M}$ and vectors $\mb{x}_1$, $\mb{x}_2$, $\mb{x}_3$, and $\mb{x}_4$, we have
\begin{equation} \label{eq:nien-16-rla-supp:inner_prod_lemma}
    \left(\mb{x}_1-\mb{x}_2\right)'\mb{M}\left(\mb{x}_3-\mb{x}_4\right)=\tfrac{1}{2}\norm{\mb{x}_1-\mb{x}_4}{\mb{M}}^2-\tfrac{1}{2}\norm{\mb{x}_1-\mb{x}_3}{\mb{M}}^2+\tfrac{1}{2}\norm{\mb{x}_2-\mb{x}_3}{\mb{M}}^2-\tfrac{1}{2}\norm{\mb{x}_2-\mb{x}_4}{\mb{M}}^2 \, .
\end{equation}
\end{lemma}
\begin{proof}
The proof is omitted here. It can be verified by expanding out all the inner product and norms on both sides.
\end{proof}

\subsection{Main results} \label{subsec:nien-16-rla-supp:main_results}
In the following theorem, we show that the duality gap defined in \eqref{eq:nien-16-rla-supp:def_duality_gap} of the time-averaged iterates $\mb{w}_K=\left(\mb{x}_K,\mb{u}_K,\msb{\mu}_K\right)$ in \eqref{eq:nien-16-rla-supp:comp_admm} converges at rate $\fx{\mathcal{O}}{1/K}$, where $K$ denotes the number of iterations.
\begin{theorem} \label{thm:nien-16-rla-supp:conv_rate_comp_admm}
Let $\mb{w}_K=\left(\mb{x}_K,\mb{u}_K,\msb{\mu}_K\right)$ be the time-averages of iterates in \eqref{eq:nien-16-rla-supp:comp_admm} where $\rho>0$, $0<\alpha<2$, and $\mb{P}$ is positive semi-definite. We have
\begin{multline} \label{eq:nien-16-rla-supp:conv_rate_comp_admm}
    \bfx{\mathcal{G}}{\mb{w}_K;\hat{\mb{w}}}=\big[\bfx{f}{\mb{x}_K,\mb{u}_K}-\fx{f}{\hat{\mb{x}},\hat{\mb{u}}}\big]-\langle\hat{\msb{\mu}},\mb{K}\mb{x}_K+\mb{B}\mb{u}_K-\mb{b}\rangle \\
    \leq
    \frac{1}{K}
    \left\{
    \tfrac{1}{2}\bnorm{\iter{\mb{x}}{0}-\hat{\mb{x}}}{\mb{D}_{\psi}}^2
    +
    \tfrac{1}{2}\bnorm{\iter{\mb{x}}{0}-\hat{\mb{x}}}{\mb{P}}^2
    +
    \tfrac{1}{2\alpha}\left[\sqrt{\rho}\bnorm{\mb{B}\big(\iter{\mb{u}}{0}-\hat{\mb{u}}\big)}{2}+\tfrac{1}{\sqrt{\rho}}\bnorm{\iter{\msb{\mu}}{0}-\hat{\msb{\mu}}}{2}\right]^2
    \right\} \, .
\end{multline}
\end{theorem}
\begin{proof} \label{pf:nien-16-rla-supp:proof_main}
We first focus on the $\mb{x}$-update in \eqref{eq:nien-16-rla-supp:comp_admm}. By the convexity of $\psi$, we have
\begin{align}
    \bfx{\psi}{\iter{\mb{x}}{k+1}}
    &\leq
    \bfx{\psi}{\iter{\mb{x}}{k}}+\langle\bfx{\nabla\psi}{\iter{\mb{x}}{k}},\iter{\mb{x}}{k+1}-\iter{\mb{x}}{k}\rangle+\tfrac{1}{2}\bnorm{\iter{\mb{x}}{k+1}-\iter{\mb{x}}{k}}{\mb{D}_{\psi}}^2 \nonumber \\
    &=
    \bfx{\psi}{\iter{\mb{x}}{k}}+\langle\bfx{\nabla\psi}{\iter{\mb{x}}{k}},\mb{x}-\iter{\mb{x}}{k}\rangle+\langle\bfx{\nabla\psi}{\iter{\mb{x}}{k}},\iter{\mb{x}}{k+1}-\mb{x}\rangle+\tfrac{1}{2}\bnorm{\iter{\mb{x}}{k+1}-\iter{\mb{x}}{k}}{\mb{D}_{\psi}}^2 \nonumber \\
    &\leq
    \fx{\psi}{\mb{x}}+\langle\bfx{\nabla\psi}{\iter{\mb{x}}{k}},\iter{\mb{x}}{k+1}-\mb{x}\rangle+\tfrac{1}{2}\bnorm{\iter{\mb{x}}{k+1}-\iter{\mb{x}}{k}}{\mb{D}_{\psi}}^2 \label{eq:nien-16-rla-supp:cvx_psi}
\end{align}
for any $\mb{x}$. Moving $\fx{\psi}{\mb{x}}$ to the left-hand side leads to
\begin{equation} \label{eq:nien-16-rla-supp:psi_ineq}
    \bfx{\psi}{\iter{\mb{x}}{k+1}}-\fx{\psi}{\mb{x}}
    \leq
    \langle\bfx{\nabla\psi}{\iter{\mb{x}}{k}},\iter{\mb{x}}{k+1}-\mb{x}\rangle+\tfrac{1}{2}\bnorm{\iter{\mb{x}}{k+1}-\iter{\mb{x}}{k}}{\mb{D}_{\psi}}^2 \, .
\end{equation}
Moreover, by the optimality condition of the $\mb{x}$-update in \eqref{eq:nien-16-rla-supp:comp_admm}, we have
\begin{equation} \label{eq:nien-16-rla-supp:opt_cond_x_update}
    \bfx{\partial\phi}{\iter{\mb{x}}{k+1}}+\bfx{\nabla\psi}{\iter{\mb{x}}{k}}+\mb{D}_{\psi}\big(\iter{\mb{x}}{k+1}-\iter{\mb{x}}{k}\big)-\mb{K}'\big(\iter{\msb{\mu}}{k}-\rho\big(\mb{K}\iter{\mb{x}}{k+1}+\mb{B}\iter{\mb{u}}{k}-\mb{b}\big)\big)+\mb{P}\big(\iter{\mb{x}}{k+1}-\iter{\mb{x}}{k}\big)\ni\mb{0} \, ,
\end{equation}
so
\begin{equation} \label{eq:nien-16-rla-supp:sub_grad_phi}
    \bfx{\partial\phi}{\iter{\mb{x}}{k+1}}
    \ni
    -\bfx{\nabla\psi}{\iter{\mb{x}}{k}}-\mb{D}_{\psi}\big(\iter{\mb{x}}{k+1}-\iter{\mb{x}}{k}\big)+\mb{K}'\iter{\msb{\lambda}}{k+1}-\mb{P}\big(\iter{\mb{x}}{k+1}-\iter{\mb{x}}{k}\big) \, .
\end{equation}
By the definition of subgradient for the convex function $\phi$, it follows that
\begin{align}
    \fx{\phi}{\mb{x}}
    &\geq
    \bfx{\phi}{\iter{\mb{x}}{k+1}}
    +
    \langle\bfx{\partial\phi}{\iter{\mb{x}}{k+1}},\mb{x}-\iter{\mb{x}}{k+1}\rangle \nonumber \\
    &=
    \bfx{\phi}{\iter{\mb{x}}{k+1}}
    +
    \langle\iter{\mb{x}}{k+1}-\mb{x},-\mb{K}'\iter{\msb{\lambda}}{k+1}\rangle
    +
    \langle\bfx{\nabla\psi}{\iter{\mb{x}}{k}},\iter{\mb{x}}{k+1}-\mb{x}\rangle
    +
    \langle\iter{\mb{x}}{k+1}-\mb{x},\left(\mb{D}_{\psi}+\mb{P}\right)\big(\iter{\mb{x}}{k+1}-\iter{\mb{x}}{k}\big)\rangle \label{eq:nien-16-rla-supp:first_order_cond_phi}
\end{align}
for all $\mb{x}$. Rearranging \eqref{eq:nien-16-rla-supp:first_order_cond_phi} leads to
\begin{multline} \label{eq:nien-16-rla-supp:phi_ineq}
    \big[\bfx{\phi}{\iter{\mb{x}}{k+1}}-\fx{\phi}{\mb{x}}\big]+\langle\iter{\mb{x}}{k+1}-\mb{x},-\mb{K}'\iter{\msb{\lambda}}{k+1}\rangle \\
    \leq
    -\langle\bfx{\nabla\psi}{\iter{\mb{x}}{k}},\iter{\mb{x}}{k+1}-\mb{x}\rangle
    +
    \langle\iter{\mb{x}}{k+1}-\mb{x},\left(\mb{D}_{\psi}+\mb{P}\right)\big(\iter{\mb{x}}{k}-\iter{\mb{x}}{k+1}\big)\rangle \, .
\end{multline}
Summing \eqref{eq:nien-16-rla-supp:psi_ineq} and \eqref{eq:nien-16-rla-supp:phi_ineq}, we get the first inequality:
\begin{equation} \label{eq:nien-16-rla-supp:x_update_ineq}
    \big[\bfx{h}{\iter{\mb{x}}{k+1}}-\fx{h}{\mb{x}}\big]+\langle\iter{\mb{x}}{k+1}-\mb{x},-\mb{K}'\iter{\msb{\lambda}}{k+1}\rangle
    \leq
    \langle\iter{\mb{x}}{k+1}-\mb{x},\left(\mb{D}_{\psi}+\mb{P}\right)\big(\iter{\mb{x}}{k}-\iter{\mb{x}}{k+1}\big)\rangle+\tfrac{1}{2}\bnorm{\iter{\mb{x}}{k+1}-\iter{\mb{x}}{k}}{\mb{D}_{\psi}}^2 \, .
\end{equation}

Following the same procedure, by the optimality condition of the $\mb{u}$-update in \eqref{eq:nien-16-rla-supp:comp_admm}, we have
\begin{equation} \label{eq:nien-16-rla-supp:first_order_g_ineq}
    \fx{g}{\mb{u}}
    \geq
    \bfx{g}{\iter{\mb{u}}{k+1}}
    +
    \langle\bfx{\partial g}{\iter{\mb{u}}{k+1}},\mb{u}-\iter{\mb{u}}{k+1}\rangle
    =
    \bfx{g}{\iter{\mb{u}}{k+1}}
    +
    \langle\iter{\mb{u}}{k+1}-\mb{u},-\mb{B}'\iter{\msb{\mu}}{k+1}\rangle
\end{equation}
for any $\mb{u}$. To substitute $\iter{\msb{\mu}}{k+1}$ in \eqref{eq:nien-16-rla-supp:first_order_g_ineq}, subtracting and adding $\iter{\msb{\lambda}}{k+1}$ on the left-hand side of \eqref{eq:nien-16-rla-supp:M_transition_l4} and rearranging it yield
\begin{equation} \label{eq:nien-16-rla-supp:lambda_eq}
    \iter{\msb{\mu}}{k+1}
    =
    \iter{\msb{\lambda}}{k+1}
    +
    \left(1-\alpha\right)\big(\iter{\msb{\mu}}{k}-\iter{\msb{\lambda}}{k+1}\big)
    +
    \rho\mb{B}\big(\iter{\mb{u}}{k}-\iter{\mb{u}}{k+1}\big) \, .
\end{equation}
Substituting \eqref{eq:nien-16-rla-supp:lambda_eq} into \eqref{eq:nien-16-rla-supp:first_order_g_ineq} and rearranging it, we get the second inequality:
\begin{equation} \label{eq:nien-16-rla-supp:u_update_ineq}
    \big[\bfx{g}{\iter{\mb{u}}{k+1}}-\fx{g}{\mb{u}}\big]+\langle\iter{\mb{u}}{k+1}-\mb{u},-\mb{B}'\iter{\msb{\lambda}}{k+1}\rangle
    \leq
    \langle\iter{\mb{u}}{k+1}-\mb{u},\rho\mb{B}'\mb{B}\big(\iter{\mb{u}}{k}-\iter{\mb{u}}{k+1}\big)+\left(1-\alpha\right)\mb{B}'\big(\iter{\msb{\mu}}{k}-\iter{\msb{\lambda}}{k+1}\big)\rangle \, .
\end{equation}
The third step differes a bit from the previous ones because the $\msb{\mu}$-update in \eqref{eq:nien-16-rla-supp:comp_admm} is not a minimization problem. By \eqref{eq:nien-16-rla-supp:lambda_minus_mu}, we have
\begin{equation}
    \mb{K}\iter{\mb{x}}{k+1}+\mb{B}\iter{\mb{u}}{k+1}-\mb{b}
    =
    -\mb{B}\big(\iter{\mb{u}}{k}-\iter{\mb{u}}{k+1}\big)+\tfrac{1}{\rho}\big(\iter{\msb{\mu}}{k}-\iter{\msb{\lambda}}{k+1}\big) \, .
\end{equation}
This gives the third equality:
\begin{equation} \label{eq:nien-16-rla-supp:lambda_update_eq}
    \langle\iter{\msb{\lambda}}{k+1}-\msb{\mu},\mb{K}\iter{\mb{x}}{k+1}+\mb{B}\iter{\mb{u}}{k+1}-\mb{b}\rangle
    =
    \langle\iter{\msb{\lambda}}{k+1}-\msb{\mu},-\mb{B}\big(\iter{\mb{u}}{k}-\iter{\mb{u}}{k+1}\big)+\tfrac{1}{\rho}\big(\iter{\msb{\mu}}{k}-\iter{\msb{\lambda}}{k+1}\big)\rangle
\end{equation}
for any $\msb{\mu}$. Summing \eqref{eq:nien-16-rla-supp:x_update_ineq}, \eqref{eq:nien-16-rla-supp:u_update_ineq}, and \eqref{eq:nien-16-rla-supp:lambda_update_eq}, we can write it compactly as
\begin{equation} \label{eq:nien-16-rla-supp:vi_ineq}
    \big[\bfx{f}{\iter{\mb{x}}{k+1},\iter{\mb{u}}{k+1}}-\fx{f}{\mb{x},\mb{u}}\big]
    +
    \langle\iter{\underline{\mb{w}}}{k+1}-\mb{w},\bfx{F}{\iter{\underline{\mb{w}}}{k+1}}\rangle
    \leq
    \langle\iter{\underline{\mb{w}}}{k+1}-\mb{w},\mb{Q}\big(\iter{\mb{w}}{k}-\iter{\underline{\mb{w}}}{k+1}\big)\rangle
    +
    \tfrac{1}{2}\bnorm{\iter{\mb{x}}{k+1}-\iter{\mb{x}}{k}}{\mb{D}_{\psi}}^2 \, .
\end{equation}
By Lemma~\ref{lma:nien-16-rla-supp:M_transition} (note that $\mb{Q}=\mb{HM}$) and Lemma~\ref{lma:nien-16-rla-supp:inner_prod_lemma}, the first term on the right-hand side of \eqref{eq:nien-16-rla-supp:vi_ineq} can be expressed as
\begin{multline} \label{eq:nien-16-rla-supp:rhs_identity}
    \langle\iter{\underline{\mb{w}}}{k+1}-\mb{w},\mb{H}\big(\iter{\mb{w}}{k}-\iter{\mb{w}}{k+1}\big)\rangle \\
    =
    \tfrac{1}{2}\bnorm{\iter{\underline{\mb{w}}}{k+1}-\iter{\mb{w}}{k+1}}{\mb{H}}^2-\tfrac{1}{2}\bnorm{\iter{\underline{\mb{w}}}{k+1}-\iter{\mb{w}}{k}}{\mb{H}}^2
    +
    \tfrac{1}{2}\bnorm{\iter{\mb{w}}{k}-\mb{w}}{\mb{H}}^2-\tfrac{1}{2}\bnorm{\iter{\mb{w}}{k+1}-\mb{w}}{\mb{H}}^2 \, .
\end{multline}
Moreover, the first term on the right-hand side of \eqref{eq:nien-16-rla-supp:rhs_identity} is
\begin{equation} \label{eq:nien-16-rla-supp:first_term_on_rhs}
    \tfrac{1}{\alpha\rho}\bnorm{\iter{\msb{\lambda}}{k+1}-\iter{\msb{\mu}}{k+1}}{2}^2
    =
    \tfrac{1}{\alpha\rho}\bnorm{\rho\mb{B}\big(\iter{\mb{u}}{k+1}-\iter{\mb{u}}{k}\big)+\left(1-\alpha\right)\big(\iter{\msb{\lambda}}{k+1}-\iter{\msb{\mu}}{k}\big)}{2}^2
\end{equation}
by \eqref{eq:nien-16-rla-supp:lambda_eq}, and the second term on the right-hand side of \eqref{eq:nien-16-rla-supp:rhs_identity} is
\begin{align}
    &\tfrac{1}{2}\bnorm{\iter{\mb{x}}{k+1}-\iter{\mb{x}}{k}}{\mb{D}_{\psi}+\mb{P}}^2 \nonumber \\
    &\quad+
    \tfrac{1}{\alpha\rho}\bnorm{\rho\mb{B}\big(\iter{\mb{u}}{k+1}-\iter{\mb{u}}{k}\big)}{2}^2
    +
    \tfrac{2\left(1-\alpha\right)}{\alpha\rho}\langle\rho\mb{B}\big(\iter{\mb{u}}{k+1}-\iter{\mb{u}}{k}\big),\iter{\msb{\lambda}}{k+1}-\iter{\msb{\mu}}{k}\rangle
    +
    \tfrac{1}{\alpha\rho}\bnorm{\iter{\msb{\lambda}}{k+1}-\iter{\msb{\mu}}{k}}{2}^2 \nonumber \\
    &\qquad=
    \tfrac{1}{2}\bnorm{\iter{\mb{x}}{k+1}-\iter{\mb{x}}{k}}{\mb{D}_{\psi}+\mb{P}}^2
    +
    \tfrac{1}{\alpha\rho}\bnorm{\rho\mb{B}\big(\iter{\mb{u}}{k+1}-\iter{\mb{u}}{k}\big)+\left(1-\alpha\right)\big(\iter{\msb{\lambda}}{k+1}-\iter{\msb{\mu}}{k}\big)}{2}^2
    +
    \tfrac{2-\alpha}{\rho}\bnorm{\iter{\msb{\lambda}}{k+1}-\iter{\msb{\mu}}{k}}{2}^2 \, . \label{eq:nien-16-rla-supp:second_term_on_rhs}
\end{align}
Substituting \eqref{eq:nien-16-rla-supp:first_term_on_rhs} and \eqref{eq:nien-16-rla-supp:second_term_on_rhs} into \eqref{eq:nien-16-rla-supp:rhs_identity}, we can upper bound the inequality \eqref{eq:nien-16-rla-supp:vi_ineq} by
\begin{align}
    &\big[\bfx{f}{\iter{\mb{x}}{k+1},\iter{\mb{u}}{k+1}}-\fx{f}{\mb{x},\mb{u}}\big]
    +
    \langle\iter{\underline{\mb{w}}}{k+1}-\mb{w},\bfx{F}{\iter{\underline{\mb{w}}}{k+1}}\rangle \nonumber \\
    &\qquad\leq
    \tfrac{1}{2}\bnorm{\iter{\mb{w}}{k}-\mb{w}}{\mb{H}}^2-\tfrac{1}{2}\bnorm{\iter{\mb{w}}{k+1}-\mb{w}}{\mb{H}}^2
    -
    \tfrac{1}{2}\bnorm{\iter{\mb{x}}{k+1}-\iter{\mb{x}}{k}}{\mb{D}_{\psi}+\mb{P}}^2
    -
    \tfrac{2-\alpha}{\rho}\bnorm{\iter{\msb{\lambda}}{k+1}-\iter{\msb{\mu}}{k}}{2}^2
    +
    \tfrac{1}{2}\bnorm{\iter{\mb{x}}{k+1}-\iter{\mb{x}}{k}}{\mb{D}_{\psi}}^2 \nonumber \\
    &\qquad\leq
    \tfrac{1}{2}\bnorm{\iter{\mb{w}}{k}-\mb{w}}{\mb{H}}^2-\tfrac{1}{2}\bnorm{\iter{\mb{w}}{k+1}-\mb{w}}{\mb{H}}^2
    -
    \tfrac{1}{2}\bnorm{\iter{\mb{x}}{k+1}-\iter{\mb{x}}{k}}{\mb{P}}^2
    -
    \tfrac{2-\alpha}{\rho}\bnorm{\iter{\msb{\lambda}}{k+1}-\iter{\msb{\mu}}{k}}{2}^2 \nonumber \\
    &\qquad\leq
    \tfrac{1}{2}\bnorm{\iter{\mb{w}}{k}-\mb{w}}{\mb{H}}^2-\tfrac{1}{2}\bnorm{\iter{\mb{w}}{k+1}-\mb{w}}{\mb{H}}^2 \label{eq:nien-16-rla-supp:vi_final}
\end{align}
because $\mb{P}$ is positive semi-definite and $2-\alpha>0$ for $\alpha\in(0,2)$.

To show the convergence rate of \eqref{eq:nien-16-rla-supp:comp_admm}, let $\left(\mb{x},\mb{u},\msb{\mu}\right)=\hat{\mb{w}}\teq\left(\hat{\mb{x}},\hat{\mb{u}},\hat{\msb{\mu}}\right)$. The last term on the left-hand side of \eqref{eq:nien-16-rla-supp:vi_final} can be represented as
\begin{align}
    &
    \langle\iter{\underline{\mb{w}}}{k+1}-\hat{\mb{w}},\bfx{F}{\iter{\underline{\mb{w}}}{k+1}}\rangle \nonumber \\
    &\qquad=
    \langle\iter{\mb{x}}{k+1}-\hat{\mb{x}},-\mb{K}'\iter{\msb{\lambda}}{k+1}\rangle
    +
    \langle\iter{\mb{u}}{k+1}-\hat{\mb{u}},-\mb{B}'\iter{\msb{\lambda}}{k+1}\rangle
    +
    \langle\iter{\msb{\lambda}}{k+1}-\hat{\msb{\mu}},\mb{K}\iter{\mb{x}}{k+1}+\mb{B}\iter{\mb{u}}{k+1}-\mb{b}\rangle \nonumber \\
    &\qquad=
    \langle\iter{\msb{\lambda}}{k+1},\mb{K}\hat{\mb{x}}-\mb{K}\iter{\mb{x}}{k+1}+\mb{B}\hat{\mb{u}}-\mb{B}\iter{\mb{u}}{k+1}+\mb{K}\iter{\mb{x}}{k+1}+\mb{B}\iter{\mb{u}}{k+1}-\mb{b}\rangle
    -
    \langle\hat{\msb{\mu}},\mb{K}\iter{\mb{x}}{k+1}+\mb{B}\iter{\mb{u}}{k+1}-\mb{b}\rangle \nonumber \\
    &\qquad=
    \langle\iter{\msb{\lambda}}{k+1},\mb{K}\hat{\mb{x}}+\mb{B}\hat{\mb{u}}-\mb{b}\rangle
    -
    \langle\hat{\msb{\mu}},\mb{K}\iter{\mb{x}}{k+1}+\mb{B}\iter{\mb{u}}{k+1}-\mb{b}\rangle \nonumber \\
    &\qquad=
    -\langle\hat{\msb{\mu}},\mb{K}\iter{\mb{x}}{k+1}+\mb{B}\iter{\mb{u}}{k+1}-\mb{b}\rangle \, . \label{eq:nien-16-rla-supp:lagrange_multiplier_term}
\end{align}
Note that $\mb{K}\hat{\mb{x}}+\mb{B}\hat{\mb{u}}-\mb{b}=\mb{0}$ due to the equality constraint. Using \eqref{eq:nien-16-rla-supp:lagrange_multiplier_term} yields
\begin{equation} \label{eq:nien-16-rla-supp:gap_at_k_plus_one}
    \bfx{\mathcal{G}}{\iter{\mb{w}}{k+1};\hat{\mb{w}}}
    =
    \big[\bfx{f}{\iter{\mb{x}}{k+1},\iter{\mb{u}}{k+1}}-\fx{f}{\hat{\mb{x}},\hat{\mb{u}}}\big]+\langle\iter{\underline{\mb{w}}}{k+1}-\hat{\mb{w}},\bfx{F}{\iter{\underline{\mb{w}}}{k+1}}\rangle
    \leq
    \tfrac{1}{2}\bnorm{\iter{\mb{w}}{k}-\hat{\mb{w}}}{\mb{H}}^2-\tfrac{1}{2}\bnorm{\iter{\mb{w}}{k+1}-\hat{\mb{w}}}{\mb{H}}^2 \, .
\end{equation}
Summing \eqref{eq:nien-16-rla-supp:gap_at_k_plus_one} from $k=0,\ldots,K-1$, dividing both sides by $K$, and applying Jensen's inequality to the convex function $f$, we have
\begin{multline} \label{eq:nien-16-rla-supp:gap_conv_rate}
    \bfx{\mathcal{G}}{\mb{w}_K;\hat{\mb{w}}}=\big[\bfx{f}{\mb{x}_K,\mb{u}_K}-\fx{f}{\hat{\mb{x}},\hat{\mb{u}}}\big]-\langle\hat{\msb{\mu}},\mb{K}\mb{x}_K+\mb{B}\mb{u}_K-\mb{b}\rangle \\
    \leq
    \tfrac{1}{K}\big(\tfrac{1}{2}\bnorm{\iter{\mb{w}}{0}-\hat{\mb{w}}}{\mb{H}}^2-\tfrac{1}{2}\bnorm{\iter{\mb{w}}{K}-\hat{\mb{w}}}{\mb{H}}^2\big)
    \leq
    \tfrac{1}{K}\cdot\tfrac{1}{2}\bnorm{\iter{\mb{w}}{0}-\hat{\mb{w}}}{\mb{H}}^2    
\end{multline}
since $\mb{H}$ is positive semi-definite for any $\alpha\in(0,2)$ and $\rho>0$ (Lemma~\ref{lma:nien-16-rla-supp:psd_H}). To finish the analysis, the remaining task is to upper bound $\tfrac{1}{2}\bnorm{\iter{\mb{w}}{0}-\hat{\mb{w}}}{\mb{H}}^2$. Note that $\tfrac{1}{2}\bnorm{\iter{\mb{w}}{0}-\hat{\mb{w}}}{\mb{H}}^2$ can be expressed as
\begin{equation} \label{eq:nien-16-rla-supp:norm_upper_bd_1}
    \tfrac{1}{2}\bnorm{\iter{\mb{x}}{0}-\hat{\mb{x}}}{\mb{D}_{\psi}}^2
    +
    \tfrac{1}{2}\bnorm{\iter{\mb{x}}{0}-\hat{\mb{x}}}{\mb{P}}^2
    +
    \tfrac{1}{2\alpha}
    \begin{bmatrix}
    \iter{\mb{u}}{0}-\hat{\mb{u}} \\
    \iter{\msb{\mu}}{0}-\hat{\msb{\mu}}
    \end{bmatrix}'
    \begin{bmatrix}
    \rho\mb{B}'\mb{B} & \left(1-\alpha\right)\mb{B}' \\
    \left(1-\alpha\right)\mb{B} & \tfrac{1}{\rho}\mb{I}
    \end{bmatrix}
    \begin{bmatrix}
    \iter{\mb{u}}{0}-\hat{\mb{u}} \\
    \iter{\msb{\mu}}{0}-\hat{\msb{\mu}}
    \end{bmatrix} \, .
\end{equation}
The last term in \eqref{eq:nien-16-rla-supp:norm_upper_bd_1} can be further expressed as and upper bounded by
\begin{align}
    &\,\,\,\,\,\,\,\,
    \tfrac{1}{2\alpha}
    \left[
    \rho\bnorm{\mb{B}\big(\iter{\mb{u}}{0}-\hat{\mb{u}}\big)}{2}^2
    +
    2\left(1-\alpha\right)\langle\mb{B}\big(\iter{\mb{u}}{0}-\hat{\mb{u}}\big),\iter{\msb{\mu}}{0}-\hat{\msb{\mu}}\rangle
    +
    \tfrac{1}{\rho}\bnorm{\iter{\msb{\mu}}{0}-\hat{\msb{\mu}}}{2}^2
    \right] \nonumber \\
    &\leq
    \tfrac{1}{2\alpha}
    \left[
    \rho\bnorm{\mb{B}\big(\iter{\mb{u}}{0}-\hat{\mb{u}}\big)}{2}^2
    +
    2\left|1-\alpha\right|\bnorm{\mb{B}\big(\iter{\mb{u}}{0}-\hat{\mb{u}}\big)}{2}\bnorm{\iter{\msb{\mu}}{0}-\hat{\msb{\mu}}}{2}
    +
    \tfrac{1}{\rho}\bnorm{\iter{\msb{\mu}}{0}-\hat{\msb{\mu}}}{2}^2
    \right] \nonumber \\
    &\leq
    \tfrac{1}{2\alpha}
    \left[
    \rho\bnorm{\mb{B}\big(\iter{\mb{u}}{0}-\hat{\mb{u}}\big)}{2}^2
    +
    2\bnorm{\mb{B}\big(\iter{\mb{u}}{0}-\hat{\mb{u}}\big)}{2}\bnorm{\iter{\msb{\mu}}{0}-\hat{\msb{\mu}}}{2}
    +
    \tfrac{1}{\rho}\bnorm{\iter{\msb{\mu}}{0}-\hat{\msb{\mu}}}{2}^2
    \right] \nonumber \\
    &=
    \tfrac{1}{2\alpha}
    \left[\sqrt{\rho}\bnorm{\mb{B}\big(\iter{\mb{u}}{0}-\hat{\mb{u}}\big)}{2}+\tfrac{1}{\sqrt{\rho}}\bnorm{\iter{\msb{\mu}}{0}-\hat{\msb{\mu}}}{2}\right]^2 \label{eq:nien-16-rla-supp:norm_upper_bd_2}
\end{align}
due to the fact that $0<\alpha<2$. Combining \eqref{eq:nien-16-rla-supp:gap_conv_rate}, \eqref{eq:nien-16-rla-supp:norm_upper_bd_1}, and \eqref{eq:nien-16-rla-supp:norm_upper_bd_2}, we get our final convergence rate bound:
\begin{multline} \label{eq:nien-16-rla-supp:conv_rate_comp_admm_in_proof}
    \bfx{\mathcal{G}}{\mb{w}_K;\hat{\mb{w}}}=\big[\bfx{f}{\mb{x}_K,\mb{u}_K}-\fx{f}{\hat{\mb{x}},\hat{\mb{u}}}\big]-\langle\hat{\msb{\mu}},\mb{K}\mb{x}_K+\mb{B}\mb{u}_K-\mb{b}\rangle \\
    \leq
    \frac{1}{K}
    \left\{
    \tfrac{1}{2}\bnorm{\iter{\mb{x}}{0}-\hat{\mb{x}}}{\mb{D}_{\psi}}^2
    +
    \tfrac{1}{2}\bnorm{\iter{\mb{x}}{0}-\hat{\mb{x}}}{\mb{P}}^2
    +
    \tfrac{1}{2\alpha}\left[\sqrt{\rho}\bnorm{\mb{B}\big(\iter{\mb{u}}{0}-\hat{\mb{u}}\big)}{2}+\tfrac{1}{\sqrt{\rho}}\bnorm{\iter{\msb{\mu}}{0}-\hat{\msb{\mu}}}{2}\right]^2
    \right\} \, .
\end{multline}
\end{proof}

Theorem~\ref{thm:nien-16-rla-supp:conv_rate_comp_admm} can be used to show the convergence rates of other AL-based algorithms. The following theorems show the convergence rates of the simple and the proposed relaxed LALM's in \cite{nien:16:rla}. From now on, suppose $\mb{A}$ is an $m\times n$ matrix, and let $\mb{G}\teq\mb{D}_{\mb{A}}-\mb{A}'\mb{A}$, where $\mb{D}_{\mb{A}}$ is a diagonal majorizing matrix of $\mb{A}'\mb{A}$.

\begin{theorem}[{\cite[Theorem~1]{nien:16:rla}}] \label{thm:nien-16-rla-supp:conv_rate_simple_relaxed_lalm}
Let $\mb{K}=\mb{A}$, $\mb{B}=-\mb{I}_m$, $\mb{b}=\mb{0}_m$, and $\mb{P}=\rho\mb{G}$. The iteration \eqref{eq:nien-16-rla-supp:comp_admm} with $\rho>0$ and $0<\alpha<2$ reduces to the simple relaxed LALM that achieves a convergence rate
\begin{equation} \label{eq:nien-16-rla-supp:conv_rate_simple_relax}
    \fx{\mathcal{G}}{\mb{w}_K;\hat{\mb{w}}}
    \leq
    \tfrac{1}{K}\left(A_{\mb{D}_{\psi}}+B_{\rho,\mb{D}_{\mb{A}}}+C_{\alpha,\rho}\right) \, ,
\end{equation}
where the first two constants
\begin{align}
    A_{\mb{D}_{\psi}}
    &\teq
    \tfrac{1}{2}\bnorm{\iter{\mb{x}}{0}-\hat{\mb{x}}}{\mb{D}_{\psi}}^2 \label{eq:nien-16-rla-supp:def_A} \\
    B_{\rho,\mb{D}_{\mb{A}}}
    &\teq
    \tfrac{\rho}{2}\bnorm{\iter{\mb{x}}{0}-\hat{\mb{x}}}{\mb{D}_{\mb{A}}-\mb{A}'\mb{A}}^2 \label{eq:nien-16-rla-supp:def_B}
\end{align}
depend on how far the initial guess is from a minimizer, and the last constant
\begin{equation} \label{eq:nien-16-rla-supp:def_C}
    C_{\alpha,\rho}
    \teq
    \tfrac{1}{2\alpha}\left[\sqrt{\rho}\bnorm{\iter{\mb{u}}{0}-\hat{\mb{u}}}{2}+\tfrac{1}{\sqrt{\rho}}\bnorm{\iter{\msb{\mu}}{0}-\hat{\msb{\mu}}}{2}\right]^2
\end{equation}
depends on the relaxation parameter.
\end{theorem}
\begin{proof} \label{pf:nien-16-rla-supp:simple_rlalm}
One just uses the substitutions $\mb{K}=\mb{A}$, $\mb{B}=-\mb{I}_m$, $\mb{b}=\mb{0}_m$, and $\mb{P}=\rho\mb{G}$ in Theorem~\ref{thm:nien-16-rla-supp:conv_rate_comp_admm} to prove the theorem.
\end{proof}
As seen in Theorem~\ref{thm:nien-16-rla-supp:conv_rate_simple_relaxed_lalm}, the convergence rate of the simple relaxed LALM scales well with the relaxation parameter $\alpha$ iff $C_{\alpha,\rho}\gg A_{\mb{D}_{\psi}}$ and $C_{\alpha,\rho}\gg B_{\rho,\mb{D}_{\mb{A}}}$. When $\psi$ has large curvature or $\mb{D}_{\mb{A}}$ is a loose majorizing matrix of $\mb{A}'\mb{A}$ (like in X-ray CT), the above inequalities do not hold, leading to worse scalability of convergence rate with the relaxation parameter $\alpha$. This motivated the proposed relaxed LALM \cite{nien:16:rla} whose convergence rate analysis is shown below.

\begin{theorem}[{\cite[Theorem~2]{nien:16:rla}}] \label{thm:nien-16-rla-supp:conv_rate_proposed_relaxed_lalm}
Let $\mb{K}=\left[\mb{A}'~\,\mb{G}^{1/2}\right]'$, $\mb{B}=-\mb{I}_{m+n}$, $\mb{b}=\mb{0}_{m+n}$, and $\mb{P}=0$. The iteration \eqref{eq:nien-16-rla-supp:comp_admm} with $\rho>0$ and $0<\alpha<2$ reduces to the proposed relaxed LALM \cite{nien:16:rla} that achieves a convergence rate
\begin{equation} \label{eq:nien-16-rla-supp:conv_rate_proposed_relax}
    \fx{\mathcal{G}'}{\mb{w}_K;\hat{\mb{w}}}
    \leq
    \tfrac{1}{K}\left(A_{\mb{D}_{\psi}}+\overline{B}_{\alpha,\rho,\mb{D}_{\mb{A}}}+C_{\alpha,\rho}\right) \, ,
\end{equation}
where $A_{\mb{D}_{\psi}}$ and $C_{\alpha,\rho}$ were defined in \eqref{eq:nien-16-rla-supp:def_A} and \eqref{eq:nien-16-rla-supp:def_C}, and
\begin{equation} \label{eq:nien-16-rla-supp:def_B_bar}
    \overline{B}_{\alpha,\rho,\mb{D}_{\mb{A}}}
    \teq
    \tfrac{\rho}{2\alpha}\bnorm{\iter{\mb{v}}{0}-\hat{\mb{v}}}{2}^2
    =
    \tfrac{\rho}{2\alpha}\bnorm{\iter{\mb{x}}{0}-\hat{\mb{x}}}{\mb{D}_{\mb{A}}-\mb{A}'\mb{A}}^2
\end{equation}
when initializing $\mb{v}$ and $\msb{\nu}$ as $\iter{\mb{v}}{0}=\mb{G}^{1/2}\iter{\mb{x}}{0}$ and $\iter{\msb{\nu}}{0}=\mb{0}_n$, respectively.
\end{theorem}
\begin{proof} \label{pf:nien-16-rla-supp:proposed_rlalm}
Applying the substitutions $\mb{K}=\left[\mb{A}'~\,\mb{G}^{1/2}\right]'$, $\mb{B}=-\mb{I}_{m+n}$, $\mb{b}=\mb{0}_{m+n}$, and $\mb{P}=0$ to Theorem~\ref{thm:nien-16-rla-supp:conv_rate_comp_admm}, except for the upper bounding \eqref{eq:nien-16-rla-supp:norm_upper_bd_2}, yields
\begin{equation} \label{eq:nien-16-rla-supp:conv_rate_proposed_lalm}
    \bfx{\mathcal{G}'}{\mb{w}_K;\hat{\mb{w}}}
    \leq
    \tfrac{1}{K}
    \big(
    A_{\mb{D}_{\psi}}
    +
    D_{\alpha,\rho}
    \big) \, ,
\end{equation}
where
\begin{equation} \label{eq:nien-16-rla-supp:def_D0}
    D_{\alpha,\rho}
    \teq
    \tfrac{1}{2\alpha}
    \begin{bmatrix}
    \iter{\mb{u}}{0}-\hat{\mb{u}} \\
    \iter{\mb{v}}{0}-\hat{\mb{v}} \\
    \iter{\msb{\mu}}{0}-\hat{\msb{\mu}} \\
    \iter{\msb{\nu}}{0}-\hat{\msb{\nu}}
    \end{bmatrix}'
    \begin{bmatrix}
    \rho\mb{I}_m & 0 & -\left(1-\alpha\right)\mb{I}_m & 0 \\
    0 & \rho\mb{I}_n & 0 & -\left(1-\alpha\right)\mb{I}_n \\
    -\left(1-\alpha\right)\mb{I}_m & 0 & \tfrac{1}{\rho}\mb{I}_m & 0 \\
    0 & -\left(1-\alpha\right)\mb{I}_n & 0 & \tfrac{1}{\rho}\mb{I}_n
    \end{bmatrix}
    \begin{bmatrix}
    \iter{\mb{u}}{0}-\hat{\mb{u}} \\
    \iter{\mb{v}}{0}-\hat{\mb{v}} \\
    \iter{\msb{\mu}}{0}-\hat{\msb{\mu}} \\
    \iter{\msb{\nu}}{0}-\hat{\msb{\nu}}
    \end{bmatrix} \, ,
\end{equation}
and $\mb{v}$ and $\msb{\nu}$ are the auxiliary variable and Lagrange multiplier of the additional redundant equality constraint $\mb{v}=\mb{G}^{1/2}\mb{x}$ in \cite{nien:16:rla}, respectively. Note that $\iter{\msb{\nu}}{k}=\mb{0}_n$ for $k=0,1,\ldots$ if we initialize $\msb{\nu}$ as $\iter{\msb{\nu}}{0}=\mb{0}_n$, and $\hat{\msb{\nu}}=\mb{0}_n$ \cite{nien:16:rla}. We have $\iter{\msb{\nu}}{0}-\hat{\msb{\nu}}=\mb{0}_n$. Hence, \eqref{eq:nien-16-rla-supp:def_D0} is further upper bounded by
\begin{align}
    D_{\alpha,\rho}
    &=
    \tfrac{1}{2\alpha}
    \begin{bmatrix}
    \iter{\mb{u}}{0}-\hat{\mb{u}} \\
    \iter{\mb{v}}{0}-\hat{\mb{v}} \\
    \iter{\msb{\mu}}{0}-\hat{\msb{\mu}}
    \end{bmatrix}'
    \begin{bmatrix}
    \rho\mb{I}_m & 0 & -\left(1-\alpha\right)\mb{I}_m \\
    0 & \rho\mb{I}_n & 0 \\
    -\left(1-\alpha\right)\mb{I}_m & 0 & \tfrac{1}{\rho}\mb{I}_m
    \end{bmatrix}
    \begin{bmatrix}
    \iter{\mb{u}}{0}-\hat{\mb{u}} \\
    \iter{\mb{v}}{0}-\hat{\mb{v}} \\
    \iter{\msb{\mu}}{0}-\hat{\msb{\mu}}
    \end{bmatrix} \nonumber \\
    &=
    \tfrac{1}{2\alpha}
    \left(
    \rho\bnorm{\iter{\mb{u}}{0}-\hat{\mb{u}}}{2}^2
    -
    2\left(1-\alpha\right)\langle\iter{\mb{u}}{0}-\hat{\mb{u}},\iter{\msb{\mu}}{0}-\hat{\msb{\mu}}\rangle
    +
    \tfrac{1}{\rho}\bnorm{\iter{\msb{\mu}}{0}-\hat{\msb{\mu}}}{2}^2
    \right)
    +
    \tfrac{\rho}{2\alpha}\bnorm{\iter{\mb{v}}{0}-\hat{\mb{v}}}{2}^2 \nonumber \\
    &\leq
    \tfrac{1}{2\alpha}\left[\sqrt{\rho}\bnorm{\iter{\mb{u}}{0}-\hat{\mb{u}}}{2}+\tfrac{1}{\sqrt{\rho}}\bnorm{\iter{\msb{\mu}}{0}-\hat{\msb{\mu}}}{2}\right]^2
    +
    \tfrac{\rho}{2\alpha}\bnorm{\mb{G}^{1/2}\iter{\mb{x}}{0}-\mb{G}^{1/2}\hat{\mb{x}}}{2}^2 \nonumber \\
    &=
    C_{\alpha,\rho}+\tfrac{\rho}{2\alpha}\bnorm{\iter{\mb{x}}{0}-\hat{\mb{x}}}{\mb{D}_{\mb{A}}-\mb{A}'\mb{A}}^2 \, .
\end{align}
Let
\begin{equation}
    \overline{B}_{\alpha,\rho,\mb{D}_{\mb{A}}}
    \teq
    \tfrac{\rho}{2\alpha}\bnorm{\iter{\mb{x}}{0}-\hat{\mb{x}}}{\mb{D}_{\mb{A}}-\mb{A}'\mb{A}}^2 \, .
\end{equation}
Thus, the convergence rate of the proposed relaxed LALM \cite{nien:16:rla} is upper bounded by
\begin{equation}
    \bfx{\mathcal{G}'}{\mb{w}_K;\hat{\mb{w}}}
    \leq
    \tfrac{1}{K}
    \big(
    A_{\mb{D}_{\psi}}
    +
    \overline{B}_{\alpha,\rho,\mb{D}_{\mb{A}}}
    +
    C_{\alpha,\rho}
    \big) \, .
\end{equation}
\end{proof}

\subsection{Practical implementation of the proposed relaxed LALM} \label{subsec:nien-16-rla-supp:derivation_rlalm}
Although the proposed relaxed LALM shows better scalability of the convergence rate with the relaxation parameter $\alpha$, a straightforward implementation with substitutions in Theorem~\ref{thm:nien-16-rla-supp:conv_rate_proposed_relaxed_lalm} is not recommended because there is no efficient way to compute the square root of $\mb{G}$ for any $\mb{A}$ in general. For practical implementation, we must avoid using multiplication by $\mb{G}^{1/2}$ in both the $\mb{x}$- and $\mb{v}$-updates. To derive the practical implementation, we first substitue $\mb{K}=\left[\mb{A}'~\,\mb{G}^{1/2}\right]'$, $\mb{B}=-\mb{I}$, $\mb{b}=\mb{0}$, and $\mb{P}=0$ in \eqref{eq:nien-16-rla-supp:comp_admm}. This leads to the following iterates (i.e., \cite[Eqn. 25]{nien:16:rla}):
\begin{equation} \label{eq:nien-16-rla-supp:relaxed_lal_iter0}
    \begin{cases}
    \iter{\mb{x}}{k+1}\in\argmin{\mb{x}}{\fx{\phi}{\mb{x}}+\bfx{Q_{\psi}}{\mb{x};\iter{\mb{x}}{k}}-\langle\iter{\msb{\mu}}{k},\mb{Ax}\rangle-\langle\iter{\msb{\nu}}{k},\mb{G}^{1/2}\mb{x}\rangle+\tfrac{\rho}{2}\bnorm{\mb{Ax}-\iter{\mb{u}}{k}}{2}^2+\tfrac{\rho}{2}\bnorm{\mb{G}^{1/2}\mb{x}-\iter{\mb{v}}{k}}{2}^2} \\
    \iter{\mb{u}}{k+1}\in\argmin{\mb{u}}{\fx{g}{\mb{u}}+\langle\iter{\msb{\mu}}{k},\mb{u}\rangle+\tfrac{\rho}{2}\bnorm{\iter{\mb{r}_{\mb{u},\alpha}}{k+1}-\mb{u}}{2}^2} \\
    \iter{\msb{\mu}}{k+1}=\iter{\msb{\mu}}{k}-\rho\big(\iter{\mb{r}_{\mb{u},\alpha}}{k+1}-\iter{\mb{u}}{k+1}\big) \\
    \iter{\mb{v}}{k+1}=\iter{\mb{r}_{\mb{v},\alpha}}{k+1}-\rho^{-1}\iter{\msb{\nu}}{k} \\
    \iter{\msb{\nu}}{k+1}=\iter{\msb{\nu}}{k}-\rho\big(\iter{\mb{r}_{\mb{v},\alpha}}{k+1}-\iter{\mb{v}}{k+1}\big) \, ,
    \end{cases}
\end{equation}
where $Q_{\psi}$ is a separable quadratic surrogate (SQS) of $\psi$ at $\iter{\mb{x}}{k}$ \cite[Eqn. 17]{nien:16:rla}, $\mb{r}_{\mb{u},\alpha}$ is the relaxation variable of $\mb{u}$, and $\mb{r}_{\mb{v},\alpha}$ is the relaxation variable of $\mb{v}$. Suppose $\iter{\msb{\nu}}{0}=\mb{0}$. Then $\iter{\msb{\nu}}{k}=\mb{0}$ for $k=0,1,\ldots$, and \eqref{eq:nien-16-rla-supp:relaxed_lal_iter0} can be further simplified as
\begin{equation} \label{eq:nien-16-rla-supp:relaxed_lal_iter1}
    \begin{cases}
    \iter{\mb{x}}{k+1}\in\argmin{\mb{x}}{\fx{\phi}{\mb{x}}+\bfx{Q_{\psi}}{\mb{x};\iter{\mb{x}}{k}}-\langle\iter{\msb{\mu}}{k},\mb{Ax}\rangle+\tfrac{\rho}{2}\bnorm{\mb{Ax}-\iter{\mb{u}}{k}}{2}^2+\tfrac{\rho}{2}\bnorm{\mb{G}^{1/2}\mb{x}-\iter{\mb{v}}{k}}{2}^2} \\
    \iter{\mb{u}}{k+1}\in\argmin{\mb{u}}{\fx{g}{\mb{u}}+\langle\iter{\msb{\mu}}{k},\mb{u}\rangle+\tfrac{\rho}{2}\bnorm{\iter{\mb{r}_{\mb{u},\alpha}}{k+1}-\mb{u}}{2}^2} \\
    \iter{\msb{\mu}}{k+1}=\iter{\msb{\mu}}{k}-\rho\big(\iter{\mb{r}_{\mb{u},\alpha}}{k+1}-\iter{\mb{u}}{k+1}\big) \\
    \iter{\mb{v}}{k+1}=\alpha\mb{G}^{1/2}\iter{\mb{x}}{k+1}+\left(1-\alpha\right)\iter{\mb{v}}{k} \, .
    \end{cases}
\end{equation}

Let $\mb{h}\teq\mb{G}^{1/2}\mb{v}+\mb{A}'\mb{y}$. By the $\mb{v}$-update in \eqref{eq:nien-16-rla-supp:relaxed_lal_iter1}, we have
\begin{align}
    \iter{\mb{h}}{k+1}
    &=
    \mb{G}^{1/2}\iter{\mb{v}}{k+1}+\mb{A}'\mb{y} \nonumber \\
    &=
    \mb{G}^{1/2}\big(\alpha\mb{G}^{1/2}\iter{\mb{x}}{k+1}+\left(1-\alpha\right)\iter{\mb{v}}{k}\big)+\mb{A}'\mb{y} \nonumber \\
    &=
    \alpha\big(\mb{G}\iter{\mb{x}}{k+1}+\mb{A}'\mb{y}\big)+\left(1-\alpha\right)\big(\mb{G}^{1/2}\iter{\mb{v}}{k}+\mb{A}'\mb{y}\big) \nonumber \\
    &=
    \alpha\big(\mb{D}_{\mb{A}}\iter{\mb{x}}{k+1}-\mb{A}'\big(\mb{A}\iter{\mb{x}}{k+1}-\mb{y}\big)\big)+\left(1-\alpha\right)\iter{\mb{h}}{k} \, . \label{eq:nien-16-rla-supp:h_update}
\end{align}
To avoid multiplication by $\mb{G}^{1/2}$ in the $\mb{x}$-update in \eqref{eq:nien-16-rla-supp:relaxed_lal_iter1}, we rewrite the last three terms in the $\mb{x}$-update cost function using Taylor's expansion around $\iter{\mb{x}}{k}$. That is,
\begin{align}
    &\,
    -\langle\iter{\msb{\mu}}{k},\mb{Ax}\rangle+\tfrac{\rho}{2}\bnorm{\mb{Ax}-\iter{\mb{u}}{k}}{2}^2+\tfrac{\rho}{2}\bnorm{\mb{G}^{1/2}\mb{x}-\iter{\mb{v}}{k}}{2}^2 \nonumber \\
    \propto&\,
    \tfrac{\rho}{2}\bnorm{\mb{Ax}-\iter{\mb{u}}{k}-\rho^{-1}\iter{\msb{\mu}}{k}}{2}^2+\tfrac{\rho}{2}\bnorm{\mb{G}^{1/2}\mb{x}-\iter{\mb{v}}{k}}{2}^2 \nonumber \\
    \propto&\,
    \big(\mb{x}-\iter{\mb{x}}{k}\big)'\big(\rho\mb{A}'\big(\mb{A}\iter{\mb{x}}{k}-\iter{\mb{u}}{k}-\rho^{-1}\iter{\msb{\mu}}{k}\big)+\rho\big(\mb{G}\iter{\mb{x}}{k}-\mb{G}^{1/2}\iter{\mb{v}}{k}\big)\big)+\tfrac{1}{2}\bnorm{\mb{x}-\iter{\mb{x}}{k}}{\rho\mb{A}'\mb{A}+\rho\mb{G}}^2 \nonumber \\
    =&\,
    \big(\mb{x}-\iter{\mb{x}}{k}\big)'\big(\rho\big(\mb{A}'\mb{A}+\mb{G}\big)\iter{\mb{x}}{k}-\rho\mb{A}'\big(\iter{\mb{u}}{k}+\rho^{-1}\iter{\msb{\mu}}{k}\big)-\rho\mb{G}^{1/2}\iter{\mb{v}}{k}\big)+\tfrac{1}{2}\bnorm{\mb{x}-\iter{\mb{x}}{k}}{\rho\left(\mb{A}'\mb{A}+\mb{G}\right)}^2 \nonumber \\
    =&\,
    \big(\mb{x}-\iter{\mb{x}}{k}\big)'\big(\rho\mb{D}_{\mb{A}}\iter{\mb{x}}{k}-\rho\mb{A}'\big(\iter{\mb{u}}{k}-\mb{y}+\rho^{-1}\iter{\msb{\mu}}{k}\big)-\rho\iter{\mb{h}}{k}\big)+\tfrac{1}{2}\bnorm{\mb{x}-\iter{\mb{x}}{k}}{\rho\mb{D}_{\mb{A}}}^2 \nonumber \\
    \propto&\,
    \tfrac{1}{2}\bnorm{\mb{x}-\iter{\mb{x}}{k}+\left(\rho\mb{D}_{\mb{A}}\right)^{-1}\big(\rho\mb{D}_{\mb{A}}\iter{\mb{x}}{k}-\rho\mb{A}'\big(\iter{\mb{u}}{k}-\mb{y}+\rho^{-1}\iter{\msb{\mu}}{k}\big)-\rho\iter{\mb{h}}{k}\big)}{\rho\mb{D}_{\mb{A}}}^2 \nonumber \\
    =&\,
    \tfrac{1}{2}\bnorm{\mb{x}-\left(\rho\mb{D}_{\mb{A}}\right)^{-1}\big(\rho\mb{A}'\big(\iter{\mb{u}}{k}-\mb{y}+\rho^{-1}\iter{\msb{\mu}}{k}\big)+\rho\iter{\mb{h}}{k}\big)}{\rho\mb{D}_{\mb{A}}}^2 \, . \label{eq:nien-16-rla-supp:x_update_prox_term}
\end{align}
The practical relaxed LALM without multiplication by $\mb{G}^{1/2}$ becomes
\begin{equation} \label{eq:nien-16-rla-supp:relaxed_lal_iter}
    \begin{cases}
    \iter{\mb{x}}{k+1}\in\argmin{\mb{x}}{\fx{\phi}{\mb{x}}+\bfx{Q_{\psi}}{\mb{x};\iter{\mb{x}}{k}}+\tfrac{1}{2}\bnorm{\mb{x}-\left(\rho\mb{D}_{\mb{A}}\right)^{-1}\big(\rho\mb{A}'\big(\iter{\mb{u}}{k}-\mb{y}+\rho^{-1}\iter{\msb{\mu}}{k}\big)+\rho\iter{\mb{h}}{k}\big)}{\rho\mb{D}_{\mb{A}}}^2} \\
    \iter{\mb{u}}{k+1}\in\argmin{\mb{u}}{\fx{g}{\mb{u}}+\langle\iter{\msb{\mu}}{k},\mb{u}\rangle+\tfrac{\rho}{2}\bnorm{\iter{\mb{r}_{\mb{u},\alpha}}{k+1}-\mb{u}}{2}^2} \\
    \iter{\msb{\mu}}{k+1}=\iter{\msb{\mu}}{k}-\rho\big(\iter{\mb{r}_{\mb{u},\alpha}}{k+1}-\iter{\mb{u}}{k+1}\big) \\
    \iter{\mb{h}}{k+1}=\alpha\big(\mb{D}_{\mb{A}}\iter{\mb{x}}{k+1}-\mb{A}'\big(\mb{A}\iter{\mb{x}}{k+1}-\mb{y}\big)\big)+\left(1-\alpha\right)\iter{\mb{h}}{k} \, .
    \end{cases}
\end{equation}

\subsection{Numerical example: LASSO regression} \label{subsec:nien-16-rla-supp:lasso}
Here we describe a numerical example that demonstrates the convergence rate bound and the effect of relaxation. Consider the following $\ell_1$-regularized linear regression problem:
\begin{equation} \label{eq:nien-16-rla-supp:lasso}
    \hat{\mb{x}}
    \in
    \argmin{\mb{x}}{\tfrac{1}{2}\norm{\mb{y}-\mb{Ax}}{2}^2+\lambda\norm{\mb{x}}{1}} \, ,
\end{equation}
where $\mb{A}\in\real^{m\times n}$, and $n\gg m$ in general. This is a widely studied problem in the field of statistics (also known as LASSO regression) and compressed sensing for seeking a sparse solution of a linear system with small measurement errors. To solve this problem using the proposed relaxed LALM \eqref{eq:nien-16-rla-supp:relaxed_lal_iter}, we focused on the following equivalent equality-constrained minimization problem:
\begin{equation} \label{eq:nien-16-rla-supp:lasso_eq}
    \left(\hat{\mb{x}},\hat{\mb{u}}\right)
    \in
    \argmin{\mb{x},\mb{u}}{\tfrac{1}{2}\norm{\mb{y}-\mb{u}}{2}^2+\lambda\norm{\mb{x}}{1}}\text{ s.t. }\mb{u}=\mb{Ax}
\end{equation}
with $\phi=\lambda\norm{\cdot}{1}$, $\psi=0$, $\mb{D}_{\mb{A}}=\fx{\lambda_{\text{max}}}{\mb{A}'\mb{A}}\mb{I}$, and $g=\tfrac{1}{2}\norm{\cdot-\mb{y}}{2}^2$. We set $\iter{\mb{x}}{0}=\mb{A}^{\dagger}\mb{y}$, $\iter{\mb{u}}{0}=\mb{A}\iter{\mb{x}}{0}$, $\iter{\msb{\mu}}{0}=\mb{y}-\iter{\mb{u}}{0}$, and $\iter{\mb{h}}{0}=\mb{D}_{\mb{A}}\iter{\mb{x}}{0}-\mb{A}'\big(\mb{A}\iter{\mb{x}}{0}-\mb{y}\big)$. Data for numerical instances were generated as follows. The entries of the system matrix $\mb{A}\in\real^{100\times 400}$ were sampled from an iid standard normal distribution. The hidden sparse vector $\mb{x}_s\in\real^{400}$ was a randomly generated $20$-sparse vector, and the noisy measurement $\mb{y}=\mb{A}\mb{x}_s+\mb{n}$, where $\mb{n}\in\real^{100}$ was sampled from an iid $\fx{\mathcal{N}}{0,0.1}$. The regularization parameter $\lambda$ was set to be unity in our experiment.

Figure~\ref{fig:nien-16-rla-supp:lasso_gap} shows the duality gap curves of relaxed LALM with different relaxation parameters ($\alpha=1, 1.999$) and AL penalty parameters ($\rho=0.5,0.1$). As seen in Figure~\ref{fig:nien-16-rla-supp:lasso_bound_vs_egap}, the ergodic duality gap $\fx{\mathcal{G}'}{\mb{w}_K;\hat{\mb{w}}}$ converges at rate $\fx{\mathcal{O}}{1/k}$, and the bound derived in Theorem~\ref{thm:nien-16-rla-supp:conv_rate_proposed_relaxed_lalm} is a tighter upper bound for large number of iterations. Furthermore, as seen in Figure~\ref{fig:nien-16-rla-supp:lasso_egap_vs_negap}, the non-ergodic duality gap $\bfx{\mathcal{G}'}{\iter{\mb{w}}{K};\hat{\mb{w}}}$ converges much faster than the ergodic one, and we can achieve about two-times speed-up by using $\alpha\approx 2$ empirically.
\begin{figure}
    \centering
    \subfigure[Bound \eqref{eq:nien-16-rla-supp:conv_rate_proposed_relax} vs. ergodic gap (E-gap)]{
    \includegraphics[width=0.45\textwidth]{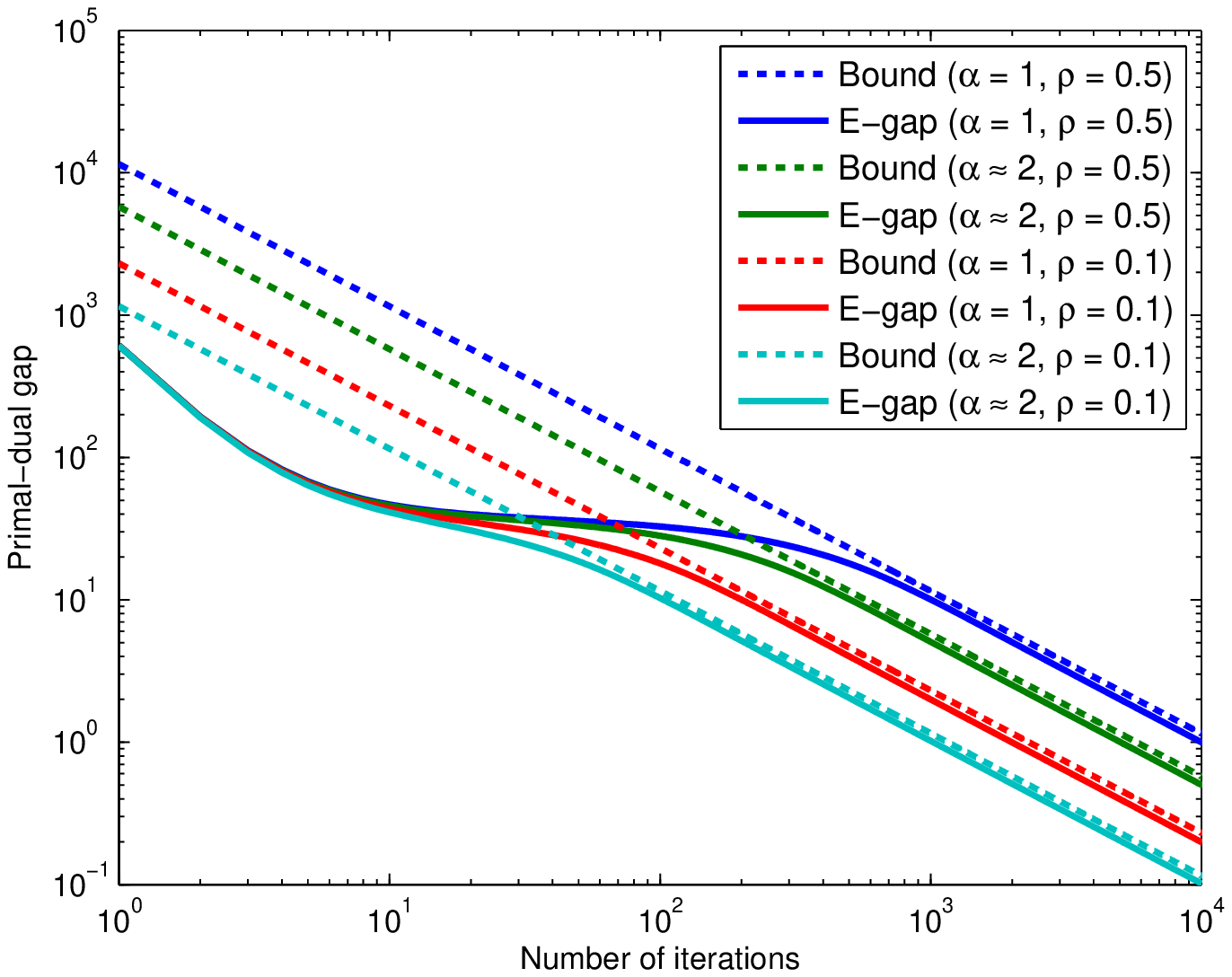}
    \label{fig:nien-16-rla-supp:lasso_bound_vs_egap}
    }
    \subfigure[Ergodic gap (E-gap) vs. non-ergodic gap (NE-gap)]{
    \includegraphics[width=0.45\textwidth]{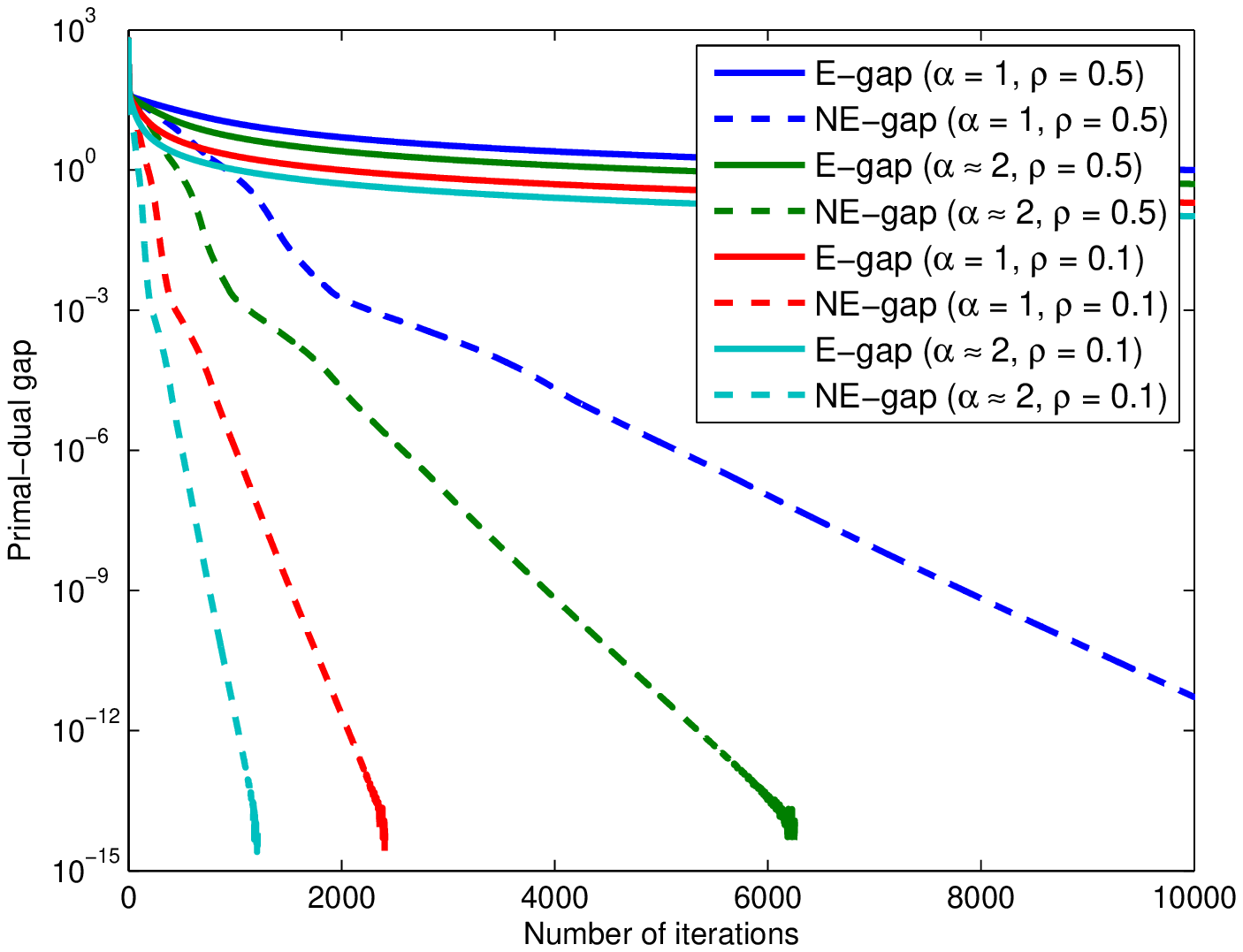}
    \label{fig:nien-16-rla-supp:lasso_egap_vs_negap}
    }
    \caption{LASSO regression: Duality gap curves of relaxed LALM with different relaxation parameters and AL penalty parameters. (a) Bound \eqref{eq:nien-16-rla-supp:conv_rate_proposed_relax} vs. ergodic gap, and (b) ergodic gap vs. non-ergodic gap.}
    \label{fig:nien-16-rla-supp:lasso_gap}
\end{figure}

\section{Continuation with over-relaxation} \label{sec:nien-16-rla-supp:continuation}
This section describes the rationale for the continuation sequence in \cite{nien:16:rla}. Consider solving a simple quadratic problem:
\begin{equation} \label{eq:nien-16-rla-supp:simple_quad_problem}
    \hat{\mb{x}}
    \in
    \nbargmin{\mb{x}}{\ts\frac{1}{2}\norm{\mb{Ax}}{2}^2} \, ,
\end{equation}
using \cite[Eqn. 33]{nien:16:rla} with $h=0$ and $\mb{y}=\mb{0}$. If $\mb{A}'\mb{A}$ is positive definite (for this analysis only), then \eqref{eq:nien-16-rla-supp:simple_quad_problem} has a unique solution $\hat{\mb{x}}=\mb{0}$. Let $\mb{V}\msb{\Lambda}\mb{V}'$ be the eigenvalue decomposition of $\mb{A}'\mb{A}$, where $\msb{\Lambda}\teq\diag{\lambda_i|\,0<\lambda_1\leq\dots\leq\lambda_n=L_{\mb{A}}}$. Updates generated by \cite[Eqn. 33]{nien:16:rla} (with $\mb{D}_{\mb{A}}=L_{\mb{A}}\mb{I}$) simplify as follows:
\begin{equation} \label{eq:nien-16-rla-supp:iter_simple_quad_problem}
    \begin{cases}
    \iter{\mb{x}}{k+1}=\tfrac{1}{\rho L_{\mb{A}}}\big((\rho-1)\iter{\mb{g}}{k}+\rho\iter{\mb{h}}{k}\big) \\
    \iter{\mb{g}}{k+1}=\tfrac{\rho}{\rho+1}\big(\alpha\mb{A}'\mb{A}\iter{\mb{x}}{k+1}+(1-\alpha)\iter{\mb{g}}{k}\big)+\tfrac{1}{\rho+1}\iter{\mb{g}}{k} \\
    \iter{\mb{h}}{k+1}=\alpha\big(L_{\mb{A}}\iter{\mb{x}}{k+1}-\mb{A}'\mb{A}\iter{\mb{x}}{k+1}\big)+(1-\alpha)\iter{\mb{h}}{k} \, .
    \end{cases}
\end{equation}

Let $\bar{\mb{x}}\teq\mb{V}'\mb{x}$, $\bar{\mb{g}}\teq\mb{V}'\mb{g}$, and $\bar{\mb{h}}\teq\mb{V}'\mb{h}$. The linear system \eqref{eq:nien-16-rla-supp:iter_simple_quad_problem} can be further diagonalized, and the $i$th components of $\bar{\mb{x}}$, $\bar{\mb{g}}$, and $\bar{\mb{h}}$ evolve as follows:
\begin{equation} \label{eq:nien-16-rla-supp:iter_simple_quad_problem_ith_x}
    \iter{\bar{x}_i}{k+1}=\tfrac{1}{\rho L_{\mb{A}}}\big((\rho-1)\iter{\bar{g}_i}{k}+\rho\iter{\bar{h}_i}{k}\big)
\end{equation}
and
\begin{equation} \label{eq:nien-16-rla-supp:iter_simple_quad_problem_ith_gh}
    \begin{cases}
    \iter{\bar{g}_i}{k+1}=\tfrac{\rho}{\rho+1}\big(\alpha\lambda_i\iter{\bar{x}_i}{k+1}+(1-\alpha)\iter{\bar{g}_i}{k}\big)+\tfrac{1}{\rho+1}\iter{\bar{g}_i}{k} \\
    \iter{\bar{h}_i}{k+1}=\alpha\big(L_{\mb{A}}\iter{\bar{x}_i}{k+1}-\lambda_i\iter{\bar{x}_i}{k+1}\big)+(1-\alpha)\iter{\bar{h}_i}{k} \, .
    \end{cases}
\end{equation}
Plugging \eqref{eq:nien-16-rla-supp:iter_simple_quad_problem_ith_x} into \eqref{eq:nien-16-rla-supp:iter_simple_quad_problem_ith_gh} leads to a second-order recursion (of $\bar{g}_i$ and $\bar{h}_i$) with a transition matrix
\begin{equation} \label{eq:nien-16-rla-supp:def_Ti}
    \mb{T}_i
    \teq
    \begin{bmatrix}
    \tfrac{\alpha\rho\lambda_i}{\rho+1}\tfrac{1}{\rho L_{\mb{A}}}\left(\rho-1\right)+\tfrac{\left(1-\alpha\right)\rho+1}{\rho+1} & \tfrac{\alpha\rho\lambda_i}{\rho+1}\tfrac{1}{\rho L_{\mb{A}}}\rho \\
    \alpha\left(L_{\mb{A}}-\lambda_i\right)\tfrac{1}{\rho L_{\mb{A}}}\left(\rho-1\right) & \alpha\left(L_{\mb{A}}-\lambda_i\right)\tfrac{1}{\rho L_{\mb{A}}}\rho+\left(1-\alpha\right) \\
    \end{bmatrix} \, ,
\end{equation}
and $\iter{\bar{x}_i}{k+1}$ is just a linear combination of $\iter{\bar{g}_i}{k}$ and $\iter{\bar{h}_i}{k}$. The eigenvalues of the transition matrix $\mb{T}_i$ defined in \eqref{eq:nien-16-rla-supp:def_Ti} determine the convergence rate of the second-order recursion, and we can analyze the second-order recursive system by studying its characteristic polynomial:
\begin{equation} \label{eq:nien-16-rla-supp:char_polynomial}
    r_i^2-\left([\mb{T}_i]_{11}+[\mb{T}_i]_{22}\right)r_i+\left([\mb{T}_i]_{11}[\mb{T}_i]_{22}-[\mb{T}_i]_{12}[\mb{T}_i]_{21}\right) \, .
\end{equation}

The proposed $\alpha$-dependent continuation sequence is based on the critical value $\rho_1^{\text{c}}$ and the damping frequency $\omega_1$ (as $\rho\approx0$) of the eigencomponent corresponding to the smallest eigenvalue $\lambda_1$ \cite{nien:15:fxr}. The critical value $\rho_1^{\text{c}}$ solves
\begin{equation} \label{eq:nien-16-rla-supp:critical_cond}
    \left([\mb{T}_1]_{11}+[\mb{T}_1]_{22}\right)^2-4\left([\mb{T}_1]_{11}[\mb{T}_1]_{22}-[\mb{T}_1]_{12}[\mb{T}_1]_{21}\right)=0 \, ,
\end{equation}
and the damping frequency $\omega_1$ satisfies \cite[p. 581]{chiang:84:fmo}
\begin{equation} \label{eq:nien-16-rla-supp:damping_cond}
    \cos\omega_1=\frac{[\mb{T}_1]_{11}+[\mb{T}_1]_{22}}{\sqrt{4\left([\mb{T}_1]_{11}[\mb{T}_1]_{22}-[\mb{T}_1]_{12}[\mb{T}_1]_{21}\right)}} \, .
\end{equation}
We solve \eqref{eq:nien-16-rla-supp:critical_cond} and \eqref{eq:nien-16-rla-supp:damping_cond} using MATLAB's symbolic toolbox. For \eqref{eq:nien-16-rla-supp:critical_cond}, we found that
\begin{equation} \label{eq:nien-16-rla-supp:critical_rho}
    \rho_1^{\text{c}}
    =
    2\sqrt{\tfrac{\lambda_1}{L_{\mb{A}}}\left(1-\tfrac{\lambda_1}{L_{\mb{A}}}\right)}
\end{equation}
is independent of $\alpha$. Hence, the optimal AL penalty parameter $\rho^{\star}\teq\rho_1^{\text{c}}$ depends only on the geometry of $\mb{A}'\mb{A}$ and does not change for different values of the relaxation parameter $\alpha$. For \eqref{eq:nien-16-rla-supp:damping_cond}, we found that
\begin{equation} \label{eq:nien-16-rla-supp:damping_psi}
    \cos\omega_1
    \approx
    \frac{1-\alpha\tfrac{\lambda_1}{L_{\mb{A}}}}{\sqrt{1-\left(2\alpha-\alpha^2\right)\tfrac{\lambda_1}{L_{\mb{A}}}}}
\end{equation}
for $\rho\approx0$. When $\alpha=1$, $\cos\omega_1\approx\sqrt{1-\lambda_1/L_{\mb{A}}}$, and thus $\omega_1\approx\sqrt{\lambda_1/L_{\mb{A}}}$ due to the small angle approximation:
\begin{equation} \label{eq:nien-16-rla-supp:small_angle_approx}
    \cos\sqrt{\theta}\approx1-\theta/2\approx\sqrt{1-\theta} \, .
\end{equation}
When $\alpha\approx2$, $\cos\omega_1\approx1-2\lambda_1/L_{\mb{A}}$, and $\omega_1\approx2\sqrt{\lambda_1/L_{\mb{A}}}$ also due to \eqref{eq:nien-16-rla-supp:small_angle_approx}.
For general $0<\alpha<2$, we can approximate $\cos\omega_1$ in \eqref{eq:nien-16-rla-supp:damping_psi} using a Taylor series as
\begin{equation} \label{eq:nien-16-rla-supp:cos_psi_general_alpha}
    \cos\omega_1
    \approx
    \left(1-\alpha\tfrac{\lambda_1}{L_{\mb{A}}}\right)
    \left(1+\tfrac{1}{2}\left(2\alpha-\alpha^2\right)\tfrac{\lambda_1}{L_{\mb{A}}}+\text{[higher-order terms]}\right)
    =
    1-\tfrac{\alpha^2}{2}\tfrac{\lambda_1}{L_{\mb{A}}}+\text{[higher-order terms]} \, .
\end{equation}
We ignore higher-order terms in \eqref{eq:nien-16-rla-supp:cos_psi_general_alpha} since $\lambda_1/L_{\mb{A}}$ is usually very small in practice. Hence, $\cos\omega_1\approx1-\big(\alpha\sqrt{\lambda_1/L_{\mb{A}}}\big)^2/2$, and $\omega_1\approx\alpha\sqrt{\lambda_1/L_{\mb{A}}}$ due to the small angle approximation \eqref{eq:nien-16-rla-supp:small_angle_approx}. This expression covers both the previous unrelaxed \mbox{($\alpha=1$)} and proposed relaxed ($\alpha\approx2$) cases. Suppose we use the same restart condition as in \cite{nien:15:fxr}; that is, restarts occur about every $\left(1/2\right)\left(\pi/\omega_1\right)$ iterations. If we restart at the $k$th iteration, we have the approximation $\sqrt{\lambda_1/L_{\mb{A}}}\approx\pi/(2\alpha k)$, and the ideal AL penalty parameter at the $k$th iteration is
\begin{equation} \label{eq:nien-16-rla-supp:rho_seq_general}
    2\sqrt{\left(\tfrac{\pi}{2\alpha k}\right)^2\big(1-\left(\tfrac{\pi}{2\alpha k}\right)^2\big)}
    =
    \tfrac{\pi}{\alpha k}\sqrt{1-\left(\tfrac{\pi}{2\alpha k}\right)^2} \, .
\end{equation}
That is, the values of $\fx{\rho_k}{\alpha}$ are scaled by the value of $\alpha$.

To demonstrate the speed-up resulting from combining continuation with over-relaxation, Figure~\ref{fig:nien-16-rla-supp:xcat_rmsd_different_alpha_nblock_12_all} shows the convergence rate curves of the proposed relaxed \mbox{OS-LALM} ($12$ subsets) using different values of the over-relaxation parameter $\alpha$ when reconstructing the simulated XCAT dataset. For comparison, the convergence rate curves that do not use continuation (fixed AL penalty parameter $\rho=0.05$) are also shown. As seen in Figure~\ref{fig:nien-16-rla-supp:xcat_rmsd_different_alpha_nblock_12}, the RMS difference of the green curve (relaxed \mbox{OS-LALM} with $\alpha=1.5$) after $10$ iterations is about the same as the RMS difference of the blue curve (unrelaxed \mbox{OS-LALM}) after $15$ iterations, exhibiting an approximately $1.5$-times speed-up. Using larger $\alpha$ (up to two) can further accelerate convergence; however, the speed-up can be slightly slower than $\alpha$-times due to the dominance of the constant $\overline{B}$ in \cite[Theorem~2]{nien:16:rla} and the accumulation of gradient errors with ordered subsets. For instance, the RMS difference of the red curve (relaxed \mbox{OS-LALM} with $\alpha=1.999$) after $5$ iterations is a bit larger the RMS difference of the blue curve (unrelaxed \mbox{OS-LALM}) after $10$ iterations.
\begin{figure}
    \centering
    \subfigure[Fixed $\rho=0.05$]{
    \includegraphics[width=0.45\textwidth]{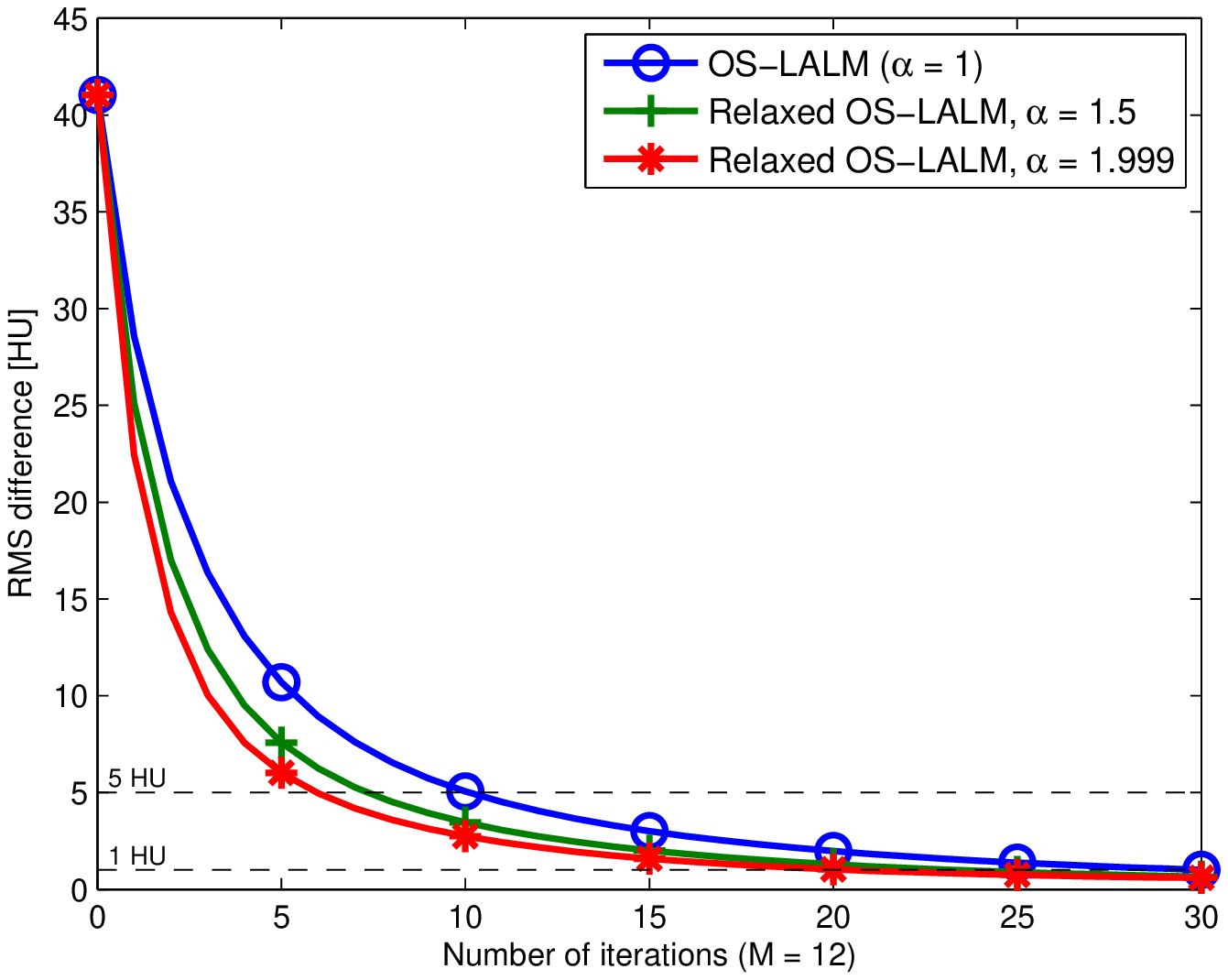}
    \label{fig:nien-16-rla-supp:xcat_rmsd_fixed_rho_different_alpha_nblock_12}
    }
    \subfigure[Proposed decreasing $\rho_k$]{
    \includegraphics[width=0.45\textwidth]{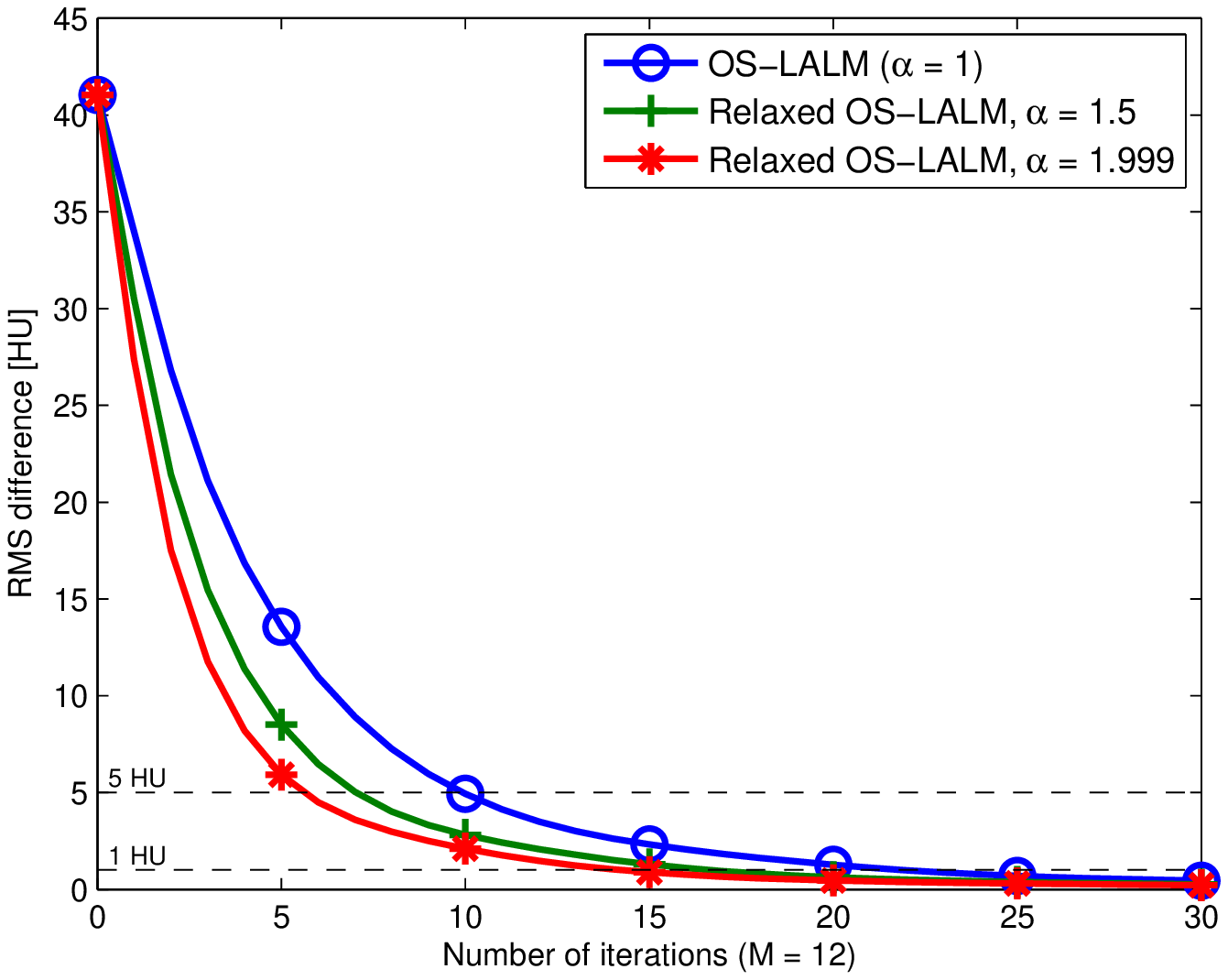}
    \label{fig:nien-16-rla-supp:xcat_rmsd_different_alpha_nblock_12}
    }
    \caption{XCAT: Convergence rate curves of relaxed \mbox{OS-LALM} ($12$ subsets) using different values of the over-relaxation parameter $\alpha$ with (a) fixed AL penalty parameter $\rho=0.05$ and (b) proposed decreasing $\rho_k$.}
    \label{fig:nien-16-rla-supp:xcat_rmsd_different_alpha_nblock_12_all}
\end{figure}

\section{Additional experimental results} \label{sec:nien-16-rla-supp:additional_results}
\subsection{XCAT phantom} \label{subsec:nien-16-rla-supp:xcat_phantom}
Additional experimental results of the simulated XCAT phantom dataset shown in \cite{nien:16:rla} are reported here. Figure~\ref{fig:nien-16-rla-supp:xcat_diff_10_iter_12_blocks} shows the difference images (in the central transaxial plane) of FBP (i.e., $\iter{\mb{x}}{0}-\mb{x}^{\star}$) and OS algorithms with $12$ subsets after $10$ iterations (i.e., $\iter{\mb{x}}{10}-\mb{x}^{\star}$). As seen in Figure~\ref{fig:nien-16-rla-supp:xcat_diff_10_iter_12_blocks}, low-frequency components converge faster than high/mid-frequency components like streaks and edges with all algorithms. This is common for gradient-based algorithms when the Hessian matrix of the cost function is more ``low-pass/band-cut'' like in X-ray CT. The difference image of the proposed relaxed \mbox{OS-LALM} shows less edge structures and looks more uniform in flat regions. Figure~\ref{fig:nien-16-rla-supp:xcat_diff_20_iter_12_blocks} shows the difference images after $20$ iterations. We can see that the proposed relaxed \mbox{OS-LALM} shows very uniform difference images, while the subtle noise-like artifacts remain with \mbox{OS-OGM2}.

To demonstrate the improvement of our ``modified'' relaxed LALM (i.e., with ordered subsets and continuation) for X-ray CT image reconstruction problems, Figure~\ref{fig:nien-16-rla-supp:xcat_rmsd_all} shows convergence rate curves of unrelaxed/relaxed \mbox{OS-LALM} using different parameter settings with (a) one subset and (b) $12$ subsets. All algorithms run $360$ subiterations; however, those with OS should be faster in runtime because they perform fewer forward/back-projections. As seen in Figure~\ref{fig:nien-16-rla-supp:xcat_rmsd_all}, convergence rate curves of OS algorithms are scaled almost perfectly (in the horizontal axis) when using modest number of subsets ($M=12$). However, the scalability might be worse when using more subsets (more severe gradient error accumulation) or in other dataset. Moreover, solid lines (relaxed algorithms) always show about two-times faster convergence rate than dashed lines (unrelaxed algorithms), without and with continuation. Note that the solid blue line (relaxed LALM, $\rho=1/6$) and the dashed green line (unrelaxed LALM, $\rho=1/12$) in both cases are overlapped after $60$ subiterations, implying that halving the AL penalty parameter $\rho$ and setting relaxation parameter $\alpha$ to be close to two have similar effect on convergence speed in this CT problem (where the data fidelity term dominates the cost function). Note that when the data-fidelity term dominates the cost function, the constant $\overline{B}$ dominates the constant multiplying $1/K$ in \cite[Theorem~2]{nien:16:rla}, leading to the better speed-up with $\alpha$.

We also investigated the effect of majorization (for both the data-fidelity term and the regularizer term) on convergence speed. Figure~\ref{fig:nien-16-rla-supp:xcat_rmsd_denom} shows the convergence rate curves of the proposed relaxed \mbox{OS-LALM} with different (a) data-fidelity term majorizations and (b) regularization term majorizations. As seen in Figure~\ref{fig:nien-16-rla-supp:xcat_rmsd_denom_loss}, the proposed algorithm diverges when $\mb{D}_{\L}$ is too small, violating the majorization condition. Larger $\mb{D}_{\L}$ slows down the algorithm. However, multiplying $\mb{D}_{\L}$ by $\kappa$-times does not necessarily slow down the algorithm by $\kappa$-times since the weighting matrix of $\overline{B}_{\alpha,\rho,\mb{D}_{\L}}$ is $\mb{D}_{\L}-\mb{A}'\mb{WA}$. Besides, larger $\mb{D}_{\L}$ helps reduce the gradient error accumulation in fast algorithms \cite{kim:15:cos}. Figure~\ref{fig:nien-16-rla-supp:xcat_rmsd_denom_reg} shows the convergence rate curves of the proposed relaxed \mbox{OS-LALM} with regularizer majorization using the maximum curvature and Huber's curvature, respectively. We can see that the speed-up of using Huber's curvature is very significant. Note that $\rho\mb{D}_{\L}+\mb{D}_{\R}$ determines the step sizes of the image update of the proposed relaxed \mbox{OS-LALM}. Better majorization of $\R$ (i.e., smaller $[\mb{D}_{\R}]_i$ for those voxels that are still far from the optimum) leads to larger image update step sizes, especially when $\rho$ is small.
\begin{figure}
    \centering
    \includegraphics[width=\textwidth]{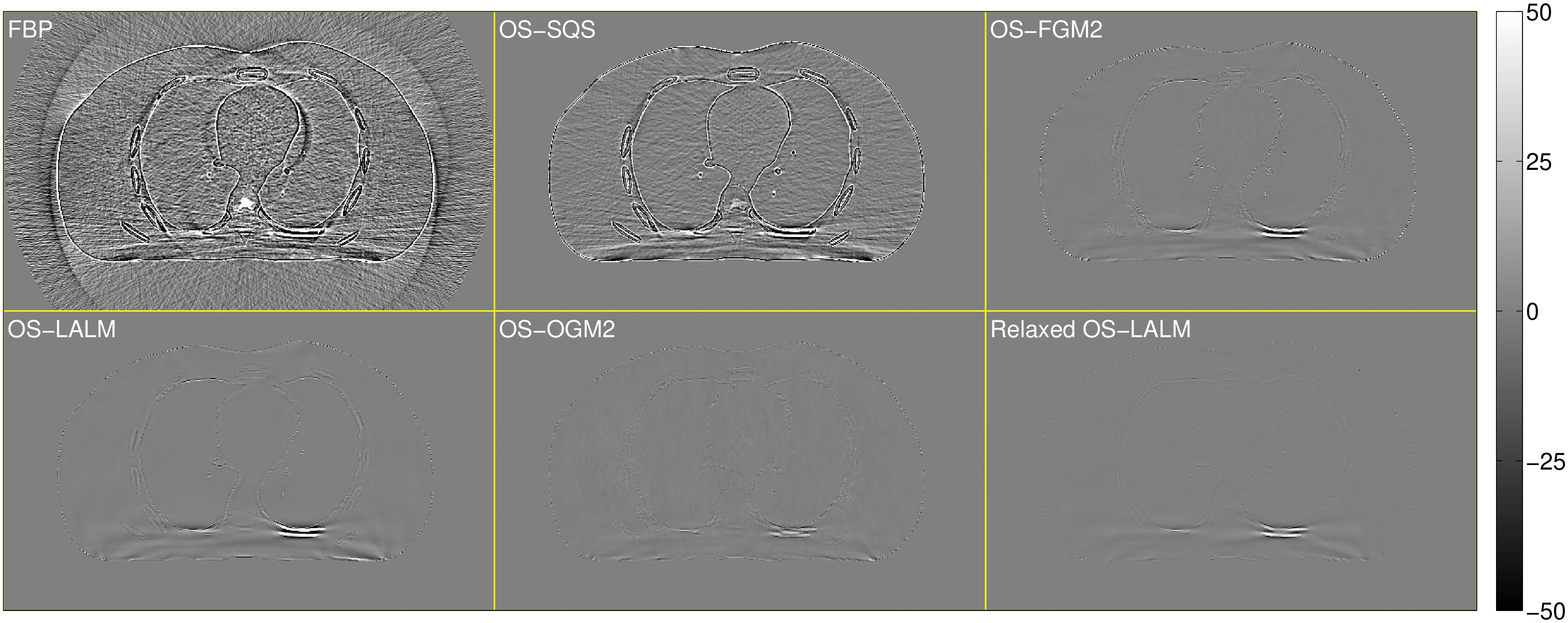}
    \caption{XCAT: Cropped difference images (displayed from ${-50}$ to $50$ HU) from the central transaxial plane of the initial FBP image $\iter{\mb{x}}{0}-\mb{x}^{\star}$ and the reconstructed image $\iter{\mb{x}}{10}-\mb{x}^{\star}$ using OS algorithms with $12$ subsets after $10$ iterations.}
    \label{fig:nien-16-rla-supp:xcat_diff_10_iter_12_blocks}
\end{figure}

\begin{figure}
    \centering
    \includegraphics[width=\textwidth]{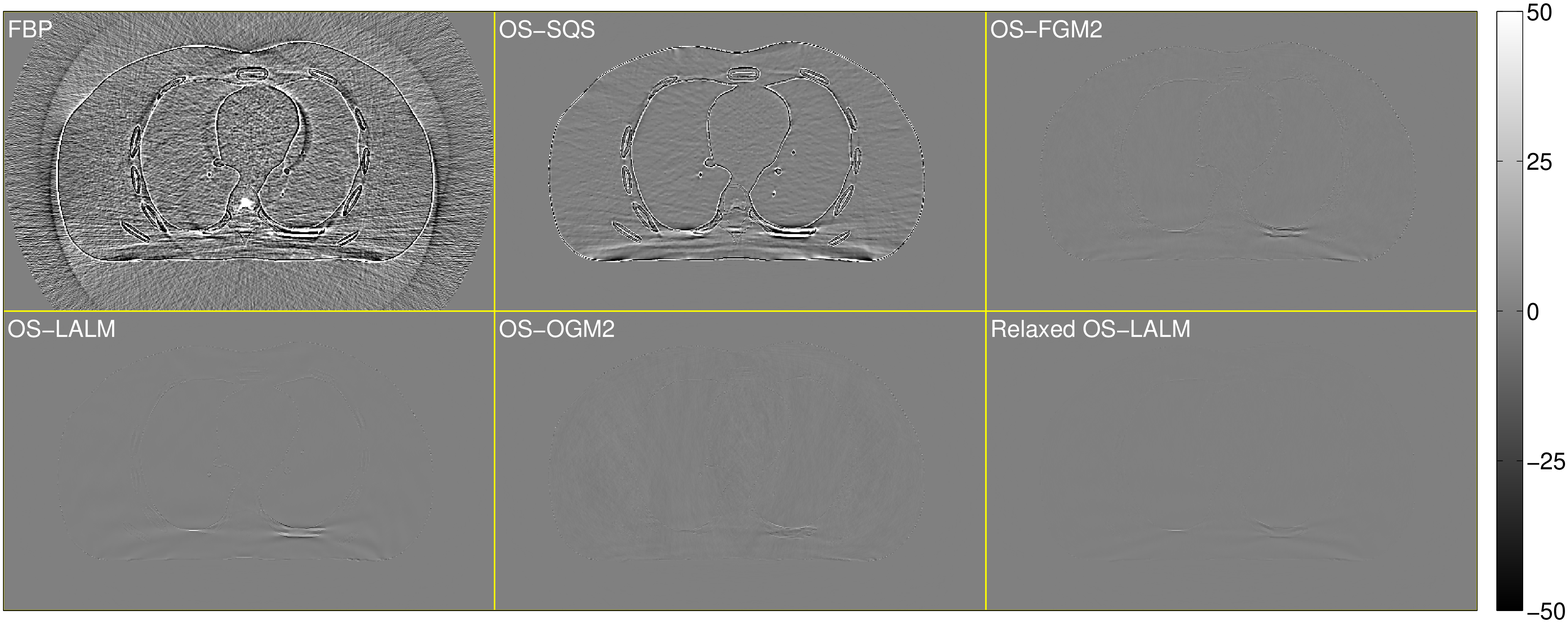}
    \caption{XCAT: Cropped difference images (displayed from ${-50}$ to $50$ HU) from the central transaxial plane of the initial FBP image $\iter{\mb{x}}{0}-\mb{x}^{\star}$ and the reconstructed image $\iter{\mb{x}}{20}-\mb{x}^{\star}$ using OS algorithms with $12$ subsets after $20$ iterations.}
    \label{fig:nien-16-rla-supp:xcat_diff_20_iter_12_blocks}
\end{figure}

\begin{figure}
    \centering
    \subfigure[One subset]{
    \includegraphics[width=0.45\textwidth]{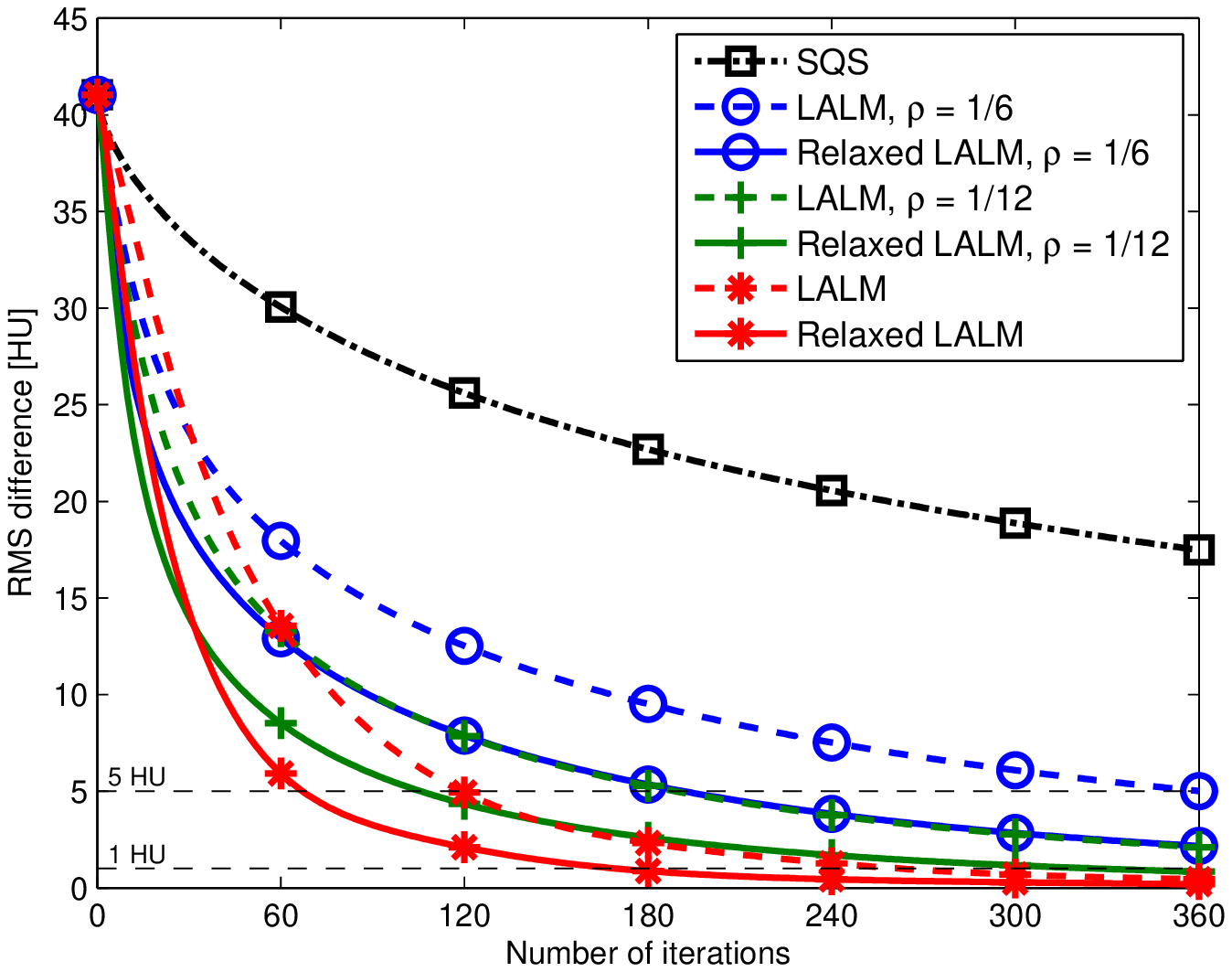}
    \label{fig:nien-16-rla-supp:xcat_rmsd_all_1_block}
    }
    \subfigure[$12$ subsets]{
    \includegraphics[width=0.45\textwidth]{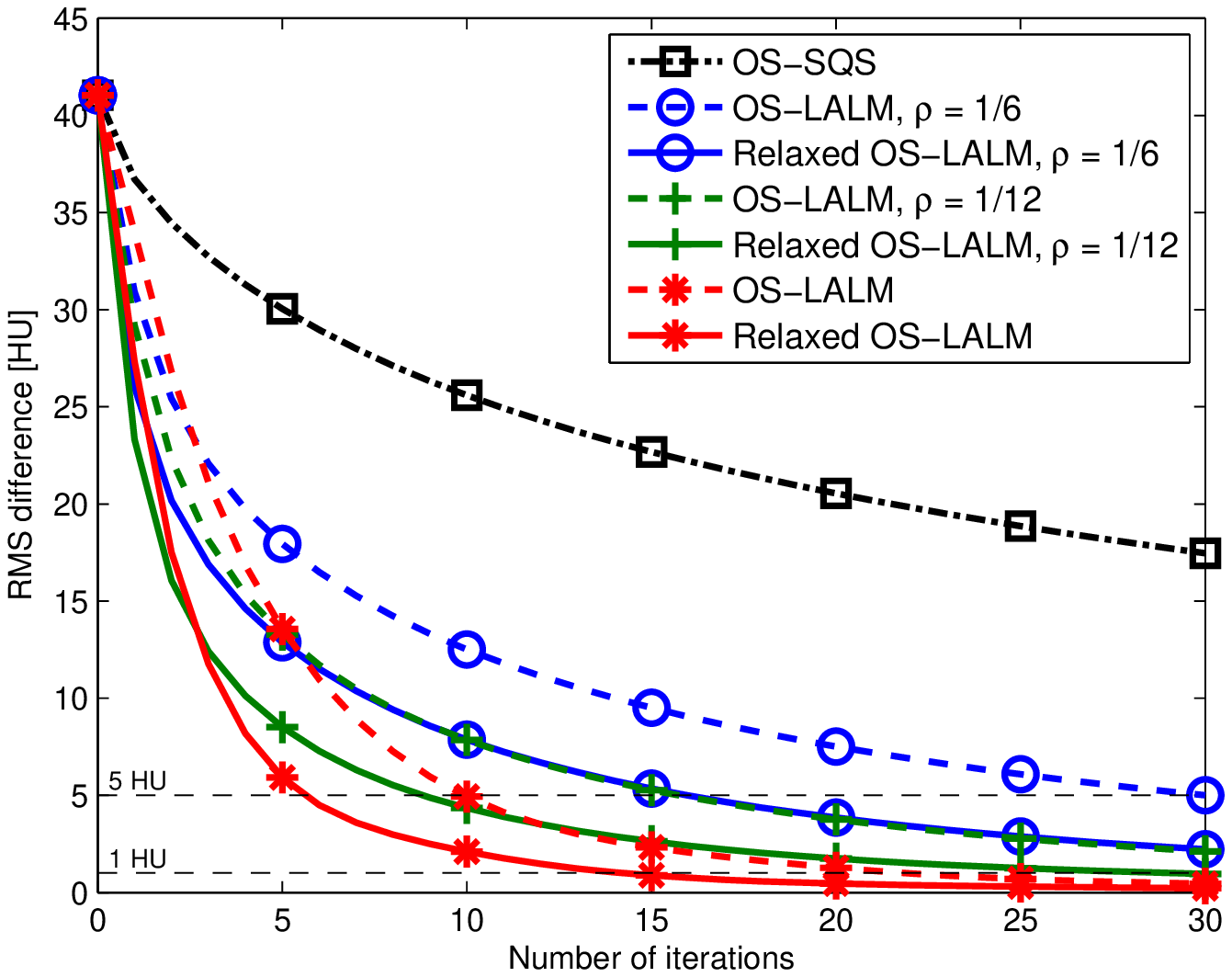}
    \label{fig:nien-16-rla-supp:xcat_rmsd_all_12_blocks}
    }
    \caption{XCAT: Convergence rate curves of unrelaxed/relaxed \mbox{OS-LALM} using different parameter settings with (a) one subset and (b) $12$ subsets. All algorithms run $360$ subiterations.}
    \label{fig:nien-16-rla-supp:xcat_rmsd_all}
\end{figure}

\begin{figure}
    \centering
    \subfigure[Data-fidelity term majorization]{
    \includegraphics[width=0.45\textwidth]{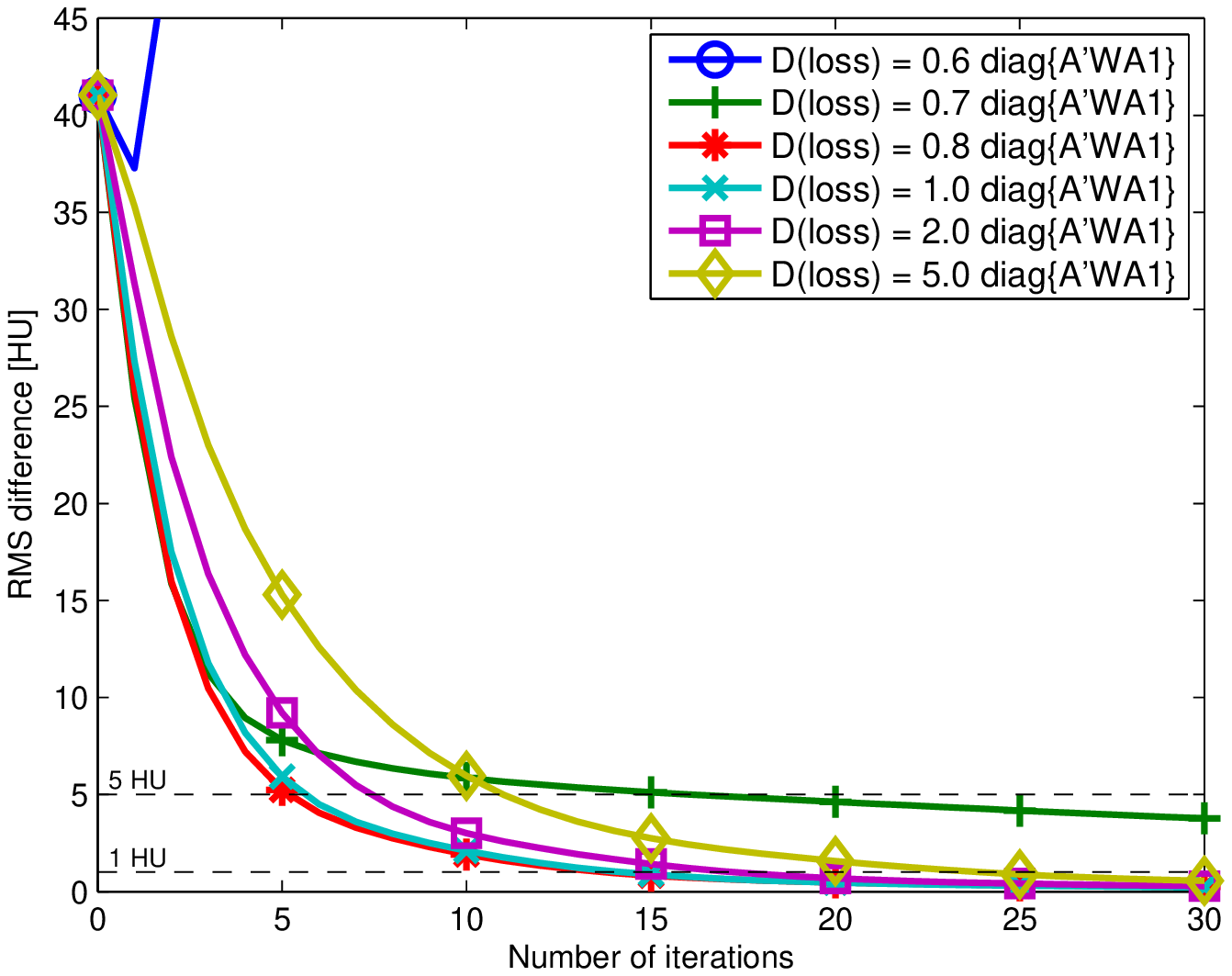}
    \label{fig:nien-16-rla-supp:xcat_rmsd_denom_loss}
    }
    \subfigure[Regularization term majorization]{
    \includegraphics[width=0.45\textwidth]{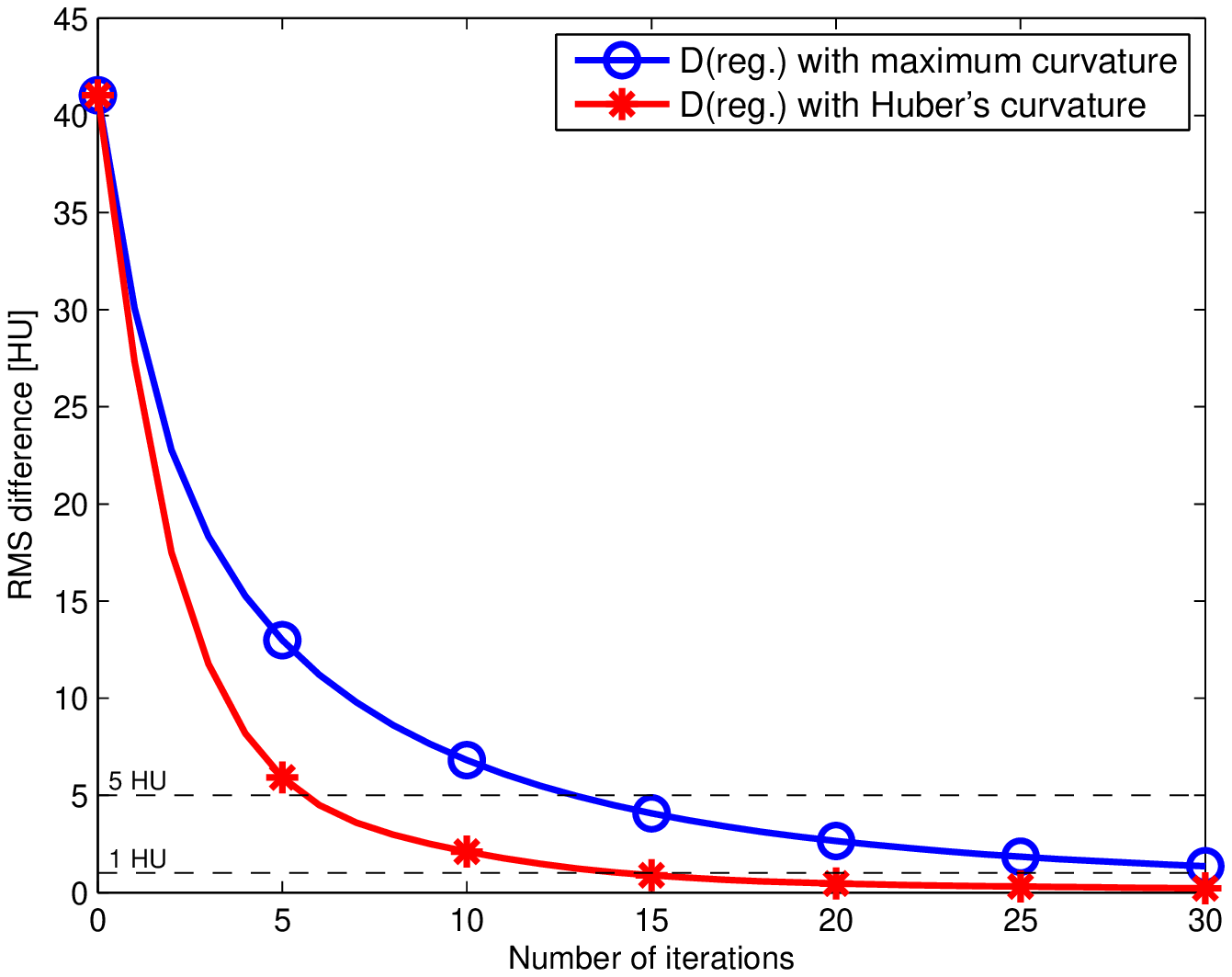}
    \label{fig:nien-16-rla-supp:xcat_rmsd_denom_reg}
    }
    \caption{XCAT: Convergence rate curves of the proposed relaxed \mbox{OS-LALM} with different (a) data-fidelity term majorizations and (b) regularization term majorizations.}
    \label{fig:nien-16-rla-supp:xcat_rmsd_denom}
\end{figure}

\subsection{Chest scan} \label{subsec:nien-16-rla-supp:chest_scan}
Additional experimental results of the chest scan dataset shown in \cite{nien:16:rla} are reported here. Figure~\ref{fig:nien-16-rla-supp:ct53_rlx_cmp_10_blocks} shows convergence rate curves of different relaxed algorithms ($10$ subsets and $\alpha=1.999$) with (a) a fixed AL penalty parameter $\rho=0.05$ and (b) the decreasing sequence $\rho_k$ proposed in \cite{nien:16:rla}. Like the experimental results with the simulated CT scan shown in \cite{nien:16:rla}, the simple relaxation does not provide much acceleration with a fixed AL penalty parameter, but it works somewhat better when using the decreasing $\rho_k$. Figure~\ref{fig:nien-16-rla-supp:ct53_diff_10_iter_10_blocks} and Figure~\ref{fig:nien-16-rla-supp:ct53_diff_20_iter_10_blocks} show the difference images (in the central transaxial plane) of FBP and OS algorithms with $10$ subsets after $10$ and $20$ iterations, respectively. Difference images of the proposed relaxed \mbox{OS-LALM} show the fewest structured artifacts among all algorithms for comparison.

\begin{figure}
    \centering
    \subfigure[Fixed $\rho=0.05$]{
    \includegraphics[width=0.45\textwidth]{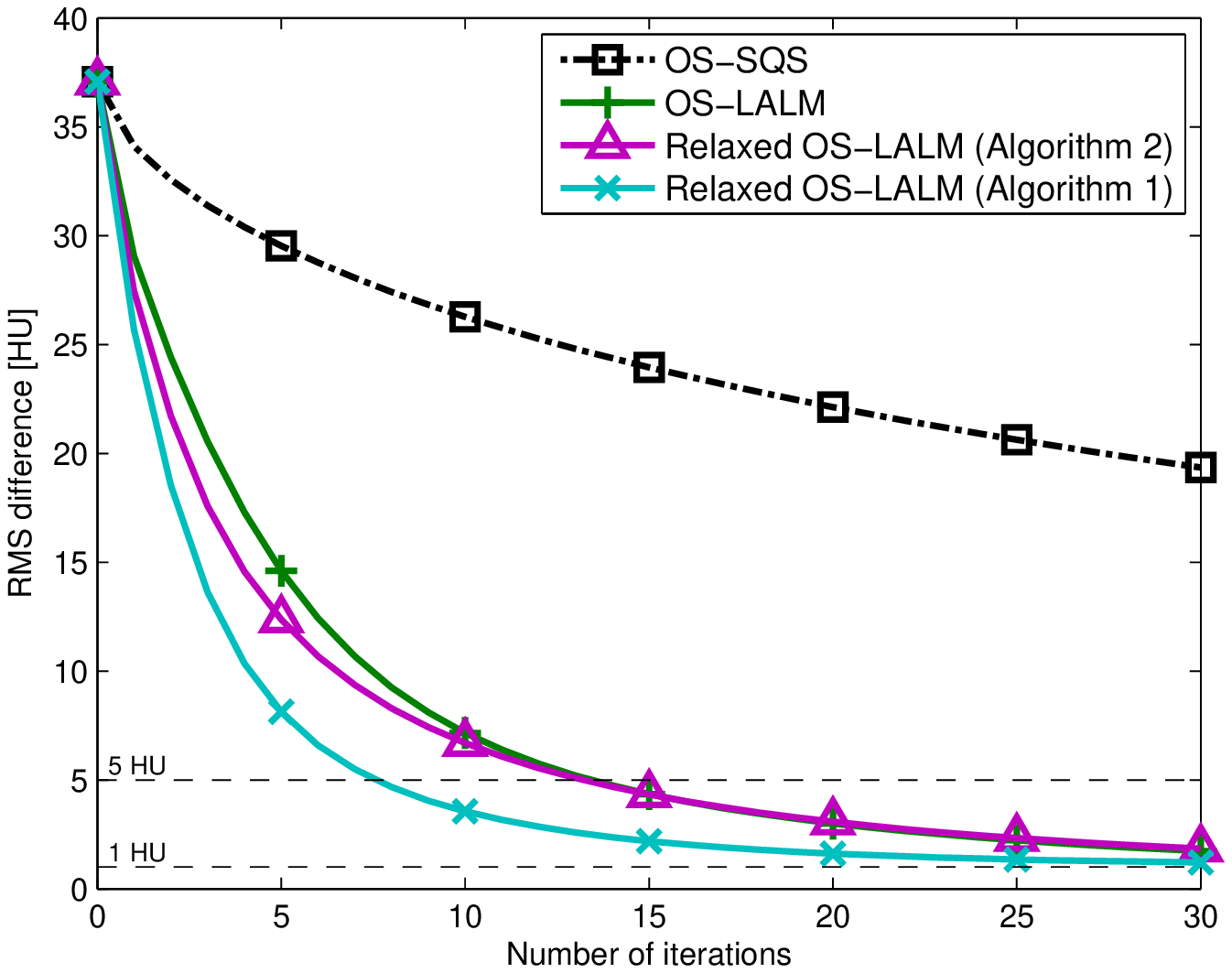}
    \label{fig:nien-16-rla-supp:ct53_relax_cmp_rho_p05_10_blocks}
    }
    \subfigure[Decreasing $\rho_k$ proposed in \cite{nien:16:rla}]{
    \includegraphics[width=0.45\textwidth]{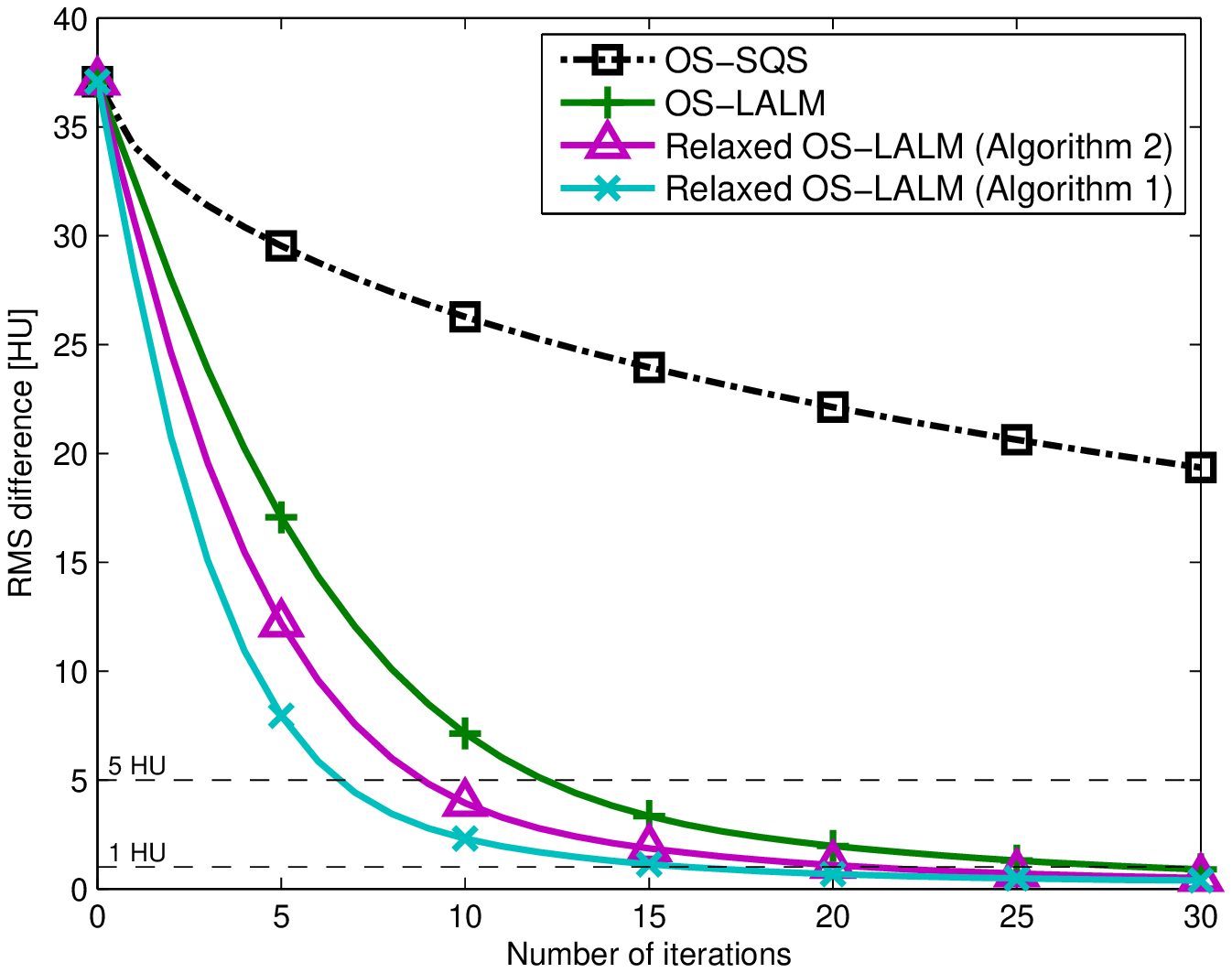}
    \label{fig:nien-16-rla-supp:ct53_relax_cmp_10_blocks}
    }
    \caption{Chest: Convergence rate curves of different relaxed algorithms ($10$ subsets and $\alpha=1.999$) with (a) a fixed AL penalty parameter $\rho=0.05$ and (b) the decreasing sequence $\rho_k$ proposed in \cite{nien:16:rla}.}
    \label{fig:nien-16-rla-supp:ct53_rlx_cmp_10_blocks}
\end{figure}

\begin{figure}
    \centering
    \includegraphics[width=\textwidth]{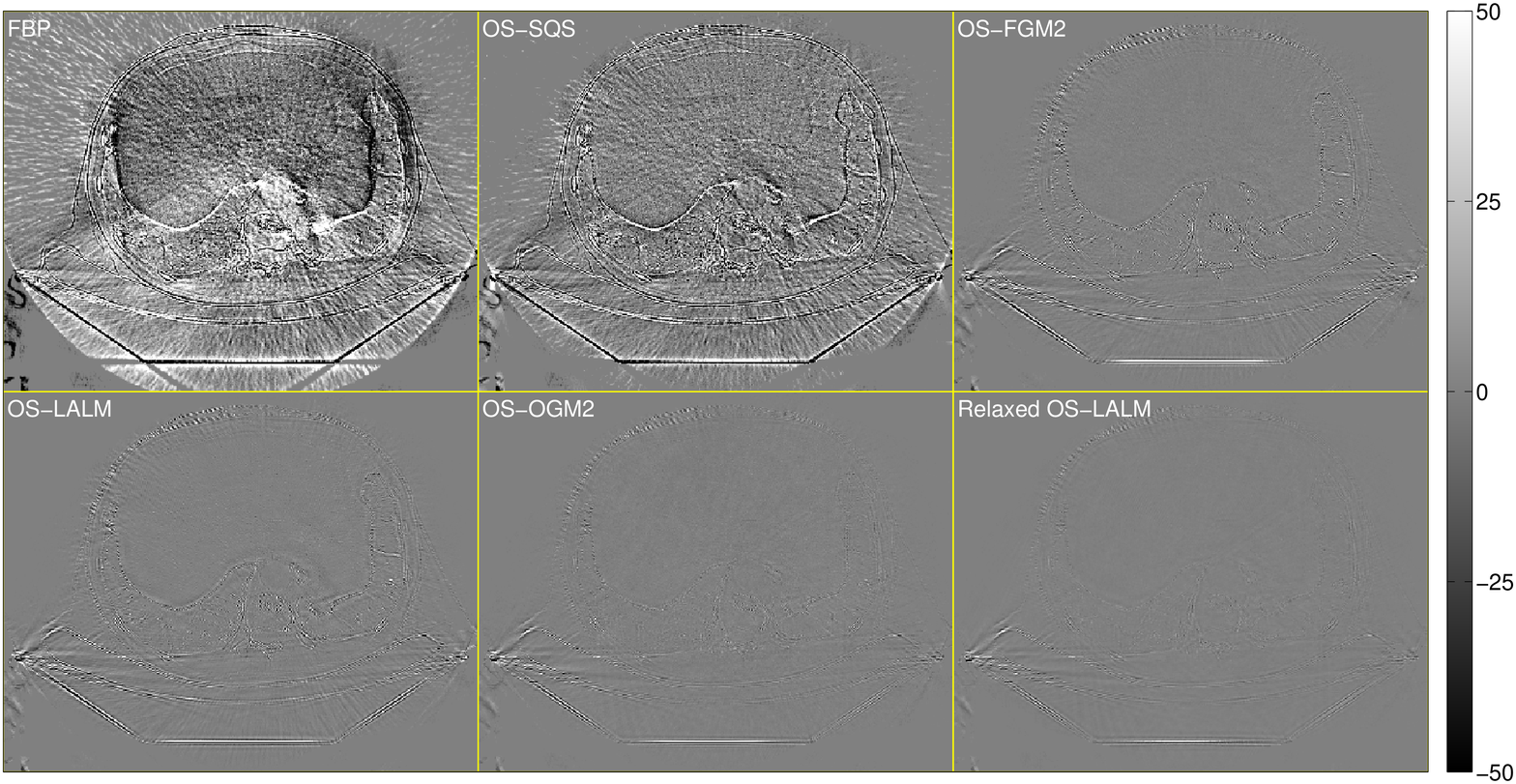}
    \caption{Chest: Cropped difference images (displayed from ${-50}$ to $50$ HU) from the central transaxial plane of the initial FBP image $\iter{\mb{x}}{0}-\mb{x}^{\star}$ and the reconstructed image $\iter{\mb{x}}{10}-\mb{x}^{\star}$ using OS algorithms with $10$ subsets after $10$ iterations.}
    \label{fig:nien-16-rla-supp:ct53_diff_10_iter_10_blocks}
\end{figure}

\begin{figure}
    \centering
    \includegraphics[width=\textwidth]{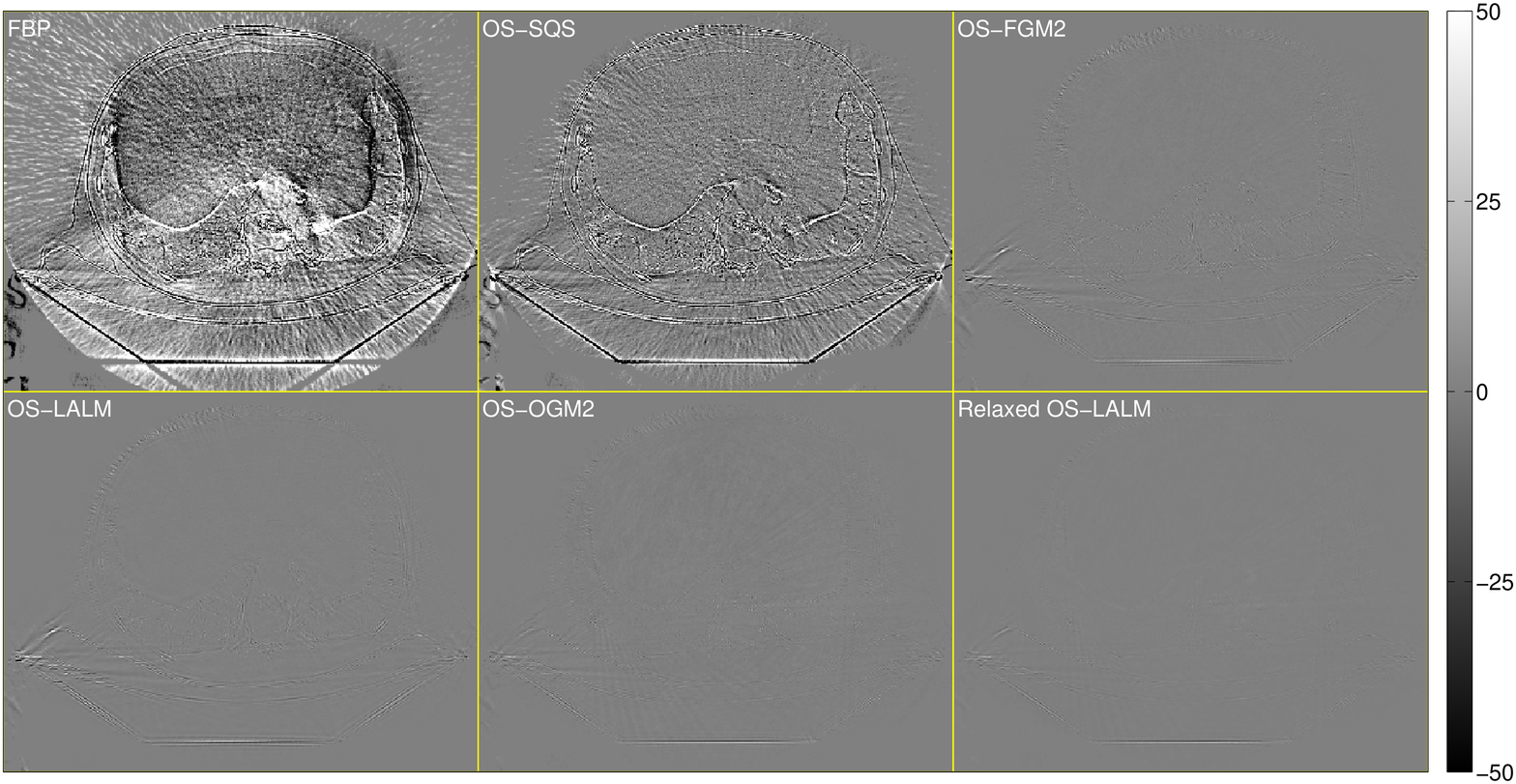}
    \caption{Chest: Cropped difference images (displayed from ${-50}$ to $50$ HU) from the central transaxial plane of the initial FBP image $\iter{\mb{x}}{0}-\mb{x}^{\star}$ and the reconstructed image $\iter{\mb{x}}{20}-\mb{x}^{\star}$ using OS algorithms with $10$ subsets after $20$ iterations.}
    \label{fig:nien-16-rla-supp:ct53_diff_20_iter_10_blocks}
\end{figure}

\subsection{Shoulder scan} \label{subsec:nien-16-rla-supp:shoulder_scan}
We reconstructed a $512\times 512\times 109$ image volume, where $\Delta_x=\Delta_y=1.3695$ mm and $\Delta_z=0.625$ mm, from a shoulder region helical CT scan. The size of sinogram is $888\times 32\times 7146$ ($\text{pitch}=0.5$, about $7.3$ rotations with rotation time $0.8$ seconds). The tube current and tube voltage of the X-ray source are $180$ mA and $140$ kVp, respectively. The initial FBP image $\iter{\mb{x}}{0}$ has lots of streak artifacts due to low signal-to-noise ratio (SNR), and we tuned the statistical weights and regularization parameters using \cite{chang:14:sxr,cho:15:rdf} to emulate \cite{thibault:07:atd,shuman:13:mbi}. We used $20$ subsets for the relaxed \mbox{OS-LALM}, while \cite[Eqn. 57]{nien:15:fxr} suggests using about $40$ subsets for the unrelaxed \mbox{OS-LALM}. Figure~\ref{fig:nien-16-rla-supp:ct01_fbp_conv_prop} shows the cropped images from the central transaxial plane of the initial FBP image $\iter{\mb{x}}{0}$, the reference reconstruction $\mb{x}^{\star}$, and the reconstructed image $\iter{\mb{x}}{20}$ using the proposed algorithm (relaxed \mbox{OS-LALM} with $20$ subsets) after $20$ iterations. Figure~\ref{fig:nien-16-rla-supp:shoulder_rmsd_cmp} shows the RMS differences between the reference reconstruction $\mb{x}^{\star}$ and the reconstructed image $\iter{\mb{x}}{k}$ using different OS algorithms as a function of iteration with $20$ and $40$ subsets. As seen in Figure~\ref{fig:nien-16-rla-supp:shoulder_rmsd_cmp}, the proposed relaxed \mbox{OS-LALM} shows faster convergence rate with moderate number of subsets, but the speed-up diminishes as the iterate approaches the solution. Figure~\ref{fig:nien-16-rla-supp:ct01_diff_10_iter_20_blocks} and Figure~\ref{fig:nien-16-rla-supp:ct01_diff_20_iter_20_blocks} show the difference images (in the central transaxial plane) of FBP and OS algorithms with $20$ subsets after $10$ and $20$ iterations, respectively. The proposed relaxed \mbox{OS-LALM} removes more streak artifacts than other OS algorithms.

\begin{figure}
    \centering
    \includegraphics[width=\textwidth]{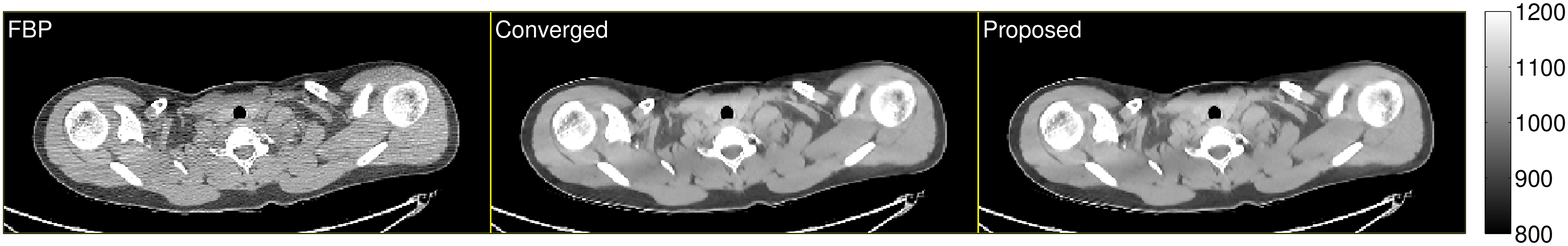}
    \caption{Shoulder: Cropped images (displayed from $800$ to $1200$ HU) from the central transaxial plane of the initial FBP image $\iter{\mb{x}}{0}$ (left), the reference reconstruction $\mb{x}^{\star}$ (center), and the reconstructed image $\iter{\mb{x}}{20}$ using the proposed algorithm (relaxed \mbox{OS-LALM} with $20$ subsets) after $20$ iterations (right).}
    \label{fig:nien-16-rla-supp:ct01_fbp_conv_prop}
\end{figure}

\begin{figure}
    \centering
    \subfigure[$20$ subsets]{
    \includegraphics[width=0.45\textwidth]{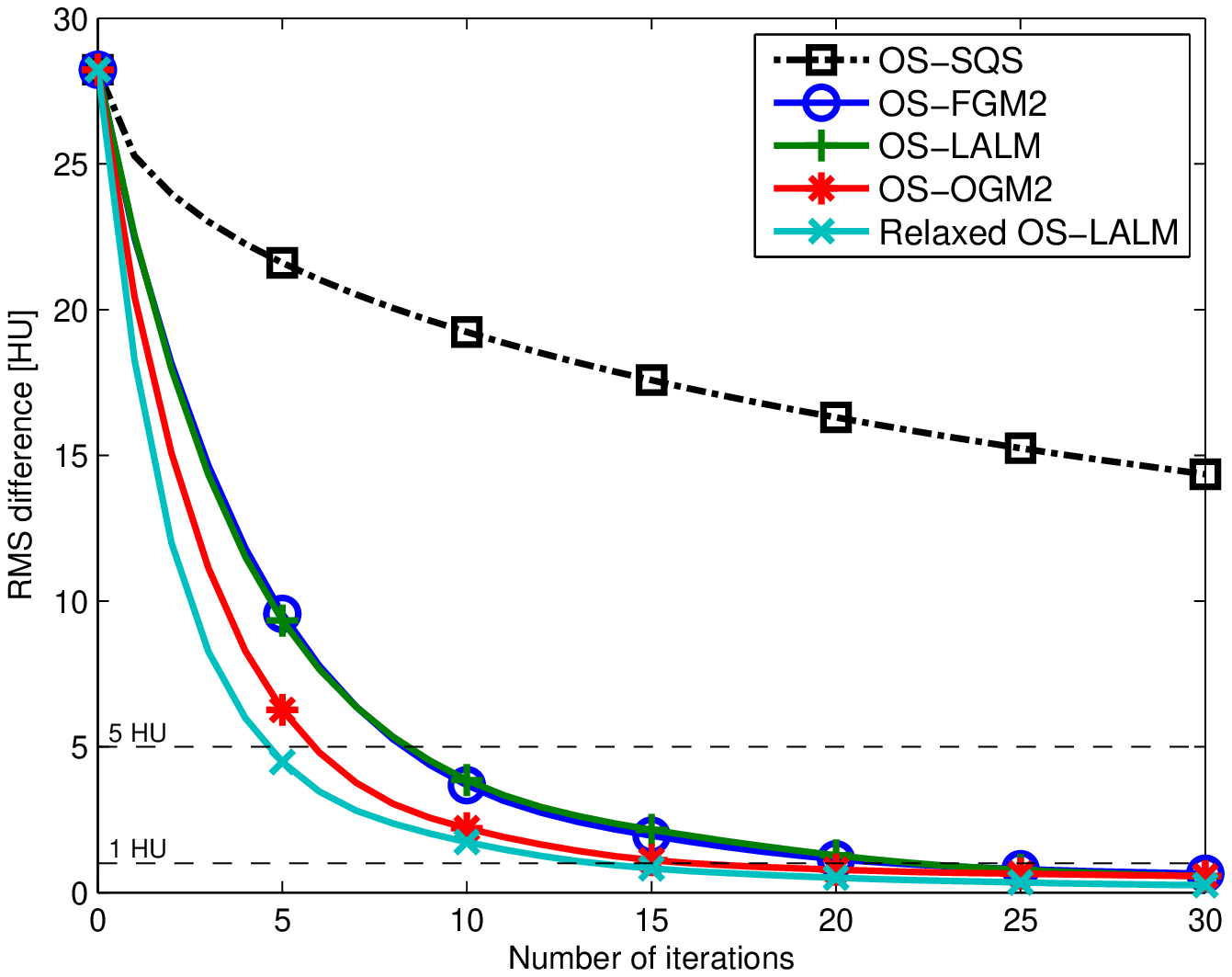}
    \label{fig:nien-16-rla-supp:shoulder_rmsd_20_blocks}
    }
    \subfigure[$40$ subsets]{
    \includegraphics[width=0.45\textwidth]{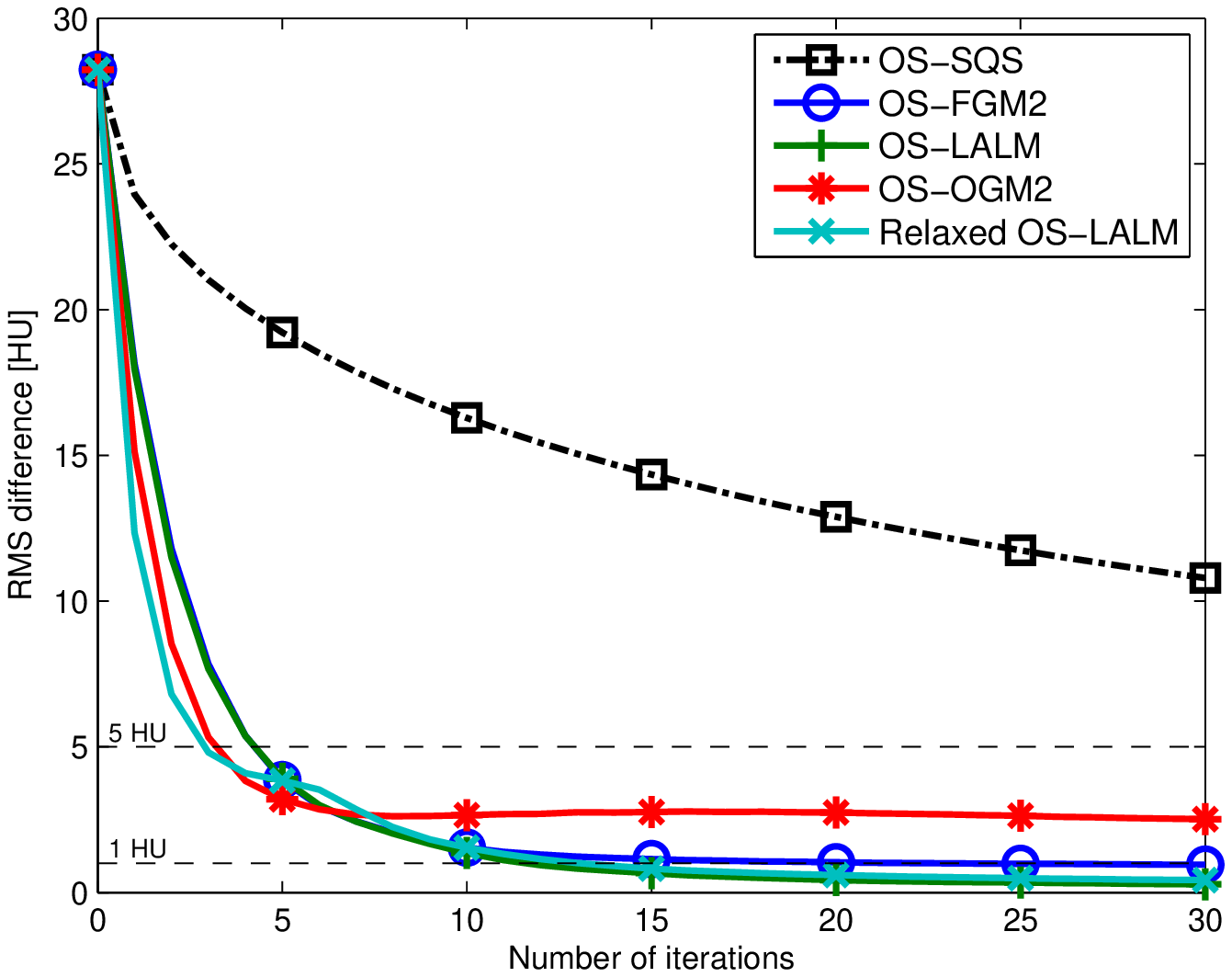}
    \label{fig:nien-16-rla-supp:shoulder_rmsd_40_blocks}
    }
    \caption{Shoulder: Convergence rate curves of different OS algorithms with (a) $20$ subsets and (b) $40$ subsets. The proposed relaxed \mbox{OS-LALM} with $20$ subsets exhibits similar convergence rate as the unrelaxed \mbox{OS-LALM} with $40$ subsets.}
    \label{fig:nien-16-rla-supp:shoulder_rmsd_cmp}
\end{figure}

\begin{figure}
    \centering
    \includegraphics[width=\textwidth]{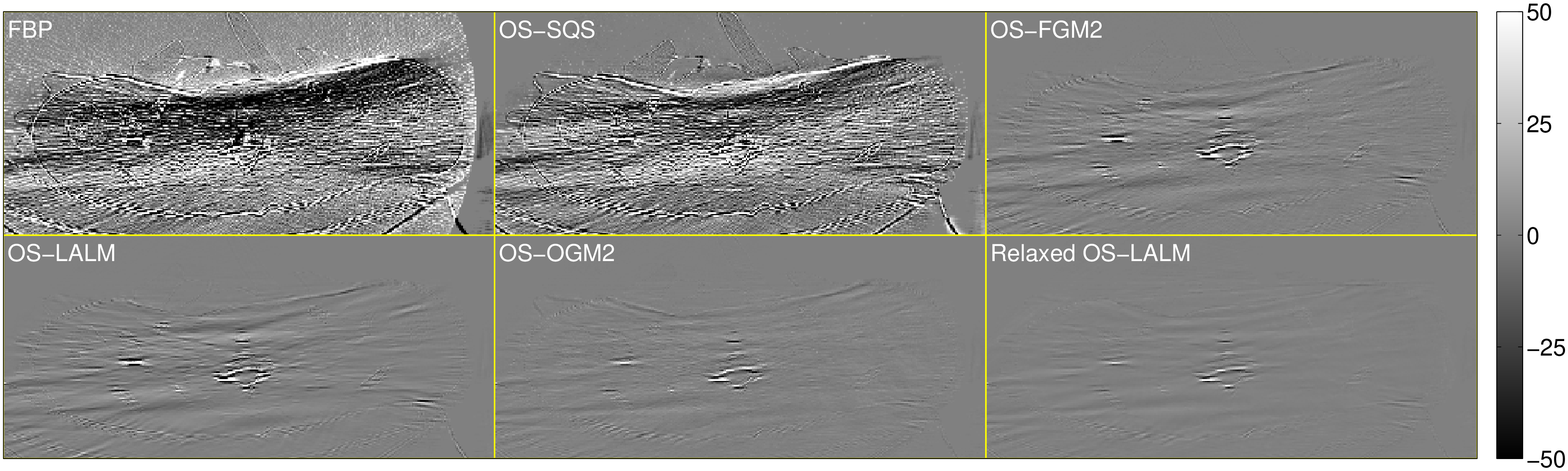}
    \caption{Shoulder: Cropped difference images (displayed from ${-50}$ to $50$ HU) from the central transaxial plane of the initial FBP image $\iter{\mb{x}}{0}-\mb{x}^{\star}$ and the reconstructed image $\iter{\mb{x}}{10}-\mb{x}^{\star}$ using OS algorithms with $20$ subsets after $10$ iterations.}
    \label{fig:nien-16-rla-supp:ct01_diff_10_iter_20_blocks}
\end{figure}

\begin{figure}
    \centering
    \includegraphics[width=\textwidth]{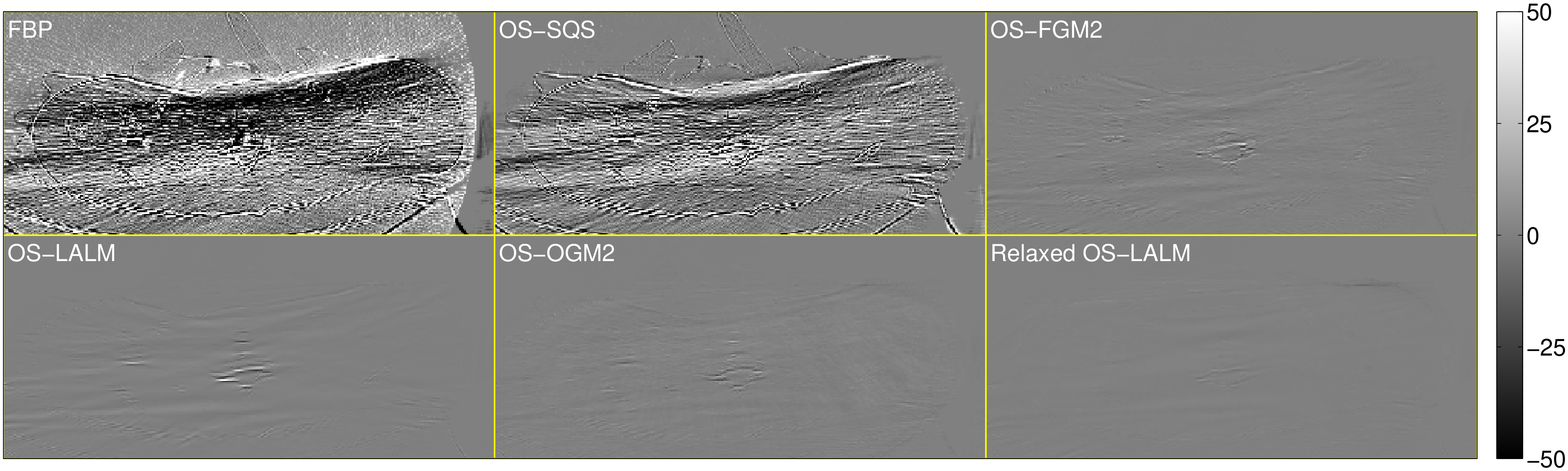}
    \caption{Shoulder: Cropped difference images (displayed from ${-50}$ to $50$ HU) from the central transaxial plane of the initial FBP image $\iter{\mb{x}}{0}-\mb{x}^{\star}$ and the reconstructed image $\iter{\mb{x}}{20}-\mb{x}^{\star}$ using OS algorithms with $20$ subsets after $20$ iterations.}
    \label{fig:nien-16-rla-supp:ct01_diff_20_iter_20_blocks}
\end{figure}

\clearpage
\bibliographystyle{ieeetr}
\bibliography{nien-16-rla-supp.bbl}